\documentclass[11pt,leqno]{article}
\usepackage{etex}
\usepackage[margin={1.2in}]{geometry}
\usepackage[titletoc]{appendix}
\usepackage[utf8]{inputenc}
\usepackage[T1]{fontenc}
\usepackage{amsmath,amsthm ,amsfonts,amssymb,amscd,mathtools}

\usepackage[english]{babel}
\usepackage{blindtext}
\usepackage{csquotes}
\usepackage[backend=biber,style=numeric, giveninits=true]{biblatex}
\DeclareSourcemap{
  \maps[datatype=bibtex]{
    \map{
      \step[fieldsource=journal, match=\regexp{Journal\s+of\s+the\s+American\s+Mathematical\s+Society}, replace={J. Amer. Math. Soc.}]
      \step[fieldsource=journal, match=\regexp{Communications\s+on\s+Pure\s+and\s+Applied\s+Mathematics}, replace={Comm. Pure Appl. Math.}]
    }
  }
}
\DeclareFieldFormat{pages}{#1}
\DeclareFieldFormat[article]{volume}{\mkbibbold{#1}} % Format volume
\DeclareFieldFormat[article]{year}{\mkbibparens{#1}}
\DeclareFieldFormat[article]{title}{#1}
\DeclareFieldFormat[incollection]{title}{#1}
\DeclareFieldFormat[book]{title}{#1}
\DeclareFieldFormat[article]{journaltitle}{\textrm{#1}}
\DeclareFieldFormat[book]{journaltitle}{\textrm{#1}}
\DeclareBibliographyDriver{article}{%
  \usebibmacro{bibindex}%
  \usebibmacro{begentry}%
\usebibmacro{author/editor}%
\setunit{\addcomma\space}
  \usebibmacro{title}%
  \setunit{\addcomma\space}
  \printfield{journaltitle}
  \setunit{\addcomma\space}
  \printfield{volume}%
  \printfield{year}%
  \setunit{\addcomma\space}
  \printfield{pages}%
  \usebibmacro{finentry}%
}

\usepackage{tikz,tikz-cd}
\usetikzlibrary{positioning}
\usetikzlibrary{matrix}
\usetikzlibrary{shapes, arrows}
\usetikzlibrary{decorations.markings,patterns,bending}
\usepackage{eucal}
\usepackage{mathrsfs}
\usepackage{stmaryrd}
\usepackage{mathtools}
\usepackage{enumerate}
\usepackage{float}
\usepackage{subcaption}

\DeclarePairedDelimiter\abs{\lvert}{\rvert}%

\usepackage{rotating}

\usepackage{longtable}

\usepackage[sc]{mathpazo}
\usepackage{courier}

\usepackage{microtype}

\usepackage{amsmath}
\usepackage{amssymb}
\usepackage{amsthm}
\usepackage{color}
\usepackage{latexsym,extarrows}

\usepackage[all]{xy}
\usepackage[stable,multiple]{footmisc}

\RequirePackage[font=small,format=plain,labelfont=bf,textfont=it]{caption}
\makeatletter
\newsavebox{\@brx}
\newcommand{\llangle}[1][]{\savebox{\@brx}{\(\m@th{#1\langle}\)}%
  \mathopen{\copy\@brx\kern-0.5\wd\@brx\usebox{\@brx}}}
\newcommand{\rrangle}[1][]{\savebox{\@brx}{\(\m@th{#1\rangle}\)}%
  \mathclose{\copy\@brx\kern-0.5\wd\@brx\usebox{\@brx}}}
\makeatother

\usepackage{hyperref}
\begin{document}
%%%%%%%%%%%%%%%%%%%%%%%%%%%%%%%%%%%%%%%%%%%%%%%%%%%%%%%%%%%%%%%%%%%%%%%%
%%%%%%%%%%%%%%%%%%%%%%%%%%     Macros      %%%%%%%%%%%%%%%%%%%%%%%%%%%%%
%%%%%%%%%%%%%%%%%%%%%%%%%%%%%%%%%%%%%%%%%%%%%%%%%%%%%%%%%%%%%%%%%%%%%%%%
\def\e#1\e{\begin{equation}#1\end{equation}}
\def\ea#1\ea{\begin{align}#1\end{align}}
\def\eq#1{{\rm(\ref{#1})}}
\theoremstyle{plain}% default
\newtheorem{thm}{Theorem}[section]
\newtheorem{lem}[thm]{Lemma}
\newtheorem{prop}[thm]{Proposition}
\newtheorem{cor}[thm]{Corollary}
\theoremstyle{definition}
\newtheorem{dfn}[thm]{Definition}
\newtheorem{ex}[thm]{Example}
\newtheorem{rem}[thm]{Remark}
\newtheorem{conjecture}[thm]{Conjecture}
\newtheorem{convention}[thm]{Convention}
\newtheorem{assumption}[thm]{Assumption}
\newtheorem{notation}[thm]{Notation}

%%%%%%%%%%%%%%%%%%%%%%%%%%%%
\newcommand{\op}{\operatorname}
\newcommand{\C}{\mathbb{C}}
\newcommand{\N}{\mathbb{N}}
\newcommand{\R}{\mathbb{R}}
\newcommand{\Q}{\mathbb{Q}}
\newcommand{\Z}{\mathbb{Z}}

\newcommand{\bl}{\textbf}
\newcommand{\mbf}{\mathbf}
\newcommand{\mbb}{\mathbb}
\newcommand{\mf}{\mathfrak}
\newcommand{\mc}{\mathcal}
\newcommand{\iso}{\cong}

%\makeatletter
%\RenewDocumentCommand{\title}{om}{%
%   \IfNoValueTF{#1}
%     {\gdef\shorttitle{My default running title}}
 %    {\gdef\shorttitle{#1}}%
%   \gdef\@title{#2}%
%}
%\makeatother

%\makeatletter
%\let\orig@afterheading\@afterheading
%\def\@afterheading{%
%   \@afterindenttrue
%  \orig@afterheading}
%\makeatother

\makeatletter
\newcommand*\bigcdot{\mathpalette\bigcdot@{.5}}
\newcommand*\bigcdot@[2]{\mathbin{\vcenter{\hbox{\scalebox{#2}{$\m@th#1\bullet$}}}}}
\makeatother

%%%%%%%%%%%%%%%%%%%%%%%%%%%%%%%%%%%%%%%%%%%%%%%%%%%%%%%%%%%%%%%%%%%%%%%%
%%%%%%%%%%%%%%%%%%%%%%%    Text of paper    %%%%%%%%%%%%%%%%%%%%%%%%%%%%
%%%%%%%%%%%%%%%%%%%%%%%%%%%%%%%%%%%%%%%%%%%%%%%%%%%%%%%%%%%%%%%%%%%%%%%%
%The following title can be changed according to needs.%
%\usepackage{hyperref}

\title{\bf Cluster algebra and quantum cohomology ring: the A-type case}
\author{Weiqiang He\footnote{Correspondence author is Weiqiang He: hewq@mail2.sysu.edu.cn}, Yingchun Zhang} 
\date{}
\maketitle
 %\keywords{}
\begin{abstract}
This work constructs a cluster algebra structure within the quantum cohomology ring of a quiver variety associated with an $A$-type quiver. 
Specifically, let $Fl:=Fl(N_1,\ldots,N_{n+1})$ denote a partial flag variety of length $n$, and $QH_S^*(Fl)[t]:=QH_S^*(Fl)\otimes \mathbb C[t]$ be its equivariant quantum cohomology ring extended by a formal variable $t$, regarded as a $\mathbb Q$-algebra. 
We establish an injective $\mathbb Q$-algebra homomorphism from the $A_n$-type cluster algebra to the algebra $QH_S^*(Fl)[t]$.
Furthermore, for a general quiver with potential, we propose a framework for constructing a homomorphism from the associated cluster algebra to the quantum cohomology ring of the corresponding quiver variety. 

The second main result addresses the conjecture of all-genus Seiberg duality for $A_n$-type quivers. 
For any quiver with potential mutation-equivalent to an $A_n$-type quiver, we consider the associated variety defined as the critical locus of the potential function.
We prove that all-genus Gromov-Witten invariants of such a variety coincide with those of the flag variety. 
 \end{abstract}

 %{\hypersetup{linkcolor=black}
 \setcounter{tocdepth}{2} \tableofcontents
 %}
\section{Introduction}
\subsection{Overview}
The theory of cluster algebras, developed in \cite{clusteralg:1, clusteralg:2,clusteralg:3,clusterpot1,clusterpot:2}, has revealed deep connections with enumerative geometry through extensive evidence. 
%For instance, Marsh and Rietsch constructed a Landau-Ginzburg B-model of Grassmannian $Gr(r,N)$ where the coordinate ring of the affine Calabi-Yau variety is a cluster algebra \cite{MR4072789}. 
Kontsevich and Soibelman established the link between Donaldson-Thomas theory and the cluster transformation in \cite{kontsevich2008stability},
while Nagao demonstrated the connection between Donaldson-Thomas theory and cluster algebras \cite{nagao2013donaldson}, proving five fundamental properties of cluster algebras using this relation.

Our work establishes a natural and direct connection between the cluster algebra and the Gromov-Witten theory. 
We consider a decorated quiver with potential $(\mathbf Q=(Q_f\subset Q_0,Q_1,W),\mathbf r)$, as defined in Definition \ref{dfn:quiver}, alongside a stability condition $\theta\in \chi(G)$. 
Let $Z:=\{dW=0\}$ denote the critical locus of $W$, and let $\mc Z:=Z \sslash_\theta G$ be the corresponding GIT quotient. 
The quantum cohomology ring of $\mc Z$, viewed as a $\mathbb Q$-algebra, is a central object of our study. 
Parallelly, a geometric cluster algebra $\mathscr A_{\mathbf Q}$ can be constructed for the given quiver, see Definitions \ref{dfn:seed}-\ref{dfn:cluster}.
We propose a construction of an injective $\mathbb Q$-algebra homomorphism from the cluster algebra $\mathscr A_{\mathbf Q}$ to the quantum cohomology ring of $\mc Z$ in Conjecture \ref{conj:clusteralgconj}.
This homomorphism maps the initial cluster variables to the Chern polynomials of  tautological bundles  over the variety $\mc Z$ and noninitial cluster variables to Chern polynomials of various quotient bundles. 
We explicitly verify this construction for $A$-type quivers, see Theorem \ref{thm:maincluster2}.

 This construction is motivated by the study of the Seiberg duality conjecture and the insights of \cite{benini2015cluster}. 
 Seiberg duality explores the relationship between the enumerative theories of two quivers with potentials related by a quiver mutation.  Given a decorated quiver with potential $(\mathbf Q=(Q_f\subset Q_0,Q_1,W),\mathbf r)$, 
 we perform a quiver mutation at a gauge node,  yielding another decorated quiver with potential $\mathbf Q'=((Q_f\subset Q_0, Q_1', W'),\mathbf r')$. 
 Let $\theta'$ denote the character of the gauge group $G'$ which is related with $\theta$ via Equation \eqref{eqn:changeofphasemu}, and  define the variety $\mc Z':=\{dW'=0\}\sslash_{\theta'}G'$ as the critical locus of the mutated potential function. 
The Seiberg duality conjecture asserts that Gromov-Witten theories of $\mc Z$ and $\mc Z'$ are equivalent, 
see \cite{Hori,HoriTong,benini2015cluster,gomis2016m2} for physics achievements, and \cite{nonabelianGLSM:YR} for the mathematical formulation.  
The authors' previous works, including \cite{zhang2021gromov, hezh2023seiberg} have proved the Seiberg duality conjecture for $A$-types, star-shaped, and $D$-type quivers. Similarly, Priddis-Wen-Shoemaker have verified the Seiberg duality conjecture for PXY and PAXY model in \cite{priddis2024seiberg}.
These works focus primarily on genus zero Seiberg duality conjecture and prove transformation formulas of quasimap small $I$-functions of $\mc Z$ and $\mc Z'$.
In this work, we extended the scope to higher-genus Gromov-Witten theory, proving all-genus Seiberg duality conjecture of quivers that are mutation-equivalent to an $A_n$-type quiver, see Theorem \ref{thm:GWinvarmu}. 

\subsection{Gromov-Witten theory and Seiberg Duality conjecture}
Consider a quiver with potential (QP) $\mathbf Q=(Q_f\subset Q_0,Q_1,W)$, as defined in Definition \ref{dfn:quiver}. 
Decorate the quiver with an integer vector $\mathbf r=(r_i)_{i\in Q_0}\in \mathbb Z_{>0}$ 
 and choose a character $\theta\in \chi (G)$.  This setup allows us to define the variety $\mc Z:=\{dW=0\}\sslash_{\theta}G$ which is the critical locus of the potential function. 
When $\mc Z$ is smooth, we can study its Gromov-Witten theory. For the definition of Gromov-Witten theory, see \cite{MR1366548, MR1483992, GW:LT}, with a comprehensive introduction in \cite{GW:mirrorsym}. 
When the potential is zero, the variety $\mc Z$ is simply the quiver variety.
We refer $\mc Z$ to the variety associated with the decorated QP  $(\mathbf Q,\mathbf r)$. 

Applying a quiver mutation to the decorated QP as in Definition \ref{dfn:mutation} and Equation \eqref{eqn:Nchange}, results in a new decorated QP $(\mathbf Q'=(Q_f\subset Q_0,Q_1',W'),\mathbf r')$, see Definition \ref{dfn:mutation} and Equation \eqref{eqn:Nchange}. 
Choosing a character $\theta'$ of the new gauge group $G'$ related to the character $\theta$ of $G$ via Equation \eqref{eqn:changeofphasemu}, we define the critical locus $\mc Z'$ of $W'$. 
 The Seiberg duality conjecture posits that the Gromov-Witten invariants of $\mc Z$ and $\mc Z'$ are equivalent, see Conjecture \ref{conj:seibergdualityconj}.
 In particular, let $QH^*_S(\mc Z):=H^*_S(\mc Z)\otimes \mathbb Q[[q]]$ be the equivariant quantum cohomology ring of $\mc Z$ where $q=(q_i)_{i\in Q_0\backslash Q_f}$ are K\"ahler variables. 
 A genus zero version of the Seiberg duality conjecture asserts that $QH^*_S(\mc Z)$ and $QH^*_S(\mc Z')$ are isomorphic, see Conjecture \ref{conj:seiberginqcoh}. 
 
Consider the $A_n$-type quiver in Figure \ref{fig:Anqv}, decorated by integers $N_1<\ldots<N_{n+1}$. 
 Let $\Omega_n$ denote the set of all QPs mutation-equivalent to the decorated $A_n$-type quiver, 
 see Lemma \ref{lem:omegan} and \ref{lem:recursivepattern} for the descriptions of QPs in this set and Lemma \ref{lem:decorationsofQinP} for the decoration of each QP in $\Omega_n$.

 Let $\mathbf Q$ and $\mathbf Q'$ be two QPs in $\Omega_n$ that are related by a quiver mutation $\mu_v$ at a gauge node $v$ and $\mc Z$ and $\mc Z'$ be their critical loci. 
 The torus $S=(\mathbb C^*)^{N_{n+1}}$ acts on both $\mc Z$ and $\mc Z'$ such that the torus fixed loci $\mc Z^S$ and $(\mc Z')^S$ are finite and of equal cardinality.
 Moreover, Lemma  \ref{lem:Sfixedpointsmu} provides a canonical bijection between the two torus fixed loci
    \[\varphi: \mc Z^S\rightarrow (\mc Z')^S.\] 
This induces a canonical isomorphism between the equivariant cohomology rings of $\mc Z$ and $\mc Z'$, see Proposition \ref{prop:equicohmu}. 

 Using the localization theorem \cite{GW:GPequiv,liu2011localization}, we analyze the one-dimensional graphs in $\mc Z$ and $\mc Z'$ and their contributions to Gromov-Witten invariants. 
 This leads to the proof of the equivalence of all-genus Gromov-Witten invariants of $\mc Z$ and $\mc Z'$. 
 \begin{thm}[Theorem \ref{thm:GWinvarmu}]\label{introdthm:seibergdualityAn}
      Let $\gamma_1,\ldots,\gamma_m\in H^*_S(\mc Z)$ and $\gamma_1',\ldots,\gamma_m'\in H^*_S(\mc Z')$ such that for each $P\in \mc Z^S$
\begin{equation*}
    \gamma_i\big{|}_{P}= \gamma_i'\big{|}_{\varphi(P)}\,.
\end{equation*}
Then for any $a_i\in \mathbb Z_{\geq 0}$ and effective curve class $\beta\in H_2(\mc Z)$ and $\beta'\in H_2(\mc Z')$ that are related as in Proposition \ref{prop:transformeff},
we have 
\begin{equation*}
    \langle \tau_{a_1}\gamma_1,\ldots,\tau_{a_m}\gamma_m \rangle_{g,m,\beta}^{\mc Z,S}=
    \langle \tau_{a_1}\gamma_1',\ldots,\tau_{a_m}\gamma_m' \rangle_{g,m,\beta'}^{\mc Z',S}.
\end{equation*}
 \end{thm}
As a corollary of the Theorem \ref{introdthm:seibergdualityAn}, the quantum cohomology rings $QH^*_S(\mc Z)$ and $QH^*_S(\mc Z')$  are isomorphic up to a change of K\"ahler variables, see Corollary \ref{cor:qcohisomorphic}.
 
\subsection{Cluster algebra conjecture}
Let $(\mathbf Q, \mathbf x=(x_v)_{v\in Q_0})$ be a seed which consists of a QP $\mathbf Q=(Q_0,Q_1,W)$ and a set of algebraically independent variables $\mathbf x=(x_v)_{v\in Q_0}$. There is an associated cluster algebra $\mathscr A_{\mathbf Q}$ as described in Definition \ref{sec:clusteralgebras}, which is a sub-algebra in the rational function field $\mathbb Q(x_1,\ldots,x_{n})$, where $n$ is the number of nodes in $Q_0$. 
Let $(\mathbf Q, \mathbf x)$ and $(\mathbf Q',\mathbf x'=(x_1',\cdots, x_n'))$ be two seeds that are related by a seed mutation at node $v\in Q_0$, 
then the two sets $\mathbf x$ and $\mathbf x'$ coincide except for the variables $x_v$ and $x_v'$, 
which satisfy the cluster exchange relation: 
\begin{equation*}
    x_vx_v'=\prod_{u\in Q_0}x_u^{[b_{uv}]_+}+\prod_{w\in Q_0}x_w^{[b_{vw}]_+}\,.
\end{equation*}

For the above mentioned QPs $\mathbf Q$ and $\mathbf Q'$, we decorate them by $\mathbf r$ and $\mathbf r'$ such that $r_u$ and $r_u'$ coincide and $r_v$ and $r_v'$ are related by \eqref{eqn:Nchange}. 
Let $\mc Z$ and $\mc Z'$ be their critical loci which admit a common torus action $S=\prod_{u\in Q_f}(\mathbb C^*)^{r_u}$ from frozen nodes.
We consider the $\mathbb Q$-algebra $QH^*_S(\mc Z)[t]:=QH^*_S(\mc Z)\otimes \mathbb Q[t]$ which we refer to as the quantum cohomology ring of $\mc Z$ by abuse of terminology.
The equivariant Chern polynomials of tautological bundles and quotient bundles $c_t^S(S_k)$ and $c_t^S(Q_k)$ lie in $QH^*_S(\mc Z)[t]$ for each node, as detailed in Section \ref{sec:seibergduality}. 
We conjecture that the quantum product of $c_t^S(S_k)$ and $c_t^S(Q_k)$ satisfies the quantum cohomological cluster exchange relation, 
see Conjecture \ref{conj:clusterexchagenrelation} which is an analogue of the cluster exchange relation. 
Therefore, we expect the existence of a $\mathbb Q$-algebra homomorphism from $\mathscr A_{\mathbf Q}$ to $QH_S^*(\mc Z)[t]$, which sends the initial cluster variables $\mathbf x$ to the Chern polynomials of tautological bundles.
The challenge is how to define the image of the non-initial cluster variables. 
We propose an explicit construction in the following Conjecture.
\begin{conjecture}[Conjecture \ref{conj:clusteralgconj}]\label{introdcoj:clusteralgconj}
Introduce a new set of formal variables $\{\xi_i\}_{i\in Q_0}$ and let $q_k=(-1)^{N_k+N_f}\prod_{j\in Q_0}\xi_j^{b_{jk}}$.
There exists a well-defined $\mathbb Q$-algebra homomorphism 
    \begin{equation*}
        \psi:\mathscr A_{\mathbf Q}\rightarrow QH^*_S(\mc Z)[t]
    \end{equation*}
    such that 
    \begin{enumerate}
        \item $\psi(x_i)=(-1)^{N_k}\xi_kc_t^S(S_k)$
        \item  
        \begin{equation*}
            \psi(x_i')=\begin{cases}
            (-1)^{N_f-N_k}\xi_k^{-1}\prod_{k\rightarrow j}\xi_jc_t^S(Q_k) \text{ when } N_f(k)>N_a(k)\\
            (-1)^{N_a-N_k}\xi_k^{-1}\prod_{k\rightarrow j}\xi_jc_t^S(Q_k) \text{ when } N_f(k)>N_a(k)
            \end{cases}
        \end{equation*}
        \item For a non-initial cluster variable $\tilde x_i$, assume that $\tilde x_i$ lies in a seed $(\tilde {\mathbf Q}, {\tilde {\mathbf x}})$ which can be obtained by a sequence of quiver mutations $\mathbf Q\xhookrightarrow{\mu_{k_1}}\mathbf Q^{1} \xrightarrow{\mu_{k_2}} \ldots\xrightarrow{\mu_{i}} \tilde {\mathbf Q}$. 
        The conjecture \ref{conj:seiberginqcoh} asserts that the quantum cohomology rings of the varieties of these QPs are isomorphic. 
        Let $\tilde\Phi: QH^*_S(\mc Z) \rightarrow QH^*_S(\tilde{\mc Z})$ be the composite of those isomorphisms.
        Let $\tilde S_i$ be the tautological bundle of the variety $\tilde{\mc Z}$. 
        Then the image of $\tilde x_i$ is a scaling of pullback of $c_t^S(\tilde{S}_i^\vee)$ by $\pm 1$ and $\xi_j$, and those scalars can be determined recursively in order to make $\psi$ a ring homomorphism. 
    \end{enumerate}
\end{conjecture}
\subsection{Cluster algebra conjecture for A-type quiver}
We begin by considering the $A$-type quiver as described in Example \ref{ex:An}. 
Denote the corresponding cluster algebra by $\mathscr A_n\subset \mathbb Q(x_1,\ldots,x_{n+1})$. 
The associated quiver variety is a partial flag variety $Fl:=Fl(N_1,\ldots,N_{n+1})$. 
The Conjecture  \ref{introdcoj:clusteralgconj} holds for $A$-type quivers.
\begin{thm}[Theorem \ref{thm:maincluster2}]\label{intrthm:clustermap}
    Define a map 
    \begin{equation*}
        \psi:\mathscr A_{n}\rightarrow QH^*_S(Fl)[t]
    \end{equation*} 
    as follows.
    \begin{enumerate}
        \item The map $\psi$ sends each initial cluster variable $x_i$ to the equivariant Chern polynomial of tautological bundle: $\psi(x_i)=(-1)^{N_i}\xi_ic_t^S(S_i)$ for $i=1,\ldots, n+1$.
        \item For each non-initial cluster variable $x_v'$, let $(\mathbf Q'=(Q_0,Q_1,W),\mathbf x')$ be one non-initial seed that contains the cluster variable $x_v'$,
        and then the decorated integers  $(r_v)_{v\in Q_0}$ assigned to the QP is determined. 
        By Lemma \ref{lem:decorationsofQinP}, each integer $r_v$ is $N_j-N_i$ for some $1\leq i<j\leq n+1$.
        The map $\psi$ sends the cluster variable $x_v'$ to $(-1)^{N_j-N_i}\xi_j\xi_i^{-1}c_t^S(S_j\slash S_i)$. 
    \end{enumerate}
    The map $\psi$ is a well-defined injective $\mathbb Q$-algebra homomorphism if we let $q_i=(-1)^{N_{i+1}+N_{i}}\xi_{i+1}^{-1}\xi_{i-1}$.
\end{thm}
Quiver mutation is a local operation, which means it only affects the nodes and arrows adjacent to the mutated node. 
The most general local behavior of the quiver mutation is illustrated in Figure \ref{fig:localbehavior}. 
We let the integers assigned to those nodes be $r_v=N_k-N_l,r_{v_1}=N_m-N_l, N_{v_3}=N_p-N_{l}$ with $m>k>p\geq l$, and then $r_{v_2}=N_m-N_{k+1}, r_{v_4}=r_{k}-r_{p+1}$, see Figure \ref{fig:localbehavior} (a).

To prove Theorem \ref{intrthm:clustermap}, we need to verify that the map $\psi$ preserves the cluster exchange relations.
Based on the definition of the map $\psi$ and the assumptions outlined in the above paragraph, we only need to prove the following relations in the quantum cohomology ring of the partial flag variety $Fl$.
\begin{thm}\label{introdthm:quantumcohexrel}
\begin{enumerate}
    \item In the quantum cohomology ring of a flag variety $QH^*_S(Fl)[t]$, for $m> k> p\geq l\geq 0$, we have the following 
    \textit{quantum cohomological cluster exchange relations}. \begin{align}\label{introdeqn:quantumcohexrel0}
        c_t^S(S_m/S_{p+1})*c_t^S(S_k/S_l)&=c_t^S(S_m/S_l)*c_t^S(S_k/S_{p+1}) \nonumber\\
        &+\prod_{a=p+1}^{k}(-1)^{N_a+N_{a-1}}q_ac_t^S(S_m/S_{k+1})*c_t^S(S_p/S_l)\,.
    \end{align}
\item The map $\psi$ in Theorem \ref{thm:maincluster2} is an injective ring homomorphism if the quantum cohomological cluster exchange relations hold.
\end{enumerate}
\end{thm}
One can find that the relations \eqref{introdeqn:quantumcohexrel0} Chern polynomials $c_t^S(S_j/S_i)$ satisfy are very close to the cluster exchange relation except for the factor $\prod_{a=p+1}^{k}(-1)^{N_a+N_{a-1}}q_a$. 
The parameters $(\xi_i)_{i\in Q_0}$ are introduced in order to cancel such factors.

Although we have assumed that the node $v$ has four adjacent nodes, it can be reduced to cases where the node $v$ has fewer adjacent nodes by taking special values for integers $m,k,p,l$, see Remark \ref{rem:integersmkpl} (1) and (2). 

In Conjecture \ref{introdcoj:clusteralgconj}, the images of cluster variables are pullbacks of the Chern polynomials of the dual tautological bundles over varieties of QPs that are mutation-equivalent to an initial QP $\mathbf Q$. However, in Theorem \ref{introdthm:quantumcohexrel}, the cluster variables are mapped to Chern polynomials of quotient bundles $\{c_t^S(S_j/S_i)\}_{j>i}$. 
Actually, Chern polynomials of these quotient bundles are pullbacks of Chern polynomials of the dual of tautological bundles over varieties of QPs in $\Omega_n$. See Corollary \ref{cor:Aclustertrue}.

Fontanine has established a ring representation in \cite{qcohflag:CF99} for the quantum cohomology ring of a partial flag variety. 
Our work provides explicit quantum products for a broader range of cohomology classes. 

In the proof of Theorem \ref{introdthm:quantumcohexrel}, we have utilized  results of Fontanine-Kim-Sabbah  \cite{abelian/nonabelian:CKB} on the abelian/nonabelian correspondence for Frobenius manifolds, as well as Webb's work on abelian/nonabelian correspondence for quasimap $I$-functions in \cite{abeliannonabelian:Webb, abelianizationlef:Webb}.

\subsection{Future directions}
There are three directions to develop the relation of cluster algebra and enumerative geometry. 
First, we will prove the cluster algebra conjecture for more quivers. 
In our upcoming work, we will prove that the cluster algebra conjecture \ref{introdcoj:clusteralgconj} is true for $D$-type quivers. 
Secondly, we propose that there is also a cluster algebra structure in the quantum K ring, which is in progress.  
Thirdly, it is natural to consider the quantum difference D-module, and we expect that it admits a quantum cluster algebra structure.
\subsection{Outline of the paper}
We will make the work self-contained. In Section \ref{sec:quiversandclusteralg}, we will introduce the basic definitions of cluster algebras and quiver varieties. In particular, in Section \ref{subsec:Anmutationequivalentqivers} we will describe all quivers that are mutation-equivalent to the $A_n$-type quiver and in Section \ref{subsec:quivervarmuteqtoAn} we will describe the corresponding varieties. 

In Section \ref{sec:GWandqcoh}, we will review the Gromov-Witten theory of quiver varieties. In Section \ref{sec:seibergduality} we will review the all genus Seiberg duality conjecture in invariants level and in quantum cohomology level. 
In Section \ref{sec:clusteralgebraconj}, we will propose the cluster algebra conjecture and we will explicitly give the cluster algebra structure in the quantum cohomology ring for a general quiver. 

In Section \ref{sec:proofofseibergduality}, we will prove the all-genus Seiberg duality conjecture for any QP in $\Omega_n$ that is mutation-equivalent to $A_n$ by utilizing the localization theorem. 
We will describe the torus fixed points, one-dimensional orbits, and contributions of decorated graphs. 
Then, we will compare the contribution of each decorated graph to the Gromov-Witten invariants of two quiver varieties in $\Omega_n$ that are related by a quiver mutation. 

In Section \ref{sec:proofclusteralgconj}, we will prove the cluster algebra conjecture for $A_n$-type quivers by using abelian/nonabelian correspondence based on \cite{abelian/nonabelian:CKB, abeliannonabelian:Webb, abelianizationlef:Webb}. 

\subsection*{Acknowledgement}
The first author thanks Changzheng Li, Jiayu Song, Yaoxiong Wen, and Mingzhi Yang for many helpful discussions on quantum cohomology rings of flag varieties.
Part of the progress was made when the first author was visiting BIMSA, YMSC, and IASM; he is grateful for their hospitality. 
The second author expresses sincere gratitude to Peng Zhao for suggesting a possible connection between Gromov-Witten theory and cluster algebras. She also thanks Professor Yongbin Ruan for his valuable discussions, encouragement to complete this work, and for providing a supportive and enjoyable working environment. 
She is also thankful to Wei Gu, Yaoxiong Wen, Jie Zhou, and Zijun Zhou for discussions on nonabelian mirror symmetry and quantum cohomology of nonabelian GIT quotients. 
Special thanks go to Professor Fang Li for helpful discussions on cluster algebras, and to Professor Sergey Fomin, Jiarui Fei, Fan Qin, and Cheng Shu for their interest in this work and for suggesting related problems.
Most of the work was carried out while the second author was at Zhejiang University, and she is deeply grateful for the support provided by the university.
The first author is supported by the National Key R\&D Program of China
(Project No. 2023YFA1009802) and by NSFC. 12422104. The second author is supported by NSFC.12301080.

\section{An introduction to quivers and cluster algebras}\label{sec:quiversandclusteralg}
\subsection{Quivers with potentials}
\begin{dfn}[\cite{quiver}]\label{dfn:quiver}
A  \textit {quiver with potential} (QP) is a finite oriented graph with a function $\mathbf Q=(Q_f\subset Q_0, Q_1, W)$ where
\begin{enumerate}
    \item $Q_0$ is the set of vertices among which $Q_f$ is the set of framing (frozen) nodes, denoted by $\framebox(3,3){}$ in the graph, and $Q_0\backslash Q_f$ is the set of gauge nodes, denoted by $\bigcirc $;
    \item $Q_1$ is the set of arrows; an arrow from nodes $u$ to $v$ is denoted by $u\xrightarrow{a} v\in Q_1$,  and the number of such arrows is denoted by $b_{uv}$; There are two maps $s,t:Q_1\rightarrow Q_0$, where $s$ sends an arrows $u\xrightarrow {a} $ to its source $u$ and $t$ sends an arrow to its target $v$;
    \item $W$ is a potential function, defined as a function on cycles in the quiver diagram. 
\end{enumerate}
\end{dfn}
A loop of a quiver is an arrow $a$ such that $s(a)=t(a)$, and a 2-cycle is a pair of opposite arrows. 
We make an additional condition that there are no loops and 2-cycles, and such quivers are called cluster quivers. 
For an integer $b\in \mathbb Z$, denote $[b]_+=\max(b,0)$.
\begin{dfn}\label{dfn:mutation}
    Fix a QP $\mathbf Q=(Q_f\subset Q_0, Q_1, W)$. A quiver mutation at a gauge node $v$, denoted by $\mu_v$, is defined by the following steps:
    \begin{enumerate}
    \item For each path $u\xrightarrow{a} v\xrightarrow {b}w$ passing through the node $v$, add a new arrow $u\xrightarrow{c} w$. 
    Invert all arrows starting from or ending at $v$, and denote the new arrows by $w\xrightarrow{b^*} v, v\xrightarrow{a^*} u$.
    \item  Delete all 2-cycles arising from the first step. 
    \item Replace the factor $ab$ in the potential function $W$ by the new one $c$ and add all cubic terms $cb^*a^*$ to the potential. 
    Denote the resulting potential function by $\tilde W$, and denote the resulting QP by $\tilde {\mathbf Q}=(Q_f\subset Q_0,\tilde Q_1,\tilde W)$. 
    \item Let ${\mathbf Q'}=(Q_f\subset Q_0, Q_1',  W')$ be the reduced QP of $\tilde{\mathbf Q}$, see  \cite[Theorem 4.6] {clusterpot1}. 
    This is the new QP we get by performing the quiver mutation $\mu_v$ to the original $\mathbf Q$.
    \end{enumerate}
\end{dfn}
We call two QPs mutation-equivalent if we can obtain one by performing a sequence of quiver mutations to the other one. 
For all examples in this work that are mutation equivalent to the $A_n$-type quiver, $\mathbf Q'=\tilde{\mathbf Q}$. Note that the set of nodes is preserved by quiver mutations up to order.
\begin{ex}
 Figure \ref{fig:A2mut} shows a quiver mutation at the node $2$. 
\begin{figure}[ht]
\centering
\begin{tikzpicture}
	\node[draw,
	circle,
	minimum size=1cm,
	] (node1) at (-6,0){};
    \node[draw,
	circle,
	minimum size=1cm,
	] (node2) at (-4,0){};
	\node[draw,
	minimum width=1cm,
	minimum height=1cm,
	] (node3) at (-2,0){};
	\draw[-stealth, dashed] (node1) to [bend left] (node3);
    %\draw[-stealth, dashed] (node3) to [bend right] (node1);
    \begin{scope}[transform canvas={yshift=.2em}]
  \draw [-stealth, dashed] (node3) to [bend right] (node1);
\end{scope}
	\draw[-stealth] (node1) -- (node2);
	\draw[-stealth] (node2) -- (node3);
    \node at (0,0.3){$\mu_2$};
    \node at (-6,-0.7){$1$};
    \node at (-4,-0.7){$2$};
    \node at (-2,-0.7){$3$};
    %right hand side
    \node[draw,
	circle,
	minimum size=1cm,
	] (node4) at (2,0){};
    \node[draw,
	circle,
	minimum size=1cm,
	] (node5) at (4,0){};
	\node[draw,
	minimum width=1cm,
	minimum height=1cm,
	] (node6) at (6,0){};
	\draw[-stealth] (node6)--(node5);
    \draw[-stealth] (node4) to [bend left] (node6);
	\draw[-stealth] (node5) -- (node4);
    \draw[implies-implies,double equal sign distance](-0.5,0)--(0.5,0);
    \node at (6,-0.7){$3$};
    \node at (4,-0.7){$2$};
    \node at (2,-0.7){$1$};
    \node at (4,0.9){$c$};
    \node at (3,-0.2){$a$};
    \node at (5,-0.2){$b$};
\end{tikzpicture}
\caption{The two quiver with potentials are related by a quiver mutation at the center node. The dashed arrows in the left QP are canceled. The right one has a potential function $W=bac$.} 
\label{fig:A2mut}
\end{figure}
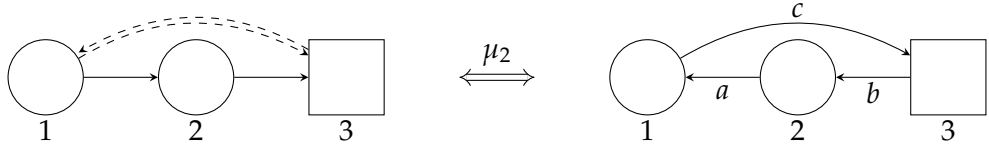
\iffalse
\begin{figure}[H]
    \centering
    \includegraphics [width=8cm]{a3.png}
    \caption{}
    \label{fig:A2mut}
\end{figure}
\fi
\end{ex}
\subsection{Cluster algebras}\label{sec:clusteralgebras}
We introduce the definition of  cluster algebras of geometric type. We refer readers to the original works for the theory of cluster algebras \cite{clusteralg:1, clusteralg:2,clusteralg:3,clusteralg:4,clusterpot1,clusterpot:2} and the introductory book \cite{fomin:cluster1-3,fomin:cluster4-5,fomin:cluster6,fomin:cluster7}. 
\begin{dfn}\label{dfn:seed}
    A \textit{seed} is a pair $(\mathbf Q, \mathbf x)$ where 
    \begin{enumerate}
        \item $\mathbf Q=(Q_f\subset Q_0,Q_1,W)$ is a QP whose underlying quiver admits $n:=\abs{Q_0}$ nodes,
        \item $\mathbf x=(x_1,\ldots,x_n)$ is a set of n algebraically independently variables and let $\mathscr F=\mathbb Q(x_1,\ldots,x_n)$.
    \end{enumerate}
The set $\mathbf x$ is called a \textit{cluster} and variables $x_1,\ldots,x_n$ are called \textit{cluster variables} of the seed $(\mathbf Q, \mathbf x)$.
\end{dfn}

\begin{dfn}\label{dfn:seedmutation}
   A seed mutation at a gauge node $v\in Q_0\backslash Q_f$ to $(\mathbf Q,\mathbf x)$, which is also denoted by $\mu_v$, is defined as follows.
    \begin{enumerate}
        \item Perform the quiver mutation $\mu_v$ to the quiver $\mathbf Q$ and get a new quiver $\mathbf Q'$. 
        \item The mutation $\mu_v$ maps the cluster $\mathbf x$ to $\mathbf x'=(x_1',\cdots,x_n')$ such that $x_u'=x_u$ for $u\neq v$, and 
        $x_v'$ is determined by the \textit{cluster exchange relation}
        \begin{equation}\label{eqn:clusterrelation}
          x_vx_v'=\prod_{u\in Q_0}x_u^{[b_{uv}]_+}+\prod_{w\in Q_0}x_w^{[b_{vw}]_+}\,.
        \end{equation}
    \end{enumerate}
\end{dfn}

We call two seeds $(\mathbf Q, \mathbf x)$ and $({\mathbf Q'}, {\mathbf x'})$ mutation-equivalent if one can be obtained from the other by quiver mutations. 
\begin{dfn}\label{dfn:cluster}
    Start from a seed $(\mathbf Q,\mathbf x=(x_1,\ldots,x_n))$,
    and let $\mathcal C$ be the collection of all cluster variables of all seeds that are mutation-equivalent to $(\mathbf Q,\mathbf x=(x_1,\ldots,x_n))$ via various seed mutations at gauge nodes. 
    A cluster algebra is defined to be $\mathscr A_{\mathbf Q}=\mathbb Q[\mathcal C]$, generated by all cluster variables, with the exchange relations. 
    The seed $(\mathbf Q,\mathbf x=(x_1,\ldots,x_n))$ is called the initial seed, $\mathbf x$ is called the initial cluster and variables $x_1,\ldots,x_n$ are called initial cluster variables. 
\end{dfn}

Cluster algebra $\mathscr{A}_{\mathbf Q}$ is a commutative ring embedded in an ambient field $\mathscr{F}=\mathbb Q(x_1,\cdots, x_n)$. 
One can check that choosing another seed $(\tilde{\mathbf Q},\tilde{\mathbf x}=(\tilde x_1,\ldots,\tilde x_n))$ that is mutation-equivalent to $(\mathbf Q,\mathbf x)$ as an initial seed. The resulting cluster algebra is isomorphic to the one constructed via $(\mathbf Q,\mathbf x)$.

Notice that we work with unlabeled seeds, which means we identify two seeds $(\mathbf Q,\mathbf x)$ and $(\mathbf Q',\mathbf x')$
if the QPs and clusters differ by relabeling. 

 In the theory of cluster algebra, 
 the cluster algebra associated to a QP with $n$ nodes among which there are  $n-m$ frozen nodes is viewed as a ring in $\mathbb Q[x_{m+1},\ldots,x_n](x_1,\ldots,x_{m})$ 
 where the cluster variables $x_{m+1},\cdots,x_{n}$ attached to the frozen nodes are viewed as coefficients. 
 However, in our setting, we view the cluster variables associated to the frozen nodes as usual cluster variables, and we don't perform quiver mutations at those frozen nodes.
\begin{ex}\label{ex:A3}
    Consider an $A_2$ quiver with one frozen node in Figure \ref{fig:clusterA3}.
\begin{figure}[ht]
\centering
\begin{tikzpicture}
	\node[draw,
	circle,
	minimum size=1cm,
	] (node1) at (-2,0){};
    \node[draw,
	circle,
	minimum size=1cm,
	] (node2) at (0,0){};
	\node[draw,
	minimum width=1cm,
	minimum height=1cm,
	] (node3) at (2,0){};
	\draw[-stealth] (node1) -- (node2);
	\draw[-stealth] (node2) -- (node3);
    \node at (-2,-0.7){$1$};
    \node at (0,-0.7){$2$};
    \node at (2,-0.7){$3$};
\end{tikzpicture}
\caption{An $A_2$ quiver with one frozen node.} 
\label{fig:clusterA3}
\end{figure}
    Since the node $3$ is frozen, we only perform mutations at nodes $1$ and $2$. Let the initial variables be $\mathbf x=(x_1,x_2,x_3)$. 
    Then under seed mutations, cluster variables change as follows.
    \begin{align*}
       & (x_1,x_2,x_3)\xrightarrow[]{\mu_1}(\frac{x_2+1}{x_1},x_2,x_3)\xrightarrow[]{\mu_2}(\frac{x_2+1}{x_1},\frac{x_1+x_3+x_2x_3}{x_1x_2},x_3)\xrightarrow[]{\mu_1}\nonumber\\
       &(\frac{x_1+x_3}{x_2},\frac{x_1+x_3+x_2x_3}{x_1x_2},x_3)\xrightarrow[]{\mu_2}(\frac{x_1+x_3}{x_2},x_1,x_3)\xrightarrow[]{\mu_1}(x_2,x_1,x_3).
    \end{align*}
    Notice that there are $6$ distinct cluster variables, and the corresponding cluster algebra is 
    $\mathscr A=\displaystyle\mathbb Q\left[x_1,x_2,x_3, \frac{x_2+1}{x_1}, \frac{x_1+x_3+x_2x_3}{x_1x_2}, \frac{x_1+x_3}{x_2}\right]$ with cluster exchange relations. 
\end{ex}
In the above example, $(1)$ all non-initial cluster variables can be written as rational functions of the initial variables whose denominators are monomials of initial cluster variables and $(2)$ numerators are polynomials of initial cluster variables with positive integer coefficients. 
The two observations $(1)$ and $(2)$ are called the Laurent phenomenon and the positivity conjecture in cluster algebra theory. 
The Laurent phenomenon is proved in \cite{clusteralg:1} where the positivity conjecture is proposed. 
The positivity conjecture is proved in \cite{cluster:positivityQin,cluster:positivityMSL1,cluster:positivityLSR2}.
\subsection{A-mutation-equivalent quivers}\label{subsec:Anmutationequivalentqivers}

Consider an $A_n$-type quiver with $n$ gauge nodes and one frozen node in
Figure \ref{fig:undecoratedAn}. 
When we talk about $A_n$-type quiver, we always refer to Figure \ref{fig:undecoratedAn}. Denote the corresponding cluster algebra by $\mathscr A_n\subset \mathbb Q(x_1,\ldots,x_{n+1})$.
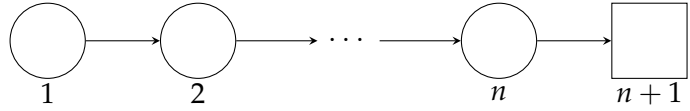
\begin{figure}[ht]
\centering
\begin{tikzpicture}
	\node[draw,
	circle,
	minimum size=1cm,
	] (node1) at (-4,0){};
    \node[draw,
	circle,
	minimum size=1cm,
	] (node2) at (-2,0){};
    \node[draw,
	circle,
	minimum size=1cm,
	] (node4) at (2,0){};
     \node(node3) at (0,0){$\cdots$};
	\node[draw,
	minimum width=1cm,
	minimum height=1cm,
	] (node5) at (4,0){};
	\draw[-stealth] (node1) -- (node2);
	\draw[-stealth] (node2) -- (node3);
	\draw[-stealth] (node3) -- (node4);
	\draw[-stealth] (node4) -- (node5);
    \node at (-4,-0.7){$1$};
    \node at (-2,-0.7){$2$};
    \node at (2,-0.7){$n$};
    \node at (4,-0.7){$n+1$};
\end{tikzpicture}
\caption{An $A_n$-type quiver with one frozen node.} 
\label{fig:undecoratedAn}
\end{figure}
\begin{lem}\label{lem:Andirections}
    All orientations of $A_n$-type quivers in Figure \ref{fig:undecoratedAn} are mutation-equivalent. 
\end{lem}
\begin{proof}
We only have to prove that we can invert any arrow $i\rightarrow i+1$ in the Figure and preserve all remaining arrows by several quiver mutations. 
This can be achieved by performing quiver mutations $\mu_1\rightarrow \mu_2\rightarrow\cdots\rightarrow \mu_i$. 
\end{proof}

The work \cite{Amutationequivalence} has  given a precise description of quivers that are mutation-equivalent to the $A_{n}$ quiver without frozen nodes. 
Our situation is similar to that in \cite{Amutationequivalence} with some modifications. 
Let $\Omega_n^A$ denote the set of QPs that are mutation-equivalent to the $A_n$-type quiver in Figure \ref{fig:undecoratedAn}. 
Let $v_f$ denote the only frozen node, which we refer to as the root of the quiver.
\begin{lem}[\cite{Amutationequivalence}]\label{lem:omegan}
A QP $\mathbf Q=(Q_0,Q_1,W)$ lying in $\Omega_n^A$ can be described as follows.
\begin{enumerate}
    \item All cycles are of length 3, and any arrow is contained in at most one $3$-cycle.
    \item The frozen node $v_f$ has at most two adjacent arrows, 
    and when it has exactly two adjacent arrows, 
    they belong to a 3-cycle.
    \item A gauge node has at most four adjacent arrows.  
    If a gauge node $v$ has only two arrows, then it can be a sink, source or there is a path passing through $v$. 
    If a gauge node has exactly three adjacent arrows, 
    then two of them are contained in a 3-cycle, and the third one doesn't belong to any 3-cycle. 
    If a gauge node has four adjacent arrows, then two of them belong to a 3-cycle and the remaining two belong to another 3-cycle.
    \item The potential function $W$ is the sum of cubic monomials of all 3-cycles. 
    \end{enumerate}
\end{lem}
 For our purpose, we give a recursive description of QPs in $\Omega^A_n$. 
\begin{lem}\label{lem:recursivepattern}
The underlying quiver diagrams of QPs in $\Omega^A_n$ can be described recursively as follows.
    \begin{enumerate}
        \item Step one, start from the root. 
        At most two nodes are adjacent to the root, 
        one going out, denoted by $v_2$ and the other coming in, denoted by $v_1$.
        If there are exactly two adjacent nodes, they form a 3-cycle $v_1\rightarrow v_f\rightarrow v_2\rightarrow v_1$, as illustrated in Figure \ref{fig:root}.
        \item Step two, if $v_f$ has exactly two adjacent nodes, 
        the sub-quiver $\mathbf Q\backslash \{v_f,v_2\rightarrow v_1\}$ is a union of two quivers $\mathbf Q_1\sqcup \mathbf Q_2$ where $\mathbf Q_1\in \Omega_k^A$ and $\mathbf Q_2\in \Omega_l^A$ with $k+l=n-2$. Here we view $v_1$ as the root of the quiver $\mathbf Q_1$ and $v_2$ as the root of the quiver $\mathbf Q_2$. 
        If $v_f$ has only one adjacent node $v_i$ $(i=1,2)$, then the sub-quiver $\mathbf Q_1:=\mathbf Q\backslash v_f$ is a quiver in $\Omega_{n-1}^A$ if we view the node $v_i$ as the root of the new quiver.  
        \item 
        Recursively construct the sub-quivers $\mathbf Q_i$ attached to $v_i$ as Step one and two.
        Since there are only $n+1$ nodes, this recursion terminates when all nodes are exhausted. 
    \end{enumerate}
\end{lem}
\begin{figure}[ht]
\centering
\begin{tikzpicture}
	\node[draw,
	circle,
	minimum size=1cm,
	] (node1) at (0,0){};
    \node at (0.7,0){$v_2$};
    \node at (0.4,1.2){$v_f$};
    \node[draw,
	minimum width=1cm,
	minimum height=1cm,
	] (nodev1) at (0,2){};
    \node[draw,
	circle,
	minimum size=1cm,
	] (node2) at (3,0){};
    \node[draw,
	minimum width=1cm,
	minimum height=1cm,
	] (nodev2) at (3,2){};
    \node[draw,
	circle,
	minimum size=1cm,
	] (node3) at (-3,0){};
    \node[draw,
	circle,
	minimum size=1cm,
	] (node4) at (-5,0){};
	\node[draw,
	minimum width=1cm,
	minimum height=1cm,
	] (nodev3) at (-4,2){};
	\draw[-stealth] (nodev1) -- (node1);
	\draw[-stealth] (node2) -- (nodev2);
	\draw[-stealth] (node3) -- (node4);
	\draw[-stealth] (node4) -- (nodev3);
	\draw[-stealth] (nodev3) -- (node3);
    \node at (3.7,0){$v_1$};
    \node at (3.4,1.2){$v_f$};
    \node at (-2.3,-0.2){$v_2$};
    \node at (-4.2,-0.2){$v_1$};
    \node at (-3.3,1.2){$v_f$};
     \node at (-4,-1){$(1)$};
     \node at (0,-1){$(2)$};
     \node at (3,-1){$(3)$};
\end{tikzpicture}
\caption{} 
\label{fig:root}
\end{figure}
\begin{proof}
One can easily check that this recursively constructed set is closed under quiver mutations and they are all mutation-equivalent to the $A_n$-quiver in Figure \ref{fig:undecoratedAn}. Hence it is equal to the $\Omega^A_n$.
\end{proof}

\begin{dfn}
    Let $\mathbf Q=(Q_0,Q_1,W)$ be an arbitrary QP in $\Omega^A_n$.
    For each node $v$, we define the \textit{distance} of $v$ and the frozen node $v_f$ to be the length of the shortest un-oriented path connecting the $v$ and the root $v_f$.  
    Let $v,v'\in Q_0$. We define $v\prec v'$ if the node $v'$ is on the shortest un-oriented path connecting the $v$ and the root $v_f$.
\end{dfn}
 It is obvious that a node $v$ has at most two maximal nodes that are smaller than it.
 When there are exactly two such maximal nodes, they form a 3-cycle.
 A node $v$ has at most one minimal node that is greater than $v$.

\subsection{Quiver varieties}
 Now we move to the definition of quiver varieties. We refer readers to the book \cite{quiver}. 
Fix a QP $\mathbf Q=(Q_f\subset Q_0,Q_1,W)$, and decorate it with integers $\mathbf r=(r_v)_{v\in Q_0}$. 
\begin{dfn}
The decorated QP gives input data for a GIT quotient:
\begin{enumerate}
    \item $V=\oplus_{a\in Q_1}\mathbb C^{r_{s(a)}\times r_{t(a)}}$ is the representation space of the quiver, and we denote a general element in $V$ by $(A_{s(a)\rightarrow t(a)})_{a\in Q_1}$ where each $A_{s(a)\rightarrow t(a)}$ is a $r_{s(a)}\times r_{t(a)}$ matrix;
    \item $G=\prod_{v\in Q_0\backslash Q_f}GL(r_v)$ is the gauge group of the quiver;
    \item the group $G$ acts on $V$ in the following way
    \begin{equation}\label{eqn:GonV}
        (g_{v})_{v\in Q_0\backslash Q_f}\cdot (A_{s(a)\rightarrow t(a)})_{a\in Q_1}=( g_{s(a)}A_{s(a)\rightarrow t(a)}g_{t(a)}^{-1})_{a\in Q_1},
    \end{equation}
    where we make $g_v$ be the identity matrix when $v\in Q_f$;
    \item $\theta\in \chi(G):=\op{Hom}(G, \mathbb C^*)$ is a character of the gauge group.
\end{enumerate}
Let $V^{ss}_{\theta}(G),\,V^s_{\theta}(G)$ be the semistable and stable locus of $V$ under the action of $G$.
The quiver variety is defined  by the GIT quotient $\mathcal X=V\sslash_{\theta} G:=V^{ss}_{\theta}\slash G$. 
When the potential $W$ is nontrivial, we need to consider the critical locus $Z=\{dW=0\}\subset V$. We can similarly define $Z^{ss}_\theta(G), Z^{s}_{\theta}(G)$ and the GIT quotient $\mathcal Z=\{dW=0\}\sslash_{\theta} G$. We call the GIT quotient $\mc Z$ \textit{the variety of the decorated QP}. 
\end{dfn}
In the above definition, we always assume that $V^{ss}_\theta(G)=V^s_{\theta}(G)$ and $Z^{ss}_\theta(G)=Z^s_\theta(G)$.
\begin{dfn}
    For a decorated QP $(\mathbf Q=(Q_f\subset Q_0, Q_1,W), \mathbf r)$, define the \textit{outgoing} of a gauge node $v$ to be $N_f(v)=\sum_{a\in Q_1,s(a)=v}r_{t(a)}$, and the \textit{incoming} to be $N_a(v)=\sum_{a\in Q_1,t(a)=v}r_{s(a)}$.
\end{dfn}
\begin{ex}\label{ex:An}
    Consider an $A_n$-type quiver with decorations $N_1<N_2<\ldots<N_{n+1}$.
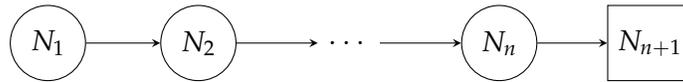
\begin{figure}[H]
\centering
\begin{tikzpicture}
	\node[draw,
	circle,
	minimum size=1cm,
	] (node1) at (-4,0){$N_1$};
    \node[draw,
	circle,
	minimum size=1cm,
	] (node2) at (-2,0){$N_2$};
    \node[draw,
	circle,
	minimum size=1cm,
	] (node4) at (2,0){$N_n$};
     \node(node3) at (0,0){$\cdots$};
	\node[draw,
	minimum width=1cm,
	minimum height=1cm,
	] (node5) at (4,0){$N_{n+1}$};
	\draw[-stealth] (node1) -- (node2);
	\draw[-stealth] (node2) -- (node3);
	\draw[-stealth] (node3) -- (node4);
	\draw[-stealth] (node4) -- (node5);
 %   \node at (-3,-0.6){$1$};
 %   \node at (-1.5,-0.6){$2$};
 %   \node at (1.5,-0.6){$n$};
 %   \node at (3,-0.6){$n+1$};
\end{tikzpicture}
\caption{A decorated $A_n$-type quiver with one frozen node.} 
\label{fig:Anqv}
\end{figure}
\iffalse
    \begin{figure}[H]
        \centering
        \includegraphics[width=2.8in]{An.png}
       \caption{}
        \label{fig:Anqv}
    \end{figure}
\fi
\noindent
The input data for the QP includes $V=\oplus_{i=1}^{n}\mathbb C^{N_i\times N_{i+1}}$, $G=\prod_{i=1}^{n}GL(N_i)$, and $G$ acts on $V$ as \eqref{eqn:GonV}. 
Choose the character $\theta=\prod_{i=1}^{n}\det(g_i)^{\sigma_i}$ for $\sigma_i>0$. Then $$V_\theta^{ss}(G)=\{\text{each } A_{s(a)\rightarrow t(a)} \text{ is nondegenerate}\}.$$ Hence $V\sslash_\theta G$ parameterizes subspaces of dimensions $N_1<N_2\ldots<N_{n+1}$ in $\mathbb C^{N_{n+1}}$, 
which is a partial flag variety and denoted by $Fl:=Fl(N_1,\ldots,N_{n+1})$.
\end{ex}
\subsection{Quiver mutations for decorated quivers}\label{sec:quivervarietieswithmu}
Fix a decorated QP $(\mathbf Q=(Q_f\subset Q_0, Q_1, W),\mathbf r=(r_v)_{v\in Q_0})$. 
Performing a quiver mutation $\mu_v$ at a gauge node $v$ to the decorated QP $(\mathbf Q,\mathbf r)$, 
the underlying QP changes as Definition \ref{dfn:mutation} and the integers change as follows, 
\begin{equation}\label{eqn:Nchange}
r_{u}'=
    \begin{cases}
        \max\{N_f(v),N_a(v)\}-r_v &\text{ if } u=v \\
       r_u &\text{ otherwise}.
    \end{cases}
\end{equation}
\begin{lem}\label{lem:decorationsofQinP}
For any decorated QP $(\mathbf Q=(Q_0,Q_1,W),\mathbf r)\in \Omega^A_n$, the associated integers are definite and can be determined recursively as follows. 
\begin{enumerate}
    \item Step one, we start from the root $v_f$ whose assigned integer is always $N_{n+1}$. 
    By Lemma \ref{lem:recursivepattern}, there are at most $2$ nodes attached to it which are denoted by $v_1,v_2$ with arrows being shown in Figure \ref{fig:root}. 
    The integer assigned to $v_1$ is $r_{v_1}=N_l$ and the integer assigned to $v_2$ is $r_{v_2}=N_{n+1}-N_{l+1}$ for $n\geq l\geq 0$. We formally set $N_0=0$. 
    When $l=0$, then there is only one node $v_2$ attached to $v_f$. When $l=n$, then $N_{n+1}-N_l=0$ and there is only one node $v_1$ attached to $v_f$, see Figure \ref{fig:root} $(2)$ and $(3)$.
    \item Step two, if there are exactly two nodes $v_1\rightarrow v_f\rightarrow v_2\rightarrow v_1$, 
    $\mathbf Q\backslash 
    \{v_f,v_2\rightarrow v_1\}=\mathbf Q_1\sqcup \mathbf Q_2$. The subquiver $\mathbf Q_1$ lies in $\Omega^A_{l-1}$ and is  mutation-equivalent to the $A_{l-1}$-type quiver with decorations $N_1<N_2<\cdots<N_l$. The other one $\mathbf Q_2$ is mutation-equivalent to $A_{n-l-1}$ with decorations $N_{l+2}-N_{l+1}<\cdots<N_{n+1}-N_{l+1}$.
    The cases when there is only one node $v_1\,(v_2)$ adjacent to the $v_f$ can be obtained by letting $l=n\, (l=0)$.
    \item Step three, recursively construct the decorations of the quivers $\mathbf Q_i$ in Step two by viewing $v_i$ as root of the new quivers $\mathbf Q_i$ and utilizing the steps one and two. 
\end{enumerate} 
\end{lem}
\begin{proof}
    One can check that the decorated QPs described above are closed under quiver mutations and mutation equivalent to the decorated $A_n$-quiver in Figure \ref{fig:Anqv}. 
\end{proof}
From the above recursive relation, one can find that if $v\prec v'$, the $r_v<r_{v'}$. 
One can also easily get that for any node $v\in Q_0$, the assigned integer can only be $r_v=N_k-N_l$ for some $0\leq l<k\leq n+1$.

By abuse of notation, we still denote the set of decorated QPs that are mutation-equivalent to the decorated $A_n$ by $\Omega^A_n$.

For an arbitrary decorated QP $(\mathbf Q, \mathbf r)\in \Omega_n^A$,
combining the local description in Lemma \ref{lem:omegan} and the recursive description for decorated integers in Lemma \ref{lem:decorationsofQinP}, we can describe decorations of nodes adjacent to any gauge node $v$. 
The node $v$ has at most four adjacent arrows which we assume to be $v_1,\cdots, v_4$.
There exist integers $m> k>p\geq l\geq 0$ such that $r_v=N_k-N_l,\, r_{v_1}=N_m-N_l, \, r_{v_3}=N_p-N_l, \, \,
r_{v_2}=N_m-N_{k+1},\, 
r_{v_4}=N_k-N_{p+1}$.
Assume that $N_f(v)>N_a(v)$.
\begin{enumerate}
\item When the node $v$ has exactly four adjacent arrows, which we denote by $v_1,\cdots,v_4$, the local picture can be described as Figure \ref{fig:localbehavior} (a) for $m-1>k>p+1\geq l\geq 0$.
We formally let $N_0=0$.

The case when $N_f(v)<N_a(v)$ is in Figure \ref{fig:localbehavior} (b) and one can find it is mutation-equivalent to the left-hand side one. 
\item 
The case when the node $v$ has only three adjacent nodes can be realized by taking special values for $m,k,p,l$. 
\begin{itemize}
\item 
When $p=l$, $m-1>k>l+1$, then $N_p-N_l=0$. In this case, we delete the node $v_3$, and the node $v$ has $3$ adjacent nodes $v_1,v_2,v_4$. 
\item When $p=k-1$, similar as above, the node $v_4$ vanishes. 
\item When $k=m-1$, $N_m-N_{k+1}=0$ and the node $v_2$ vanishes.
\end{itemize}
\item The case when the node $v$ has only two adjacent arrows can also be obtained by taking special values for $m,k,p,l$. 
When $p=l,k=l+1$, the node $v$ has two adjacent nodes $v_1,v_2$ and they form a three-cycle. When $m=k+1, p=k-1$, the node $v$ has two adjacent nodes $v_1,v_3$.  When $m=k+1, p=l$, the node $v$ is a source and has two adjacent $v_1,v_4$. 
\item When $m=k+1=l+2$, the node $v$ has only one adjacent node.
\end{enumerate}
Therefore, the case where the node $v$ has four adjacent arrows is typically used to represent the local structure of a gauge node $v$. It is important to note, however, that this representation implicitly includes other special cases.
\begin{figure}[H]
\centering
\begin{tikzpicture}
	\node[draw,
	circle,
	minimum size=1cm,
	] (nodev) at (-3,0){$r_v$};
    \node at (-3,-0.7){$v$};
    \node[draw,
	circle,
	minimum size=1cm,
	] (nodev1) at (-4.5,1.5){$r_{v_1}$};
    \node at (-4.5,0.8){$v_1$};
    \node[draw,
	circle,
	minimum size=1cm,
	] (nodev2) at (-1.5,1.5){$r_{v_2}$};
    \node at (-1.4,0.8){$v_2$};
    \node at (-4.5,-2.2){$v_3$};
    \node at (-1.5,-2.2){$v_4$};
    \node at (-2,0.6){$B_{v_1}$};
    \node at (-3,-1.7){$B_{v}$};
    \node at (1.3,0.8){$v_1$};
    \node at (4.7,0.8){$v_2$};
    \node at (1.5,-2.2){$v_3$};
    \node at (4.5,-2.2){$v_4$};
    \node at (4.7,0){$B_{v}'$};
    \node at (2.1,-0.6){$B_{v_1}'$};
    \node at (-3,-2.8){$(a)$};
    \node at (3,-2.8){$(b)$};
    \node at (3,-0.7){$v$};
    \draw[implies-implies,double equal sign distance](-0.5,0)--(0.5,0);
    \node at (0,0.2){$\mu_v$};
    \node[draw,
	circle,
	minimum size=1cm,
	] (nodev3) at (-4.5,-1.5){$r_{v_3}$};
    \node[draw,
	circle,
	minimum size=1cm,
	] (nodev4) at (-1.5,-1.5){$r_{v_4}$};
	\draw[-stealth] (nodev) -- (nodev1);
	\draw[-stealth] (nodev1) -- (nodev2);
	\draw[-stealth] (nodev2) -- (nodev);
	\draw[-stealth] (nodev) -- (nodev4);
	\draw[-stealth] (nodev4) -- (nodev3);
	\draw[-stealth] (nodev3) -- (nodev);
    %right hand side
	\node[draw,
	circle,
	minimum size=1cm,
	] (nodev) at (3,0){$r_v'$};
    \node[draw,
	circle,
	minimum size=1cm,
	] (nodev2) at (4.5,1.5){$r_{v_2}$};
    \node[draw,
	circle,
	minimum size=1cm,
	] (nodev1) at (1.5,1.5){$r_{v_1}$};
    \node[draw,
	circle,
	minimum size=1cm,
	] (nodev4) at (4.5,-1.5){$r_{v_4}$};
    \node[draw,
	circle,
	minimum size=1cm,
	] (nodev3) at (1.5,-1.5){$r_{v_3}$};
	\draw[-stealth] (nodev1) -- (nodev);
	\draw[-stealth] (nodev) -- (nodev3);
	\draw[-stealth] (nodev3) -- (nodev1);
	\draw[-stealth] (nodev) -- (nodev2);
	\draw[-stealth] (nodev2) -- (nodev4);
	\draw[-stealth] (nodev4) -- (nodev);
\end{tikzpicture}
\caption{The local picture near a gauge node $v$ and its quiver mutation. There exist integers $m> k>p\geq l\geq 0$ such that $r_v=N_k-N_l, r_{v_1}=N_m-N_l, r_{v_3}=N_p-N_l, r_{v_2}=N_{m}-N_{k+1}, r_{v_4}=N_k-N_{p+1}$, $r_{v}'=N_m-N_{p+1}$ if we assume that $N_f(v)>N_a(v)$.} 
\label{fig:localbehavior} 
\end{figure}
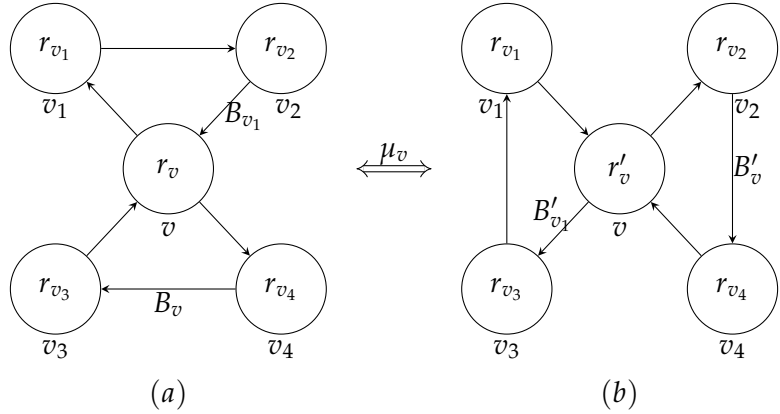
\iffalse
\begin{figure}[H]
    \centering
    \includegraphics[width=0.7\linewidth]{localproof.png}
    \caption{The local picture near a gauge node $v$ with $m>k>p\geq l\geq 0$.}
    \label{fig:localbehavior}
\end{figure}
\fi
\subsection{Varieties of QPs mutation equivalent to A-type}\label{subsec:quivervarmuteqtoAn}
 Let $(\mathbf Q=(Q_0,Q_1,W), \mathbf r)\in \Omega^A_n$ be an arbitrary decorated QP in $\Omega^A_n$. 
 Denote the corresponding input data by $(V, G, \theta)$ where $G=\prod_{i\in Q_0\backslash Q_f}GL(r_i)$ and $\theta(g)=\prod_{i\in Q_0\backslash Q_f}\det(g_i)^{\sigma_i}$.
\begin{lem}\label{lem:phasevargeneralquiver}
    Then the only valid stability condition $\theta$ can be described as follows. 
    \begin{enumerate}
    \item
        We start from the two nodes $v_1,v_2$ in Lemma \ref{lem:recursivepattern} that are attached to the root. 
       Let $v_1=w_0\rightarrow w_1\rightarrow w_2\rightarrow\cdots\rightarrow w_h$ be the longest oriented path attached to $v_1$ and $w_i$ is the minimal node that is greater than $w_{i+1}$. 
        Let $v_2=u_0\leftarrow u_1\leftarrow u_2\leftarrow \cdots\leftarrow u_{m}$ be the longest oriented path attached to $v_2$ such that $u_{i}$ is the minimal node that is greater than the node $u_{i+1}$. 
        The $\sigma_{v_i}$ for $i=1,2$ satisfy the following conditions
        \begin{align}\label{eqn:phasesforgeneral}
            &\sigma_{v_1}>0,\,\sigma_{v_2}<0
            \nonumber\\
            &\sigma_{v_1}+ \sum_{i=1}^{h}\sigma_{w_i}>0,\,\,
           \sigma_{v_2}+ \sum_{i=1}^{m}\sigma_{u_i}<0\,.
        \end{align}
        \item We move to a smaller node $v_i$ which is regarded as the root of the subquiver $\mathbf Q_i$ in Lemma \ref{lem:decorationsofQinP}. 
        Let $\tilde v_1$ and $\tilde v_2$ be the two nodes that are smaller than $v_i$, and then the conditions for $\sigma_{\tilde{v_1}}$ and $\sigma_{\tilde v_2}$ are similar as \eqref{eqn:phasesforgeneral}.
    \end{enumerate}
    In this phase, the semistable locus of $Z=\{dW=0\}$ can be described as follows. 
    For each node $v$, there are at most two maximal nodes $v_1$ and $v_2$ such that $v_i\prec v$ with directions of arrows being the same as Figure \ref{fig:root}$(1)$. We denote the matrices by $v_1\xrightarrow[]{A_{v_1}} v \xrightarrow []{A_{v_2}}v_2\xrightarrow[]{B_v}v_1$. 
    \begin{enumerate}
        \item If there is only one maximal node that is smaller than $v$, $v_1\xrightarrow[]{A_{v_1}} v$ or $v_2\xleftarrow[]{A_{v_2}}v$, then in the phase described above, $A_{v_i}$ $i=1$ or $i=2$ is non-degenerate. 
        \item If there are exactly two maximal nodes smaller than $v$, in the phase described above, the two matrices $A_{v_i},i=1,2$ are non-degenerate, $B_v=0$, and $A_{v_1}A_{v_2}=0$.
    \end{enumerate}

We consider a new quiver $\bar{\mathbf Q}$ by deleting all arrows $v_2\rightarrow v_1$.  
Then $\bar{\mathbf Q}$ is a binary tree. 
Let $\mc Y$ be the corresponding quiver variety under the phase above. 
Then our critical locus $\mc Z=Z\sslash_{\theta}G$ is a complete intersection in $\mc Y$ defined by $A_{v_1}A_{v_2}=0$ for those matrices representing two nodes $v_1,v_2$ that are maximal nodes smaller than a node $v$. 
\end{lem}
This proof of this Lemma is overly intricate, so we prefer to omit it and instead consider a concrete $A_{3}$-mutation-equivalent quiver in the next example.
\begin{ex}\label{ex:A2variety}
    Consider a quiver in $\Omega^A_3$ as Figure \ref{fig:a3quiver} which is obtained via a quiver mutation $\mu_2$ to $A_3$ quiver at the gauge node $2$. 
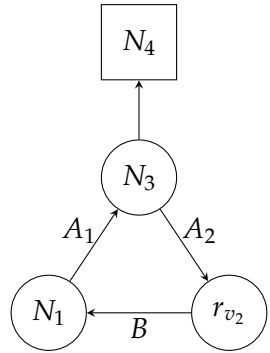
\begin{figure}[H]
\centering
\begin{tikzpicture}
	\node[draw,
	circle,
	minimum size=1cm,
	] (node3) at (0,0){$N_3$};
   % \node at (0.5,0.1){$3$};
    \node[draw,
	circle,
	minimum size=1cm,
	] (node1) at (-1.2,-1.8){$N_1$};
    \node at (0,-2){$B$};
    \node[draw,
	circle,
	minimum size=1cm,
	] (node2) at (1.2,-1.8){$r_{v_2}$};
    \node at (-0.8,-0.7){$A_1$};
    \node at (0.8,-0.7){$A_2$};
    \node[draw,
	minimum width=1cm,
	minimum height=1cm,
	] (nodef) at (0,1.8){$N_4$};
	\draw[-stealth] (node1) -- (node3);
	\draw[-stealth] (node3) -- (node2);
	\draw[-stealth] (node2) -- (node1);
	\draw[-stealth] (node3) -- (nodef);
\end{tikzpicture}
\caption{$r_{v_2}=N_3-N_2$.} 
\label{fig:a3quiver}
\end{figure}
\noindent
    The only valid stability condition $\theta(g)=\prod_{i=1}^3\det(g_i)^{\sigma_i}$ is 
    \begin{equation*}
        \sigma_1>0,\,\sigma_2<0,\,\sigma_3+\sigma_2>0\,.
    \end{equation*}
    Let $Z:=\{dW=0\}=\{A_1A_2=0,\,A_2B=0,\,BA_1=0\}$ be the critical locus of the potential. 
    The semistable locus under the stability condition is 
    \begin{equation*}
        Z^{ss}_\theta(G)=\{B=0, A_1A_2=0\,|\,\text{matrices }A_1,A_2,A_3 \text{ nondegenerate}\}.
    \end{equation*}
    By deleting the arrow $v_2\rightarrow v_1$, we get a new quiver which is a binary tree. The critical locus $\mc Z:=Z\sslash_\theta G$ is a subvariety defined by $\{A_1A_2=0\}$ of the quiver variety of this binary tree.
\end{ex}
\iffalse
\begin{figure}[H]
    \centering
    \includegraphics[width=1in]{a3quiver.png}
    \caption{}
    \label{fig:a3quiver}
\end{figure}
\fi

\section{Seiberg duality conjecture and cluster algebra conjecture}\label{sec:GWandqcoh}
\subsection{Gromov-Witten theory and quantum cohomology ring}
We refer to the nice book \cite{GW:mirrorsym} about the fundamental properties of Gromov-Witten(GW) theory. 
\begin{dfn}
Let $\mc X$ be a smooth projective variety.
A stable map to $\mc X$ denoted by $(C, p_1,\ldots,p_n;f)$ consists of the following data:
\begin{enumerate}
    \item
    a nodal curve $(C,p_1,\ldots,p_n)$ with $n\geq 0$ distinct nonsingular markings $(p_1,\ldots,p_n)$, 
    \item a \textit{stable map} 
    $f:(C,p_1,\ldots,p_n)\rightarrow \mc X$ such that every component of $C$ of genus 0, which is contracted by $f$, must have at least three special (marked or singular) points, and every component of $C$ of genus one which is contracted by $f$, must have at least one special point. 
    \end{enumerate}
\end{dfn}
The degree of a stable map $(C, p_1,\ldots, p_n; f)$ is defined as the homology class of the image $\beta=f_*[C]$.
For a fixed curve class $\beta\in H_2(\mc X, \Z)$, let $\overline{M}_{g,n}(\mc X, \beta)$ denote the stack of stable maps from the genus-g curves $C$ with $n$ markings to $\mc X$ such that $f_*[C]=\beta$.
When $\mc X$ is projective, $\overline{M}_{g,n}(\mc X, \beta)$ is a proper and separated DM stack and admits a perfect obstruction theory. 
Hence, we can construct the virtual fundamental class $[\overline{M}_{g,n}(\mc X, \beta)]^{vir}\in A_{\text{vdim}}(\overline{M}_{g,n}(\mc X, \beta))$ 
where $\text{vdim}=\int_{\beta}c_1(\mc X)+(\dim(\mc X)-3)(1-g)+n$. See \cite{GW:LT,GW:BFintrinsic,GW:B}.

Let
$\pi:\mc C_{g,n}\rightarrow\overline{M}_{g,n}(\mc X,\beta)\,$ be the universal curve and $s_i$ be sections of $\pi$ for each marking $p_i$. 
Let $\omega_\pi$ be the relative dualizing sheaf and $\mc P_i=s_i^*(\omega_\pi)$ be the cotangent bundle at the $i$-th marking.
Define the $\psi$-class by $\psi_i:=c_1(\mc P_i) \in H^2(\overline M_{g,n}(\mc X,\beta))$. 
Define evaluation maps by
\begin{align*}
   ev_i:\overline M_{g,n}(\mc X,\beta)&\longmapsto \mc X \nonumber\\
   (C, p_1,\ldots,p_n;f)&\longmapsto f(p_i)\,.
\end{align*}
Let $\gamma_1,\ldots,\gamma_n\in H^*(\mc X,\mathbb Q)$ be cohomology classes and $a_i$ $i=1,\ldots,n$ be positive integers. The GW invariant is defined as
\begin{equation}\label{eqn:GWinv}  \langle\tau_{a_1}\gamma_1,\ldots,\tau_{a_n}\gamma_n \rangle_{g,n,\beta}^{\mc X}:=\int_{[\overline M_{g,n}(\mc X, \beta)]^{vir}}\prod_{i=1}^{n}\psi_i^{a_i}ev_i^*(\gamma_i)\,.
\end{equation}
\iffalse
Let $\alpha_0=1,\alpha_1,\ldots, \alpha_m\in H^*(\mc X,\mathbb Q)$ be a set of generators of the cohomology group, and $\alpha^0,\alpha^1,\ldots, \alpha^m\in H^*(\mc X,\mathbb Q)$ be the Poincar\'e dual. 
The small $\mc J$-function of $\mc X$, which comprises genus-zero GW invariants, is defined by
\begin{equation}\label{eqn:GWJ}
    \mc J^{\mc X}(q,\mathbf{t},u)=
  \sum_{i=0}^{m} \sum_{(k\geq 0,\beta )} \alpha^i \langle \frac{\alpha_i}{u(u-\psi_{\bullet})}\mathbf{t}\ldots \mathbf{t}\rangle_{0,k+1,\beta}^{\mc X}\frac{q^{\beta}}{k!} \,.
\end{equation}
where $\mathbf{t}\in H^{\leq 2}(\mc X,\mathbb Q)$.
\fi
\begin{dfn}
    Let $\gamma_1, \gamma_2,\gamma_3\in H^*(\mc X)$ be arbitrary elements. The small quantum product $*$ is defined by the relation
    \begin{equation*}
        (\gamma_1*\gamma_2,\gamma_3)= \sum_{\beta}\langle\gamma_1,\gamma_2,\gamma_3\rangle_{0,3,\beta}^{\mc X}{q^\beta}
    \end{equation*}
    where $(,)$ is the Poincar\'e pairing of $H^*(\mc X)$. 
    Let $R=\mathbb Q[[q^{\beta}]]$ be the ring of formal power series in $q^{\beta}$ 
  with coefficients in $\mathbb Q$.
 The ring $H^*(\mc X, \mathbb Q)\otimes R$, equipped with the product $*$, forms a commutative associative ring called the small quantum cohomology ring, denoted by $QH^*(\mc X)$. 
\end{dfn}
Assume that $\mc X$ admits a torus action, denoted by $S=(\mathbb C^*)^n$. 
Let $H^*_S(\mc X):=H^*(\mc X_G\times EG)$ be the equivariant cohomology ring of $\mc X$, and $H^*_S(pt,\mathbb Q)=\mathbb Q[\lambda_1,\ldots,\lambda_n]$. 
The torus action induces an action on $\overline M_{g,n}(\mc X, \beta)$ by sending a stable map $(C, p_1,\ldots, p_n;\,f)$ to $(C, p_1,\ldots, p_n;\, s\circ f)$ for each $s\in S$. 
There is an equivariant perfect obstruction theory in $\overline M_{g,n}(\mc X, \beta)$ and equivariant virtual fundamental class $[\overline M_{g,n}(\mc X, \beta)]^{vir,S}$. For $\omega_i\in H^*_S(\mc X)$, the equivariant Gromov-Witten invariants valued in $H^*_S(pt)$ are defined as
\begin{equation*}
   \langle\tau_{a_1}\omega_1,\ldots,\tau_{a_n}\omega_n \rangle_{g,n,\beta}^{\mc X,S}=
   \int_{[\overline M_{g,n}(\mc X, \beta)]^{vir,S}}\prod_{i=1}^{n}(\psi_i^S)^{a_i}ev_i^*(\omega_i)\,
\end{equation*}
where $\psi_i^S\in H^*_S(\overline M_{g,n}(\mc X, \beta))$ are lifts of $\psi_i$.
Let $F\subset \overline M_{g,n}(\mc X, \beta)$ denote a torus-fixed locus. 
There is an induced equivariant perfect obstruction theory on $F$ and 
\begin{equation*}  
\langle\tau_{a_1}\omega_1,\ldots,\tau_{a_n}\omega_n \rangle_{g,n,\beta}^{\mc X,S}=\sum_{F}\int_{F}  \frac{i_F^*\left(\prod_{i=1}^{n}(\psi_i^S)^{a_i}ev_i^*(\omega_i)\right)}{e^S(N^{vir}_F)}\,.
\end{equation*}
The summation is over all torus fixed loci $F$, the map $i_F: F\rightarrow \overline M_{g,n}(\mc X, \beta)$ is an embedding, and $N_F^{vir}$ is the virtual normal bundle of $F$.
Suppose $\mc X$ is projective and $\omega's$ are lifting of $\gamma_i's$ in  $H^*_S(\mc X,\mathbb Q)$. 
Then the nonequivariant limit of $\langle\tau_{a_1}\omega_1,\ldots,\tau_{a_n}\omega_n \rangle_{g,n,\beta}^{\mc X,S}$ is equal to the regular Gromov-Witten invariant $\langle\tau_{a_1}\gamma_1,\ldots,\tau_{a_n}\gamma_n \rangle_{g,n,\beta}^{\mc X}$.
See \cite{GW:GPequiv}.

We can similarly define an equivariant quantum product. 
For $\omega_1,\omega_2,\omega_3\in H^*_S(\mc X)$, we define the small quantum product $*$ as follows,
\begin{equation*}
    (\omega_1*\omega_2,\omega_3)= \sum_{\beta}\langle\omega_1,\omega_2,\omega_3\rangle^{\mc X,S}_{0,3,\beta}q^{\beta}.
\end{equation*}
Denote the equivariant small quantum cohomology ring by $QH^*_S(\mc X,\mathbb Q)$.

\subsection{Seiberg duality conjecture}\label{sec:seibergduality}
Let $(\mathbf Q=(Q_f\subset Q_0,Q_1,W), \mathbf r)$ and $(\mathbf Q'=(Q_f\subset Q_0, Q_1',W'),\mathbf r')$ be two decorated QPs that are related by a quiver mutation $\mu_v$ at the gauge node $v$. 
Let $(V,G,\theta)$ and $(V',G',\theta')$ be the input data of the associated GIT quotients, and $Z:=\{dW\}$ and $Z':=\{dW'=0\}$ be critical loci of potential functions $W$ and $W'$. 
 Write their stability conditions as $\theta(g)=\prod_{u\in Q_0\backslash Q_f}\det(g_u)^{\sigma_u}\in \chi(G)$ and $\theta'(g')=\prod_{u\in Q_0\backslash Q_f}\det(g_u')^{\sigma_u'}\in \chi(G')$. 
 Assume that the two phases $(\sigma_u)_{u\in Q_0\backslash Q_f}$ and $(\sigma_u')_{u\in Q_0\backslash Q_f}$ are related as follows,
 \begin{equation}\label{eqn:changeofphasemu}
 \theta_u'=
     \begin{cases}
         -\theta_v & \text{ for }u=v\\
         \theta_u+[sign(\sigma_v)b_{vu}]_+\theta_v &\text{ for }u\neq v
     \end{cases}
 \end{equation}
Let $\mc Z=Z\sslash_{\theta}G$ and $\mc Z'=Z'\sslash_{\theta'}G'$ be the associated varieties. 
Let $S=\prod_{a\in Q_f}(\mathbb C^*)^{N_a}$ be a torus acting on $\mc Z$ and $\mc Z'$ naturally.
The genus-zero Seiberg duality conjecture asserts that there is a ring isomorphism
\begin{equation*}
  \phi : H^*_S(\mc Z)\rightarrow H^*_S(\mc Z').
\end{equation*}
Let $\phi_*:H_2(\mc Z)\rightarrow H_2(\mc Z')$ be a group homomorphism such that for each $\alpha\in H^2(\mc Z)$, we have 
\begin{equation}\label{eqn:homologicalchange}
    \int_{\beta}\alpha=
    \int_{\phi_*(\beta)}\phi(\alpha).
\end{equation}
\begin{conjecture}[\cite{nonabelianGLSM:YR}]\label{conj:seibergdualityconj}
Let $\omega_1,\ldots,\omega_m\in H^*_S(\mc Z)$ be a set of cohomology classes and $\omega_i'=\phi(\omega_i)$.  
The all-genus Seiberg duality conjecture asserts that for $\beta'=\phi_*(\beta)$
\begin{equation*}
<\tau_{a_1}\omega_1,\ldots,\tau_{a_m}\omega_m>^{\mc Z,S}_{g,n,\beta} =
<\tau_{a_1}\omega'_1,\ldots,\tau_{a_m}\omega_m'>^{\mc Z',S}_{g,n,\beta'}\,.
\end{equation*}
\end{conjecture}

\begin{dfn}\label{dfn:tautbundle}
    For each node $u\in Q_0\backslash \{v_f\}$, we define its tautological bundle $S_u$ to be $Z\times \mathbb C^{r_u}/G$, where $G$ acts on $Z$ in the usual way and $G$ acts on $\mathbb C^{r_u}$ as follows: 
        \begin{equation*}
          g\cdot x= g_u^{-sign(\sigma_u)}x, \, x\in \mathbb C^{r_u}, g=(g_v)_{v\in Q_0\backslash Q_f}\in G, g_u\in GL(r_{u}),
        \end{equation*}
         where we view the vector $x$ as a column vector. 
         For each frozen node $f\in Q_f$, we formally define the tautological bundle $S_f$ to be a trivial bundle $\mathbb C^{r_f}$. 
\end{dfn}

Define bundles over gauge nodes
\begin{equation*}
    S_v^*=
    \begin{cases}
        S_v &\text{if } \sigma_v>0\\
        S_v^\vee& \text{ if }\sigma_v<0\,.\\
    \end{cases}
\end{equation*}
\begin{dfn}\label{dfn:quotbundle}
    For each node $u\in Q_0\backslash\{v_f\}$, we define the quotient bundle $Q_u$ over $\mc Z$ as follows. 
    If $\sigma_u>0$, the quotient bundle $Q_u$ is defined as a quotient via the following short exact sequence 
    \begin{equation*}
        0\rightarrow S_u\rightarrow \oplus_{u\rightarrow w\in Q_1}S_w^*\rightarrow Q_u\rightarrow 0.
    \end{equation*} 
    If $\sigma_u<0$, the quotient bundle $Q_u$ is defined via the following short exact sequence
    \begin{equation*}
    0\rightarrow S_v\rightarrow (\oplus_{u\rightarrow v\in Q_1}S_u^*)^\vee
    \rightarrow Q_v\rightarrow 0.
\end{equation*}
\end{dfn}
 
For each bundle $E$ of rank $r$ over the variety $\mc Z$, we define its equivariant Chern polynomial to be 
\begin{equation}\label{eqn:chernpolynomial}
    c_t^S(E)=\sum_{i=0}^{r}(-1)^ic_i^S(E)t^{r-i}
\end{equation}
where $t$ is a formal variable and $c_i^S(E)$ are the equivariant $i$-th Chern classes. 
Consider a $\mathbb Q$-algebra $QH^*_S(\mc Z)[t]:=QH^*_S(\mc Z)\otimes \mathbb Q[t]$ which we still call quantum cohomology ring by abuse of terminology. 
Then the Chern polynomials of the quotient bundles and tautological bundles $c_t^S(S_v)$ and $c_t^S(Q_v)$ lie in $QH^*_S(\mc Z)[t]$. 
In genus zero, the Seiberg Duality conjecture above implies the following isomorphism, which can be viewed as the Seiberg Duality conjecture at the quantum cohomology level. 
\begin{conjecture}\label{conj:seiberginqcoh}
The following two equivariant small quantum cohomology rings are isomorphic
    \begin{equation*}
  \Phi : QH^*_S(\mc Z)[t]\rightarrow QH^*_S(\mc Z')[t]
\end{equation*}
up to change of K\"ahler variables $q_v'=q_v^{-1}$ 
and 
    $q_u'=
        q_uq_v^{[\op{sign}(\sigma_v)b_{vu}]_+} $.
In particular, if the two QPs are related by a quiver mutation at the gauge node $v$, then $\Phi(c_t^S(Q_v))=c_t^S((S_v')^\vee)$ where $S_v'$ is the tautological bundle over $\mc Z'$ at the gauge node $v$.
\end{conjecture}

\subsection{Cluster algebra conjecture}\label{sec:clusteralgebraconj} 
In the quantum cohomology ring $QH^*_S(\mc Z)[t]$,  we establish the following relations, referred to as the quantum cohomological cluster exchange relations.
\begin{conjecture}\label{conj:clusterexchagenrelation}
For each node $v\in Q_0\backslash Q_f$,
if $N_f(v)>N_a(v)$, we have
\begin{equation}\label{eqn:conjNf>Na}
        c_t^S(S_v)*c_t^S(Q_v)= \prod_{v\rightarrow u}c_t^S(S_u^*)+(-1)^{r_v+N_a(v)}q_v\prod_{w\rightarrow v\in Q_1}c_t^S(S_w^*);
    \end{equation}
and if $N_f(v)<N_a(v)$, we have
\begin{equation}\label{eqn:conjNf<Na}
        c_t^S(S_v)*c_t^S(Q_v)= \prod_{u\rightarrow v}c_t^S((S_u^*)^\vee)+(-1)^{r_v+N_f(v)}q_v^{-1}\prod_{v\rightarrow w\in Q_1}c_t^S((S_w^*)^\vee).
    \end{equation}
\end{conjecture}
We have proved this conjecture for $A$-type quivers which can be found in the last section. Because of this observation, we believe that there is a relation between the cluster algebra and the quantum cohomology ring. 

Given a cluster quiver $\mathbf Q=(Q_0,Q_1,W)$, 
one can construct a cluster algebra which is called a geometric cluster algebra as in Section \ref{sec:clusteralgebras}. 
Denote the corresponding cluster algebra by $\mathscr A_{\mathbf Q}$. 
Let $\mathbf x=(x_u)_{u\in Q_0}$ be the initial cluster variables.

We introduce another set of formal variables $\{\xi_i\}_{i\in Q_0}$ and
let $q_v=(-1)^{N_f(v)-N_a(v)}\prod_{u\in Q_0}\xi_u^{b_{uv}}$.
Assume that Conjecture \ref{conj:clusterexchagenrelation} holds. 
Combining the Seiberg duality conjecture and the relation of tautological bundles and quotient bundles in Conjecture \ref{conj:seiberginqcoh}, we propose the following construction for a $\mathbb Q$ algebra homomorphism from the cluster algebra to the equivariant quantum cohomology ring of the associated variety of decorated QP $QH^*_S(\mc Z)[t] $.
\begin{conjecture}\label{conj:clusteralgconj}
    There is a well-defined $\mathbb Q$-algebra homomorphism 
    \begin{equation*}
        \psi:\mathscr A_{\mathbf Q}\rightarrow QH^*_S(\mc Z)[t]
    \end{equation*}
    such that 
    \begin{enumerate}
        \item $\psi(x_v)=(-1)^{r_v}\xi_vc_t^S(S_v)$
        \item  
        \begin{equation*}
            \psi(x_v')=\begin{cases}
            (-1)^{N_f(v)-r_v}\xi_v^{-1}\prod_{v\rightarrow u}\xi_uc_t^S(Q_v) \text{ when } N_f(v)>N_a(v)\\
            (-1)^{N_a(v)-r_v}\xi_v^{-1}\prod_{w\rightarrow v}\xi_wc_t^S(Q_v) \text{ when } N_f(v)<N_a(v)
            \end{cases}
        \end{equation*}
        \item 
        For other non-initial cluster variables $\tilde x_v$, 
        assume that $\tilde x_v$ lies in a seed $(\tilde {\mathbf Q},  \tilde {\mathbf{x}})$ that can be obtained by a sequence of quiver mutations from the initial seed $(\mathbf Q,\mathbf x)\xhookrightarrow{\mu_{v_i}\cdots\mu_{v_1}} ( \tilde {\mathbf Q},\tilde{\mathbf x})$. 
        The conjecture \ref{conj:seiberginqcoh} asserts that the quantum cohomology rings of varieties of these QPs are isomorphic. 
        Let $\tilde\Phi: QH^*_S(\mc Z) \rightarrow QH^*_S(\tilde{\mc Z})$ be the composite of these isomorphisms.
        Let $\{\tilde S_u\}_{u\in Q_0}$ be tautological bundles over the variety $\tilde{\mc Z}$ of the decorated QP $(\tilde{\mathbf Q},\tilde{\mathbf r})$. 
        Then the image of $\tilde x_v$ is the pullback of $c_t^S(\tilde{S}_v)$ if $\tilde \sigma_v>0$ and the pullback of $c_t^S(\tilde{S}_v^\vee)$ if $\tilde \sigma_v<0$ scaled by $\xi_v$, and these scalars can be determined recursively to make $\psi$ a ring homomorphism. 
    \end{enumerate}
    The image of $\tilde x_v$ for noninitial cluster variables in the last item does not depend on the seed we choose and does not depend on the sequence of quiver mutations we obtain. 
\end{conjecture}
One can check that when Conjecture \ref{conj:clusterexchagenrelation} holds, the map $\psi$ defined above preserves all cluster exchange relations.

\section{Proof for all-genus Seiberg duality Conjecture }\label{sec:proofofseibergduality}
We will utilize the localization theorem to prove that the Seiberg duality conjecture of all genus \ref{conj:seibergdualityconj} holds for all decorated QPs $(\mathbf Q,\mathbf r)\in \Omega^A_n $ that are mutation equivalent to $A_n$. 
Throughout this section, we will let $(\mathbf Q,\mathbf r)$ and $(\mathbf Q',\mathbf r')$ be two decorated QPs that are related by a quiver mutation $\mu_v$ in $\Omega^A_n$ and let $\mc Z$ and $\mc Z'$ be the varieties of the two QPs. 

\begin{thm}\label{thm:GWinvarmu}
   There exists a $H^*_S(pt)$-linear map $\phi: H^*_S(\mc Z)\rightarrow H^*_S(\mc Z')$ and a morphism $\phi_*: H_2(\mc Z, \Z)\rightarrow H_2(\mc Z', \Z)$  such that, 
 for arbitrary $\gamma_1, \ldots,\gamma_m\in H^*_S(\mc Z)$,  $a_i\in \mathbb Z_{\geq 0}$ and $\beta\in H_2(\mc Z)$, we have 
\begin{equation*}
    \langle \tau_{a_1}\gamma_1,\ldots,\tau_{a_m}\gamma_m \rangle_{g,m,\beta}^{\mc Z,S}=
    \langle \tau_{a_1} \phi(\gamma_1),\ldots,\tau_{a_m}\phi(\gamma_m) \rangle_{g,m,\phi_*(\beta)}^{\mc Z',S}. 
\end{equation*}
See Proposition \ref{prop:equicohmu} and Proposition \ref{prop:transformeff} for the constructions of maps $\phi$ and $\phi_*$.
\end{thm}
As discussed in Section \ref{sec:seibergduality}, one quick corollary of the all genus Seiberg duality conjecture is that the equivariant quantum cohomology rings of $\mc Z$ and $\mc Z'$ are isomorphic. 
\begin{cor}\label{cor:qcohisomorphic}
    The small quantum cohomology rings $QH^*_S(\mc Z):=H^*_S(\mc Z)\otimes \mathbb Q[q]$ and $QH^*_S(\mc Z'):=H^*_S(\mc Z')\otimes \mathbb Q[q']$ are isomorphic under the change of K\"ahler variables
    \begin{align*}
        &q_v'=q_v^{-1}\,,\nonumber\\
        &q_u'=q_v
        \prod_{u\in Q_0\backslash Q_f}(q_u)^{[\op{sign}(\sigma_v)b_{vu}]_+}\,.
    \end{align*}
\end{cor}
\begin{proof}
    The isomorphism between the quantum cohomology rings is due to Theorem \ref{thm:GWinvarmu} and the transformation of K\"ahler variables is a result of the relation of curve classes in Proposition \ref{prop:transformeff}.
\end{proof}

We will study the ingredients in the localization Theorem \ref{thm:localization} for $\mc Z$ and $\mc Z'$ in detail and then compare Gromov-Witten invariants of $\mc Z$ and $\mc Z'$.
\subsection{Localization theorem}
Localization is a very efficient tool for computing Gromov-Witten invariants of a GKM space, and we review it in this section following \cite{GW:GPequiv,liu2011localization}. 
A variety $\mc X$ with torus action $S$ is a GKM space if the set of zero-dimensional $S$-orbits is of zero dimension and the set of one-dimensional $S$-orbits is of one dimension.
Under this assumption, the torus fixed locus $\mc X^S$ is discrete and the closure of any irreducible 1-dimensional $S$-orbit in $\mc X$ is isomorphic to $\mathbb{P}^1$, with $0$ and $\infty$ being $S$-fixed. 

We define the 1 skeleton of $\mc X$ as the union of all closures of one-dimensional orbits, which we denote by $\mc X^1$. 
\subsubsection{Torus fixed graphs}\label{lem:loctorusfixedgraph}
We assign a graph (a compact, connected 1-dimensional CW complex) $\Upsilon$ to the 1-skeleton $\mc X^1$.
The vertices $\sigma$ in $\Upsilon$ are $S$-fixed points $p_{\sigma}\in \mc X^S$, and edges ${\tau}$ in $\Upsilon$ are irreducible 1-dimensional $S$-orbits $\ell_\tau$. 
A vertex $\sigma$ is contained in an edge $\tau$ if and only if the fixed point $p_\sigma$ is contained in the closure of the 1-dimensional $S$-orbit $\ell_\tau$.
A flag in $\Upsilon$ is defined as a pair $(\tau,\sigma)\in \mc X^S\times \mc X^1$ where the vertex $\sigma$ is contained in the edge $\tau$. 
We denote the sets of vertices, edges, and flags of $\Upsilon$ by $V(\Upsilon)$, $E(\Upsilon)$, and $F(\Upsilon)$.

For a flag $(\tau, \sigma)\in F(\Upsilon)$, we denote by $\mathbf{w}(\tau, \sigma)$  the weight of $S$-action on the tangent line $T_{p_\sigma}\ell_\tau$,
\begin{align}\label{eqn:weights}
	\mathbf{w}(\tau, \sigma)=c_1^S(T_{p_\sigma}\ell_\tau)\in H^{2}_S(pt).
\end{align}
This map gives rise to a map $\mathbf{w}: F(\Upsilon)\rightarrow H^{2}_S(pt)$ satisfying the following properties. Let $\sigma, \sigma`$ be two vertices connected by an edge $\tau$. 
\begin{itemize}
	\item We have equality $\mathbf{w}(\tau, \sigma)+\mathbf{w}(\tau, \sigma`)=0$.
	\item Let $E_\sigma=\{\tau\in E(\Upsilon)| \sigma\in \tau\}=\{\tau_1, \tau_2, \ldots, \tau_r=\tau\}$ be the set of edges that contain the vertex $\sigma$. Then for any $\tau_i\in E_\sigma$, there exists a unique $\tau_i`\in E_{\sigma'}$ and $a_i\in \Z$ such that
     \begin{equation}\label{eqn:ai}
         \mathbf{w}(\tau_i`, \sigma`)=\mathbf{w}(\tau_i, \sigma)-a_i\mathbf{w}(\tau, \sigma).
     \end{equation}
	 In particular, $\tau_r`=\tau_r=\tau$ and $a_r=2$.
\end{itemize}
\subsubsection{Decorated graphs} 
In this section, we deal with torus fixed points in the moduli stack of stable maps $\overline{\mc M}_{g,n}(\mc X,\beta)$.
Given a stable map $f:(C,x_1,\cdots, x_n)\rightarrow \mc X$ such that $[f:(C,x_1,\cdots, x_n)\rightarrow \mc X]\in \overline{\mc M}_{g,n}(\mc X,\beta)^S$, we will associate with it a decorated graph $\overrightarrow \Gamma$. 
We first give a formal definition.
\begin{dfn}\label{decorated graph}
	A decorated graph $\overrightarrow{\Gamma}=(\Gamma, \overrightarrow{f}, \overrightarrow{d}, \overrightarrow{g}, \overrightarrow{s})$ for $n$-pointed, genus $g$, degree $\beta$ stable maps to $\mc X$ consists of the following data.
	\begin{enumerate}
		\item $\Gamma$ is a compact, connected 1 dimensional CW complex. We denote the set of vertices (resp. edges) in $\Gamma$ by $V(\Gamma)$ (resp. $E(\Gamma)$, and the set of  flags by 
		 $F(\Gamma)=\{(e, v)\in E(\Gamma)\times V(\Gamma)|v\in e\}$.
		 \item The label map $\overrightarrow{f}: 
         V(\Gamma)\cup E(\Gamma)\rightarrow V(\Upsilon) \cup E(\Upsilon)$ sends a vertex $v\in V(\Gamma)$ to a vertex $\sigma_v\in V(\Upsilon)$, 
         and sends an edge $e\in E(\Gamma)$ to an edge $\tau_e\in E(\Upsilon)$. 
         Moreover, if $(e, v)$ is a flag of $\Gamma$, then $(\tau_e, \sigma_v)$ is also a flag of $\Upsilon$.
		 \item The degree map $\overrightarrow{d}: E(\Gamma)\rightarrow\Z_{\geq 0}$ sends an edge $e\in E(\Gamma)$ to a positive integer $d_e$.
		 \item The genus map $\overrightarrow{g}: V(\Gamma)\rightarrow \Z_{\geq 0}$ sends a vertex $v\in V(\Gamma)$ to a nonnegative integer $g_v$.
		 \item The marking map $\overrightarrow{s}: \{1, 2, \ldots, n\}\rightarrow V(\Gamma)$ is defined if $n>0$.
		\end{enumerate}
		The above maps satisfy the following two constraints:
    \begin{itemize}
    \item  $\sum_{v\in V(\Gamma)}g_v+|E(\Gamma)|-|V(\Gamma)|+1=g$.
    \item $\sum_{e\in E(\Gamma)}d_e[\ell_{\tau_e}]=\beta$.
		\end{itemize}
		
		Let $G_{g, n}(\mc X, \beta)$ be the set of all decorated graphs $\overrightarrow{\Gamma}=(\Gamma, \overrightarrow{f}, \overrightarrow{d}, \overrightarrow{g}, \overrightarrow{s})$ that satisfy the above constraints.
\end{dfn}

For later convenience, we introduce some definitions.
\begin{dfn}
	Given a vertex $v\in V(\Gamma)$, we define 
	$$E_v=\{e\in E(\Gamma)|(e, v)\in F(\Gamma)\},$$
	the set of edges emanating from $v$ and define $S_v=\overrightarrow{s}^{-1}(v)\subset \{1, \ldots, n\}$. 
    The valency of $v$ is defined by $val(v)=|E_v|+|S_v|$. 
\end{dfn}

Let $f:(C, x_1,\cdots, x_n)\in \overline{\mc M}_{g,n}(\mc X,\beta)^S$ be a stable map fixed by the $S$-action. Then for an irreducible component $C'\subset C$, $f\rvert_{C'}$ either contracts to a point $p_\sigma\in \mc X^S$ or $C'\iso \mathbb P^1$ and $f(C')=\ell_{\tau}$. 

A decorated graph $\overrightarrow{\Gamma}$ associated with $f: (C, x_1, \ldots, x_n)\rightarrow \mc X$ can be defined as follows.
\begin{enumerate}
	\item We assign a vertex $v$ to each connected component of $f^{-1}(\mc X^S)$, and denote the corresponding component by $C_v$.
     Assume that $f(C_v)=\{p_{\sigma}\}$, then the label map $\overrightarrow f(v)=\sigma$, and the genus $\overrightarrow{g}(v)=g_v$ is the arithmetic genus of the component $C_v$. 
     The marking map $\overrightarrow s$ sends the marking points on the component $C_v$ to $v\in V(\Gamma)$.
     
	 \item We assign an edge $e$ to each connected component $f^{-1}(\mc X^1\backslash \mc X^S)$. 
  Let $C_e\iso \mathbb P^1$ be the closure of one connected component in $f^{-1}(\mc X^1\backslash \mc X^S)$ and assume that $f(C_e)=\ell_\tau$. We define $\overrightarrow f(e)=\tau$.
  \item The set of flags is defined by 
  \[ F(\Gamma)=\{ (e,v)\in E(\Gamma)\times V(\Gamma)|\,C_e\cap C_v\neq \emptyset\}.\]
\end{enumerate}
Then we get a decorated graph $\overrightarrow{\Gamma}=(\Gamma, \overrightarrow{f}, \overrightarrow{d}, \overrightarrow{g}, \overrightarrow{s})\in G_{g,\beta}(\mc X, \beta)$. 

The construction above gives a map from $\overline{\mathcal{M}}_{g, n}(\mc X, \beta)^S$ to the discrete set $\overrightarrow{\Gamma}\in G_{g, n}(\mc X, \beta)$. 
Let $\mathcal{F}_{\overrightarrow{\Gamma}}\subset \overline{\mathcal{M}}_{g, n}(\mc X, \beta)^S$ denote the preimage of $\overrightarrow{\Gamma}$. 
Then 
$$\overline{\mathcal{M}}_{g, n}(\mc X, \beta)^S=\displaystyle\sqcup_{\overrightarrow{\Gamma}\in G_{g, n}(\mc X, \beta)}\mathcal{F}_{\overrightarrow{\Gamma}}$$
where the right hand side is a disjoint union of connected components. 
\subsubsection{Localization formula on Gromov-Witten theory}
Define 
\begin{align*}
	&\mathbf{w}(\sigma)=\prod_{(\tau, \sigma)\in F(\Upsilon)}\mathbf{w}(\tau, \sigma)\in H^{\dim {\mc X}}_S(pt),\\
	&\mathbf{h}(\sigma, g)=\prod_{(\tau, \sigma)\in F(\Upsilon)}\frac{\Lambda^{\vee}_g(\mathbf{w}(\tau, \sigma))}{\mathbf{w}(\tau, \sigma)}\in H^{2r(g-1)}_S(\overline{\mathcal{M}}_{g, n}),
\end{align*}
	where $\Lambda^{\vee}_g(u)=\Sigma_{i=0}^g(-1)^i\lambda_iu^{g-i}$ and $\lambda_i$ are $\lambda$-classes, see the definition in \cite[Sec 3.1]{liu2011localization}.
Furthermore, we define 
\begin{align*}
	\mathbf{h}(\tau, d)=\frac{(-1)^dd^{2d}}{(d!)^2\mathbf{w}(\tau, \sigma)^{2d}}\prod_{i=1}^{r-1}b(\frac{\mathbf{w}(\tau, \sigma)}{d},\mathbf{w}(\tau_i, \sigma), da_i)
\end{align*}
where  $a_i$ are defined in \eqref{eqn:ai}, and
\begin{align*}
	b(u, w, a)=    
   \begin{cases}
      \prod_{j=0}^a(w-ju)^{-1}  &a\in \mathbb Z, a\geq 0, \\ 
       \prod_{j=1}^{-a-1}(w+ju) & a\in \mathbb Z, a< 0.
    \end{cases}
\end{align*}
\begin{thm}(\cite[Theorem 73]{liu2011localization}, \cite{GW:GPequiv})\label{thm:localization}
\begin{align*}
		&\langle \tau_{a_1}(\gamma_1)\cdots \tau_{a_n}(\gamma_n)\rangle^{\mc X,S}_{g, n,\beta}\\
	=&\sum_{\overrightarrow{\Gamma}\in G_{g, n}(\mc X, \beta)}\frac{1}{|Aut(\overrightarrow{\Gamma})|}	\prod_{e\in E(\Gamma)}\frac{\mathbf{h}(\tau_e, d_e)}{d_e}\prod_{v\in V(\Gamma)}\left(\mathbf{w}(\sigma_v))^{val(v)}\prod_{i\in S_v}i^*_{\sigma_v}\gamma_i\right)\\
	&\cdot \prod_{v\in V(\Gamma)}\int_{\overline{\mathcal{M}}_{g, E_v\cup S_v}}\frac{\mathbf{h}(\sigma_v, g_v)\prod_{i\in S_v}\psi_{i}^{a_i}}{\prod_{e\in E_v}(\mathbf{w}(\tau_e, \sigma_v)/d_e-\psi_{(e, v)})},
\end{align*}
where we have the following convention for unstable curves:
\begin{align*}
	&\int_{\overline{\mathcal{M}}_{0, 1}}\frac{1}{w_1-\psi_2}=w_1, \quad \int_{\overline{\mathcal{M}}_{0, 2}}\frac{1}{(w_1-\psi_1)(w_2-\psi_2)}=\frac{1}{w_1+w_2},\\
&\int_{\overline{\mathcal{M}}_{0, 2}}\frac{\psi_2^a}{w_1-\psi_1}=(-w_1)^a, a\in \Z_{\geq 0}.
\end{align*}
\end{thm}

\subsection{Torus fixed points}\label{sec:torusfixedpoints}
\iffalse
We first consider a simple case to illustrate the idea.  When $n=1$, we get  the Grassmannian $Gr(N_1,N_2)$. 
It has a torus action $S=(\mathbb C^*)^{N_2}$, 
\begin{equation*}
    t\cdot A=At^{-1}, A\in Gr(N_1,N_2),\, t\in (\mathbb C^*)^{N_2}\,.
\end{equation*}
We can find matrix representations in $G$-orbits of torus fixed points in reduced row echelon forms whose entries all vanish except for pivots. 
Such matrices are uniquely determined by the columns where pivots lie in, and therefore we use those column numbers $c_1<\ldots<c_{N_1}$ to represent such matrices. 

Similarly, a matrix of size $N_2\times (N_2-N_1)$ can represent a point in the dual Grassmannian $Gr(N_2-N_1,N_2)$, and the torus fixed points can be represented by matrices in reduced column echelon forms whose entries vanish except for pivots. 
We use the row numbers $l_1<\ldots<l_{N_2-N_1}$ of pivots to denote such points. 
\fi

We first consider the flag variety $Fl:=Fl(N_1,\ldots,N_{n+1})$ in Example \ref{ex:An} which admits a torus action  
$S=(\mathbb C^*)^{N_{n+1}}$,
\begin{equation*}
    t\cdot(A_1,\ldots, A_n)=(A_1,\ldots,A_nt^{-1}), t\in S.\nonumber
\end{equation*}
The torus fixed points can be represented by matrices $(A_1,\ldots,A_n)$ such that each $A_i$ is in reduced row echelon form whose entries vanish except for pivots. 
Hence, torus fixed points in $Fl$ can be described as follows.
\begin{equation}\label{eqn:Sfixedpointsinflag}
   (Fl)^S= \{\vec I=(I_1,\cdots,I_n)| I_1\subset I_2\subset \cdots\subset [N_{n+1}], \abs{I_i}=N_i, i=1,\cdots, n \}\,,
\end{equation}
where $[N_{n+1}]=\{1,2,\cdots,N_{n+1}\}$. 
The cardinality of the set is 
\begin{equation}\label{eqn:cardinalitySfixedpoint}
    \abs{(Fl)^S}= \prod_{i=1}^{n}\binom{N_{i+1}}{N_i}.
\end{equation}

For an arbitrary decorated QP $(\mathbf Q=(Q_0,Q_1,W),\mathbf r)$ in $\Omega^A_n$, let $\mc Z=\{dW=0\}\sslash_\theta G$ be the associated variety described in Lemma \ref{lem:phasevargeneralquiver}. 
Denote the frozen node by $v_f$ and $r_{v_f}=N_{n+1}$. 
By the description of quivers in Section \ref{sec:clusteralgebras}, 
there are at most two maximal nodes that are smaller than $v_f$, which we denote by $v_1$ and $v_2$. Assume that $v_1\xrightarrow[]{A_{v_1}}v_f\xrightarrow[]{A_{v_2}} v_2\xrightarrow[]{B_{v_f}} v_1$ form a 3-cycle. 
The torus $S=(\mathbb C^*)^{N_{n+1}}$ acts on a representative of a point as $A_{v_1}t^{-1},tA_{v_2}$ and trivially on other matrices, for $t\in S$. 
\iffalse
We will adopt the description in Lemma \ref{lem:phasevargeneralquiver} 
for $Z^{ss}(G)$. The variety $\mc Z$ can be embedded into another quiver variety $\mc X$, which is obtained as follows. 

For any vertex $v\in Q_0$, if there are two maximal nodes $v_1,v_2$ that are smaller than $v$ and we know that in $\mathbf Q$ there is a 3-cycle $v_1\rightarrow v\rightarrow v_2\rightarrow v_1$, then delete the arrow $v_2\rightarrow v_1$. We thus obtain a new quiver denoted by $\tilde{\mathbf Q}=(Q_0,\tilde Q_1)$ without cycles and potential. The embedding quiver variety $\mc X$ is the variety defined by such a quiver in the phase given in Lemma \ref{lem:phasevargeneralquiver}. 
\fi
\begin{lem}\label{lem:torusfixedpointgeneral}
    The torus fixed locus $\mc Z^S$ is a set $\{\vec I=(I_{v})_{v\in Q_0\backslash v_f}\}$ satisfying the following conditions. 
    \begin{enumerate}
        \item The cardinality of the set $I_{v}$ is equal to the integer $r_v$.
        \item If $v\prec w$, then $I_v\subset I_w$. 
        \item If there are two maximal nodes $v_1\prec v, v_2\prec v$, then $I_{v_1}\cap I_{v_2}=\emptyset$. 
    \end{enumerate}
\end{lem}
\begin{proof}
One can check that in matrices formulae, the torus fixed points can be described as follows: 
        \begin{enumerate}
        \item If a node $v$ has only one maximal node $w$ smaller than it, the associated matrix of the arrow is non-degenerate and in reduced row/column echelon form whose non-pivots entries vanish.  
        \item If a node $v$ has two maximal nodes $v_1,v_2$ smaller than it, 
       we assume they form a 3-cycle $v_1\xrightarrow{A_{v_1}} v\xrightarrow{A_{v_2}} v_2\xrightarrow{B_v} v_1$. 
        Then the two matrices $A_{v_1}$  and $A_{v_2}$ are in reduced row and column echelon forms with non-pivots entries zero and $B=0,\,A_{v_1}A_{v_2}=0$. 
    \end{enumerate}
We use the number of rows/columns of pivots of each non-degenerate matrix in reduced column/row echelon form to represent the matrix and get a set of sets satisfying the conditions in the Lemma. 
\end{proof}

Let $(\mathbf Q=(Q_0,Q_1,W),\mathbf r)$ and $(\mathbf Q'=(Q_0',Q_1',W'),\mathbf r')$ be two decorated QPs related by a quiver mutation at a gauge node $v$. 
According to the description for the local picture around a node at the end of Section \ref{sec:quivervarietieswithmu},
a gauge node $v$ has at most four adjacent nodes which we denote by $v_1,v_2,v_3,v_4$. 
We assume the arrows among them are $v\rightarrow v_1\rightarrow v_2\rightarrow v$, $v_3\rightarrow v\rightarrow v_4\rightarrow v_3$. 
Denote the integers assigned to those nodes by $r_{v_i}, i=1,\cdots 4$. 
We assume that the QP in Figure \ref{fig:localbehavior} (a) is $(\mathbf Q,\mathbf r)$ and the one in \ref{fig:localbehavior} (b) is $(\mathbf Q',\mathbf r')$. 
Let $\mc Z$ and $\mc Z'$ be their associated varieties. 
\begin{lem}\label{lem:Sfixedpointsmu}
    There is a natural bijection
    \[\varphi: \mc Z^S\rightarrow (\mc Z')^S\] 
   defined as follows. 
    Let $\vec I=(I_u)_{u\in Q_0\backslash v_f}$ be an arbitrary point in $\mc Z^S$.  
    \begin{enumerate}
        \item If $N_f(v)>N_a(v)$, the map $\varphi$ maps $I_v$ to $(I_{v_1}\sqcup I_{v_4})\backslash I_{v}$ and keeps all other $I_u$. 
        \item 
        If $N_f(v)<N_a(v)$, the map $\varphi$ maps $I_{v}$ to $(I_{v_2}\sqcup I_{v_3})\backslash I_v$ and keeps all other $I_u$.
    \end{enumerate}
\end{lem}
\begin{proof} 
The map $\varphi$ is naturally bijective if we can prove that it is a well-defined map. 
One can find that the two situations are inverse maps to each other.  For the node $v$, if $N_f(v)>N_a(v)$, then after quiver mutation, at the node $v$ in $(\mathbf Q', \mathbf r')$, the outgoing is obviously smaller than the incoming. 

To prove that the map $\varphi$ is well-defined, we need to prove that the images lie in $(\mc Z')^S$.
We will prove the situation when the node $v$ has four adjacent nodes, and one needs to note that other cases can be reduced from this. 

When $N_f(v)>N_a(v)$ for $(\mathbf Q, \mathbf r)$, then one of $v_1$ and $v_4$ is greater than $v$ and we assume $v_4\prec v\prec v_1$. 
Let $\vec I=(I_u)\in \mc Z^S$.
Denote by $\vec I'=(I_u')_{u\in Q_0}=\varphi(\vec I)$. 
The cardinality of $I_v'$ is $\abs{(I_{v_1}\sqcup I_{v_4})\backslash I_v}=\abs{(I_{v_1}\backslash I_v)\cup I_{v_4}}=r_{v_1}-r_{v}+r_{v_4}$ which is $r_v'$ since $I_{v_4}\subset I_{v}\subset I_{v_1}$. 
Hence the item $(1)$ of Lemma \ref{lem:torusfixedpointgeneral} is true. 

One can easily check that $I_v'\subset I_{v_1}'=I_{v_1}$ and $I_{v_2}'\subset I_{v}'$ since $I_{v_2}'\cap I_{v}=\emptyset$. 
For the QP in Figure \ref{fig:localbehavior} (b), nodes $v_3$ and $v$ are two maximal nodes that are smaller than $v_1$, and nodes $v_1$ and $v_4$ are two maximal nodes that are smaller than $v$.
One can find $I_{v_3}'\cap I_v'=\emptyset$ since $I_{v_3}'=I_{v_3}\subset I_v\subset I_{v_1}$, and $I_{v_2}'\cap I_{v_4}'=\emptyset$ since $I_{v_2}\cap I_v=\emptyset$ and $I_{v_4}\subset I_{v}$, so the Lemma \ref{lem:torusfixedpointgeneral} (2)(3) are checked.

The situation when $N_f(v)<N_a(v)$ is similar and we omit it. 
\iffalse
\begin{figure}[H]
    \centering
    \includegraphics[width=3.6in]{local.png}
    \caption{The two quivers are related by a quiver mutation at the center node $v$, with $r_v'=r_{v_1}+r_{v_2}-r_v$.}
    \label{fig:local}
\end{figure}
\fi
\end{proof}
Combine the Equation \eqref{eqn:cardinalitySfixedpoint} and the Lemma \ref{lem:Sfixedpointsmu}, and we obtain the following Corollary. 
\begin{cor}\label{cor:torusfixedpointscard}
    For an arbitrary decorated QP $(\mathbf Q=(Q_0,Q_1,W),\mathbf r)\in \Omega^A_n$, let $\mc Z$ be the critical locus of $W$.
    The cardinality of the torus fixed locus is $$\abs{\mc Z^S}= \abs{(Fl)^S}= \prod_{i=1}^{n}\binom{N_{i+1}}{N_i}.$$
\end{cor}
According to the Atiyah-Bott localization theorem  \cite{ATIYAH19841}, 
\begin{equation}\label{eqn:atiyahbottlocalization}
    H_S^*(\mc Z)\iso \oplus_{p\in \mc Z^S}H^*_S(p)\,,
\end{equation}
we get the following Seiberg duality on the equivariant cohomology level. 

\begin{prop}\label{prop:equicohmu}
Suppose two varieties $\mc Z$ and $\mc Z'$ are related by a quiver mutation at the gauge node $v$. 
There is a group isomorphism  
\begin{equation*}
    \phi: H^*_S(\mc Z)\xrightarrow{\sim}H^*_S(\mc Z')\,. 
\end{equation*}
such that, for arbitrary $\gamma\in H^*_S(\mc Z)$, $P\in \mc Z^S$
\begin{equation}\label{eqn:gammaresP}
 \gamma\big{|}_{P}= \phi(\gamma)\big{|}_{\varphi(P)},
\end{equation}
In particular, $\phi$ maps the Chern classes  $c_k^S(Q_v)$ over $\mc Z$ to those of the dual of the tautological bundle $c_k^S(S_v'^\vee)$ over $\mc Z'$, $c_k^S(S_v^\vee)$ to $c_k^S(Q_v')$ over $\mc Z'$, 
 Chern classes $c_k^S(S_u^\vee)$ to $c_k^S(S_u')$ if $\op{sign}(\sigma_u')\op{sign}(\sigma_u)=-1$, and  preserves Chern classes $c_k^S(S_w)$ of other tautological bundles.
\end{prop}
\begin{proof}
   It is obvious that the two equivariant cohomology groups are isomorphic by Atiyah-Bott localization theorem \eqref{eqn:atiyahbottlocalization} and the Corollary \ref{cor:torusfixedpointscard}. 
   One can find the mutation $\mu_v$ maps $c_k^S(S_u^\vee)$ to $c_k^S(S_u')$ if $\sigma_u$ changes its sign according to the Definition \ref{dfn:tautbundle}.

    In order to prove that $\phi$ maps Chern classes $c_k^S(Q_v)$ to those of $(S_v')^\vee$, we only have to prove that for each point $p_{\vec I}\in \mc Z^S$, the following equation holds, 
    \begin{equation*}
        c_k^S(Q_v)|_{ p_{\vec  I}}=(-1)^kc_k^S(S_v')|_{\varphi(p_{\vec I})}.
    \end{equation*}
    Without loss of generality, we assume that there are four nodes adjacent to the node $v$ as shown in Figure \ref{fig:localbehavior} (a)  with $r_{v_1}+r_{v_4}>r_{v_2}+ r_{v_3}$ and $v_3\prec v\prec v_1$. 
    We also assume that $\sigma_{v_1}>0$ and $\sigma_{v_4}<0$.
    Then we have a short exact sequence
    \begin{equation*}
        0\rightarrow S_v\rightarrow S_{v_1}\oplus S_{v_4}^\vee\rightarrow Q_v \rightarrow 0,
    \end{equation*} 
    and a relation in equivariant cohomology ring $H_S^*(\mc Z)$
    \begin{equation*}
        c_t^S(S_v)c_t^S(Q_v)= c_t^S(S_{v_1})c_t^S(S_{v_4}^\vee).
    \end{equation*}
    Without loss of generality, we assume that $I_{v_4}= 
    \{1<2<\cdots<r_{v_4}\} \subset I_{v}=
    \{1<2<\cdots<r_v\}
    \subset I_{v_1}=
    \{1<2<\cdots<r_{v_1}\}$. 
    Then $c_t^S(S_{v_i}^*)|_{p_{\vec I}}=\prod_{j=1}^{r_i}(t-\lambda_{j})$ for $i=1,4$, 
    and $c_t^S(S_v)|_{p_{\vec I}}=
    \prod_{j=1}^{r_v}(t-\lambda_j)$. 
    One can easily compute  
    $c_t^S(Q_v)|_ {p_{\vec  I}}=\prod_{j=r_v+1}^{r_{v_1}}(t-\lambda_{j})\prod_{j=1}^{r_{v_4}}(t-\lambda_{j})$. 
    On the other hand, after the quiver mutation $\mu_v$, we let $p_{\vec I'}=\varphi(p_{\vec I})$ with $I_{u}'=I_{u}$ for $u\neq v$ and $I_{v}'=(I_{v_1}\backslash I_v)\cup I_{v_{v_4}}=\{1<2<\cdots r_{v_4}<r_{v}+1<\cdots r_{v_1}\}$. 
    One can easily find $c_t^S(S_v')|_{p_{\vec  I'}}= \prod_{j=r_v+1}^{r_{v_1}}(t+\lambda_{j})\prod_{j=1}^{r_{v_4}}(t+\lambda_{j})$. 
    Hence we have proved that $c_k^S(Q_v)|_{p_{\vec I}}=(-1)^{k}c_k^S(S_v')|_{p_{\vec I'}}$. 
    One can similarly prove that $c_k^S(S_v)=(-1)^{k}c_k^S(Q_v')$ and we omit its proof. 
\end{proof}

For an arbitrary decorated QP $\mathbf Q\in \Omega^A_n$ and the corresponding $\mc Z$, we choose a basis $\{\beta_u\}_{u\in Q_0\backslash Q_f}$ of $H_2(\mc Z,\mathbb Z)$ such that 
\begin{equation}\label{eqn:basisofeff}
    \int_{\beta_u}c_1(S_w)=-\op{sign}(\sigma_w)\delta_{uw}\,.
\end{equation}
Then we have 
$H_2(\mc Z,\mathbb Z)\iso\mathbb Z^{n}$\,.
For two varieties $\mc Z$ and $\mc Z'$ related by a quiver mutation at the gauge node $v$, we present a transformation of their effective curve classes inspired by \eqref{eqn:homologicalchange}.
\begin{prop}\label{prop:transformeff}
   Let $\beta=(\beta_u)_{u\in Q_0\backslash Q_f}$ and $\beta'=(\beta'_u)_{u\in Q_0\backslash Q_f}$ be the bases of $H_2(\mc Z,\mathbb Z)$ and $H_2(\mc Z',\mathbb Z)$, respectively, chosen as \eqref{eqn:basisofeff}.
   Then their transformation is given as follows. 
    If $\sigma_v>0$, which means $N_f(v)>N_a(v)$, then 
\begin{equation}\label{eqn:transforeff1}
    \phi_*(\beta_u)=
        \begin{cases}
            -\beta_u', &\text{ for }u= v\\
           \beta_u'+\beta_v' &\text { for }v\rightarrow u\\
           \beta_u' &\text { otherwise }\,.
        \end{cases}
    \end{equation}
    If $\sigma_v<0$, we have 
    \begin{equation}\label{eqn:transforeff2}
    \phi_*(\beta_u)=
        \begin{cases}
            -\beta_u', &\text{ for }u= v\\
           \beta_u'+\beta_v' &\text { for }u\rightarrow v\\
           \beta_u' &\text { otherwise }\,.
        \end{cases}
    \end{equation}
\end{prop}
\begin{proof}
Notice that $(c_1(S_u))_{u\in  Q_0\backslash Q_f}$ spans $H^2(\mc Z)$. One can check directly that the $\phi_*$ defined above satisfies \eqref{eqn:homologicalchange}, via Proposition \ref{prop:equicohmu} and \eqref{eqn:basisofeff}.
\end{proof}

\subsection{One-dimensional torus-fixed orbits}
\subsubsection{Grassmannian}
In this section, we will study the one-dimensional $S$-orbits in Grassmannian $\mc X=Gr(N_1, N_2)$ to prepare for that in flag variety with $S=(\mathbb C^*)^{N_2}$. 
We will prove that $Gr(N_1, N_2)$ is a GKM space and describe its one-dimensional $S$-orbits.
We already know that the $S$-fixed points can be represented by sets $\{I\subset [N_2]\}$, and we denote the point by $p_I$.
\begin{lem}\label{1 dim Gr}
	The set of one-dimensional $S$-orbits in $\mc X=Gr(N_1,N_2)$ is one-dimensional. 
    The closure of a one-dimensional $S$-orbit contains exactly two $S$-fixed points.
    Two fixed points $p_{I}$ and $p_J$ are connected by the closure of a one-dimensional $S$-orbit $\ell$ if and only if $|I\cap J|=N_1-1$.
	\end{lem}
\begin{proof}
	Consider the Pl\"ucker embedding $\iota: \mc X\rightarrow \mathbb{P}^{\binom{N_2}{N_1}-1}$, which sends a matrix $A$ to the determinants of submatrices of size $N_1\times N_1$.
    We denote the homogeneous coordinates of $\mathbb{P}^{C_{N_2}^{N_1}-1}$ by $(x_I)$, where $I$ denote the number of columns of the submatrix.
    The induced  $S$-action on $\mathbb{P}^{C_{N_2}^{N_1}-1}$
	$$(t_1, t_2, \ldots, t_{N_2})\cdot (x_I)=((\prod_{i\in I}t_i)x_I)$$ 
	makes the Pl\"uker embedding an $S$-equivariant map. 
   The $S$-fixed points in $\mathbb{P}^{C_{N_2}^{N_1}-1}$ are:
	$$q_I=[0:\ldots:0: x_I: 0:\ldots: 0]$$
    which is the image of $S$-fixed point $p_I\in \mc X$, so $\iota$ induces an one-to-one correspondence between  $\mc X^S$ and $(\mathbb{P}^{C_{N_2}^{N_1}-1})^S$.
	
	Suppose $\ell$ is an one-dimensional $S$-orbit closure in $\mc X$, then $\iota(\ell)$ is an one-dimensional $S$-orbit closure in $\mathbb{P}^{C_{N_2}^{N_1}-1}$ as well. 
    Since $\mathbb{P}^{C_{N_2}^{N_1}-1}$ is a GKM space, $\mc X$ is also a GKM space. 
    The one-dimensional $S$-orbit closure $\ell$ contains two $S$-fixed points $p_I$ and $p_J$ in $\mc X$ if $\iota(\ell)$ contains two fixed points $q_I$ and $q_J$ in $\mathbb{P}^{C_{N_2}^{N_1}-1}$. 
    
    We will prove that two such different points $p_I,p_J$ are contained in an one-dimensional $S$-orbit $\ell$ if and only if they satisfy $|I\cap J|=N_1-1$.
	\begin{enumerate}
		\item If $|I\cap J|=N_1$, then $q_I=q_J$, which is clearly impossible.
		\item If $|I\cap J|\leq N_1-2$, we will prove it is impossible by contradiction.
        We first can find two distinct $i_1, i_2\in I\backslash J$. 
        Choose an arbitrary point $p\in \ell\backslash\{p_I, p_J\}$, whose matrix representation can be written as
		\begin{equation*}
        A=
        \begin{bmatrix}
            \vec v_1& \vec v_2& \cdots & \vec v_{N_2}
        \end{bmatrix}.
    \end{equation*}
    Consider the image $\iota(\ell)$, the one-dimensional $S$-orbit in $\mathbb P^{C_{N_2}^{N_1}-1}$, and $\iota(p)$, a point in $\iota(\ell)$. 
    Since $\iota(p)\in \iota(\ell)$, its homogeneous coordinates $x_K$ vanish if and only if $K\neq I, J$. 
    The condition $x_I\neq 0$ implies that column vectors $\{\vec v_i|i\in I\}$ are linear independent, so are $\{\vec v_j|j\in J\}$. 
    For any $j\in J\backslash I$, $(\vec v_i)_{i\in I\backslash \{i_1\}\cup \{j\}}$ are linear dependent according to $x_K=0$ for $K\neq I,J$, so $\vec v_{j}$ is a linear combination of $\{\vec v_i|i\in I\backslash\{i_1\}\}$. 
    This implies that all $N_1$ vectors $\{\vec v_j|j\in J\}$ are linear combinations of $r-1$ vectors $\{\vec v_i|i\in (I-\{i_1\})\}$, so $\{\vec v_j|j\in J\}$ are linear dependent. 
    Then we get a contradiction.
    \item If $|I\cap J|= N_1-1$, and assume that $I\backslash J=i_s$ for some $s\in [N_1]$ and $J\backslash I=j_s$, then one can easily prove that such an $\ell$ exists and is parameterized by $t\in \C\cup \{\infty\}$ with a matrix representation $A_t$ constructed as follow. 
    Denote $I\cap J=\{f_1<f_{s-1}<f_{s+1}\cdots<f_{N_1}\}$. 
     \begin{equation}\label{1 dim orb}
     \begin{cases}
        (A_{t})_{ab}=\delta_{f_a b} \text{ for } a\neq s\\
        (A_t)_{s i_s}=\frac{1}{1+t}\\
        (A_t)_{sj_s}= \frac{t}{t+1} 
     \end{cases}
        \end{equation}
	\end{enumerate}
\end{proof}
\subsubsection{flag varieties}
In this subsection, we let $\mc X=Fl(N_1, N_2, \ldots, N_{n+1})$. 
\begin{lem}\label{lem:1dimorbflag}
	The set of one-dimensional $S$-orbits of $\mc X$ is one-dimensional. 
    Each one-dimensional $S$-orbit closure contains exactly two fixed points. 
    Two fixed points $p_{\vec {I}}$ and $p_{\vec {J}}$ are connected by a one-dimensional $S$-orbit closure $\ell$ if and only if there exists $k\in [n]$, such that the following three conditions are satisfied.
	\begin{enumerate}
	\item $|I_k\cap J_k|=N_k-1$,
	\item $I_l=J_l$, for $l>k$,
	\item $I_l=J_l$ or $I_k\backslash I_l=J_k\backslash J_l$,  for $l<k$.
\end{enumerate}
\end{lem}
\begin{proof}
	We prove by induction on $n$. 
    For $n=1$, this is Lemma \ref{1 dim Gr}. 
    Assume that the statement holds for $n-1$.  
    We will prove the statement for length $n$ flag variety.  
    Consider the following fiber bundle structure
    \begin{equation}\label{Gr-fiber}
    \begin{tikzcd}
    Fl(N_1, \ldots, N_n)\arrow[hookrightarrow]{r}{} 
     & Fl(N_1, \ldots, N_n, N_{n+1}) \arrow[]{d}{\pi}\\
    & Gr(N_n, N_{n+1})
     \end{tikzcd}
     \end{equation}
	Suppose $\ell$ is an one-dimensional $S$-orbit closure in $Fl(N_1, N_2, \ldots, N_n, N_{n+1})$, we consider $\pi(\ell)\subset Gr(N_n, N_{n+1})$. 
	\begin{enumerate}
		\item If $\pi(\ell)$ is a point $q\in Gr(N_n, N_{n+1})$, then $\ell$ lies entirely in the fiber $\pi^{-1}(q)\cong Fl(N_1, N_2, \ldots, N_n)$. 
        Then we are done by the induction assumption.
		\item If $\pi(\ell)$ is a line $\ell_b\subset Gr(N_n,N_{n+1})$, then $\ell_b$ is an one-dimensional $S$-orbit closure in $Gr(N_n, N_{n+1})$ since $\pi$ is $S$-equivariant. 
        By Lemma \ref{1 dim Gr}, $\ell_b$ contains two $S$-fixed points $q_I, q_J\in Gr(N_n, N_{n+1})^S$, satisfying $|I\cap J|=N_n-1$. Without loss of generality, we assume that $I\backslash J=\{N_{n+1}-1\}$, $J\backslash I=\{N_{n+1}\}$. 
        Fix a point $q\in \ell_b\backslash \{q_I, q_J\}$, and we assume that its matrix representation is given by $A_{t_0}$ in \eqref{1 dim orb} for some fixed $t_0$. 
        The matrix representation of the fiber $\pi^{-1}(q)$ can be $(A_1, A_2, \ldots, A_{n-1}, A_{t_0})$. 
        Now one can check that there is a $S'=(\C^*)^{N_{n+1}-1}$-action on $\pi^{-1}(q)$, where the inclusion $S'\hookrightarrow S$  is given by $(t_1, t_2, \ldots, t_{N-2}, t, t)$. 
        Actually, this action can be viewed as an extension of the natural $(\C^*)^{N_n}$-action on $Fl(N_1, N_2, \ldots, N_n)$ to $\pi^{-1}(q)$.
        
        The point $p\in \ell\cap \pi^{-1}(q)$ must be a $S'$-fixed point in $Fl(N_1, N_2, \ldots, N_n)$. Otherwise the whole $S$-action on $p$ will result in a larger than one dimensional $S$-orbit, which contradicts the assumption that $\ell$ is one-dimensional.
        By continuity, when $q$ varies, $p$ should have the same matrix representation. In the formula of \eqref{eqn:Sfixedpointsinflag}, either  $I_l=J_l$, or $I_l$ contains the $N_{n+1}-1$  and $J_l$  contains the $N_{n+1}$. 
        Hence we have proved the induction in case $n$.
		\end{enumerate}
\end{proof}
\subsubsection{general case}
We first consider the quiver in Figure \ref{fig:oneskeleton-case1} which is mutation-equivalent to $A_2$. Then the variety $\mc Z$ is a complete intersection in $Gr(r_1, N)\times Gr(r_2, N)$ with defining equation $A_1A_2=0$. 
\begin{figure}[ht]
\centering
\begin{tikzpicture}
    \node[draw,
	circle,
	minimum size=1cm,
	] (node1) at (-1.2,0){$r_1$};
    \node at (0,-0.2){$B$};
    \node[draw,
	circle,
	minimum size=1cm,
	] (node2) at (1.2,0){$r_{2}$};
    \node at (-0.9,1){$A_1$};
    \node at (0.9,1){$A_2$};
    \node[draw,
	minimum width=1cm,
	minimum height=1cm,
	] (nodef) at (0,1.8){$N$};
	\draw[-stealth] (node1) -- (nodef);
	\draw[-stealth] (nodef) -- (node2);
	\draw[-stealth] (node2) -- (node1);
\end{tikzpicture}
\caption{} 
\label{fig:oneskeleton-case1}
\end{figure}
\iffalse
\begin{figure}[H]
    \centering
    \includegraphics[width=1in]{oneskeleton-case1.png}
    \caption{}
    \label{fig:oneskeleton-case1}
\end{figure}
\fi
As Lemma \ref{lem:torusfixedpointgeneral}, the $S$-fixed points $\mc Z$ are parameterized as follows
\begin{equation*}
	\{(I_1, I_2)\big|I_1, I_2\subset [N], |I_1|=r_1, |I_2|=r_2, I_1\cap I_2=\emptyset\}.\nonumber
\end{equation*}
Now we investigate the one-dimensional $S$-orbits in $\mc Z$. \begin{lem}\label{lem:1 dim mut}
	The union of one-dimensional $S$-orbits in $\mc Z$ is one-dimensional. 
    Each one-dimensional $S$-orbit closure contains exactly two $S$-fixed points, and two fixed points $p_{(I_1, I_2)}$ and $p_{(J_1, J_2)}$ are connected by a one-dimensional $S$-orbit closure $\ell$ if and only if one of the following three conditions is satisfied. 
	\begin{enumerate}
		\item $|I_2\cap J_2|=r_2-1$, $I_1=J_1$, $I_1\cap I_2=\emptyset$, $J_1\cap J_2=\emptyset$,
        \item $|I_1\cap J_1|=r_1-1$, $I_2=J_2$, $I_1\cap I_2=\emptyset$, $J_1\cap J_2=\emptyset$,
		\item $|I_k\cap J_k|=r_k-1$, $k=1, 2$, and $I_1\cup I_2=J_1\cup J_2$, $I_1\cap I_2=\emptyset$, $J_1\cap J_2=\emptyset$.
	\end{enumerate}
	\end{lem}
	\begin{proof}
		Consider maps which are composite of an inclusion and one projection $f_k: X\hookrightarrow Gr(r_1, N)\times Gr(r_2, N)\rightarrow Gr(r_k, N)$, $k=1, 2$, and they are $S$-equivariant.
        Suppose $\ell$ is an one-dimensional $S$-orbit closure in $\mc Z$. The images $f_1(\ell)$ and $f_2(\ell)$ have the following three situations.
		\begin{enumerate}
			\item The first case is when $f_1(\ell)$ is point in $Gr(r_1, N)$, then it must be an $S$-fixed point in $Gr(r_1,N)$, which we denote by $I\subset [N]$ with $\abs{I}=r_1$. Then we have $I_1=J_1=I$.
            The image $f_2(\ell)$ is an one-dimensional $S$-orbit closure in $Gr(r_2, N)$ containing two $S$-fixed points indexed by $I_2,J_2$. 
            According to the description for one-dimensional $S$-orbits in $Gr(r_2,N)$ in Lemma \ref{1 dim Gr}, the two  sets ${I_2}$ and ${J_2}$ satisfy $|I_2\cap J_2|=r_2-1$. 
            Furthermore, one can check that all $\ell$ in  $Gr(r_1, N)\times Gr(r_2, N)$ satisfying this condition entirely lie in $\mc Z$, 
            since $I\cap I_2=I\cap J_2=\emptyset$.
            Hence, we have proved the case (1). Similarly, one can find the case $(2)$ which is omitted.
			\item The complicated case is when both $f_1(\ell)$ and $f_2(\ell)$ are one-dimensional.
            Torus fixed points $p_{I_k},p_{J_k}$ are contained in $f_k(\ell)$.
            Without loss of generality, we assume that $I_1=\{1, 3, \ldots, r_1+1\}$, $J_1=\{2, 3, \ldots, r_1+1\}$. 
            Then a matrix representative of $f_1(\ell)=\mathbb P^1$  can be written as follows 
            \begin{equation}\label{eqn:oneskeletonA1}
            A_1=
            \begin{bmatrix}
            s&t&0&\cdots &0&\ldots&0\\
            0&0&1&\cdots &0&\ldots&0\\
            \vdots&\vdots&\vdots& \ddots&\vdots& &\vdots\\
            0&0&0&\cdots &1&\ldots&0
            \end{bmatrix}
            \end{equation}
            where $[s:t]$ are the homogeneous coordinates of $f_1(\ell)=\mathbb P^1$.
			
            We can view $Gr(r_1, N)\times Gr(r_2, N)$ as a total space of $Gr(r_2, N)$ fiber bundle over $Gr(r_1, N)$. 
            For an arbitrary point $p$ in $f_1(\ell)$, 
            there is a $S'=(\C^*)^{N-1}$-action on the $Gr(r_2, N)$-fiber over $p$ by the inclusion $S'\hookrightarrow S$:  $t\rightarrow (t, t, t_2, \ldots, t_{N-1})$. 
            Since $f_2(\ell)$ is invariant under $S'$-action, the matrix representation of $f_2(\ell)$ can be written as 
            \begin{equation}\label{eqn:oneskeletonA2}
            A_2=
            \begin{bmatrix}
            a&0&0&\cdots &0\\
            b&0&0&\cdots &0\\
            \vdots&\vdots&\vdots& \ddots&\vdots\\
            0&0&1&\cdots &0\\
            \vdots&\vdots&\vdots& \ddots&\vdots\\
            0&0&0&\cdots &1
            \end{bmatrix}.
            \end{equation}
            Since $\ell$ lies in $\mc Z$, $A_1A_2=0$. Then we have $a=-t$ and $b=s$. 
            By letting $t=0, s=1$, we obtain the two $S$-fixed points $(I_1,I_2)$ such that $1\in I_1$, $2\in I_2$, 
            and $I_1\cap I_2=\emptyset$.
            Similarly, by letting $t=1,s=0$, we get the other $S$-fixed point $(J_1,J_2)$, such that $2\in J_1, 1\in J_2$, $J_1\cap J_2=\emptyset$. At the same time, we have $I_1\backslash \{1\}=J_1\backslash \{2\}$, $I_2\backslash \{2\}=J_2\backslash \{1\}$ and   $I_1\cup I_2=J_1\cup J_2$.
            Hence, we have proved the third case. 
            \end{enumerate}
    \end{proof}
We consider a general decorated QP in $\Omega^A_n$. 
Let $\mc Z:=\{dW=0\}\sslash_{\theta}G$ be the associated variety.
We denote by $p_{\vec I}$ the torus fixed point indexed by a set of sets $\vec I=(I_v)_{v\in Q_0\backslash v_f}$. 
\begin{prop}\label{prop:oneskeletonAnmutation}
     Each one-dimensional $S$-orbit closure contains exactly two $S$-fixed points, and 
     two distinct $S$-fixed points $p_{\vec {I}}$ and $p_{\vec {J}}$ are contained in the closure of an one-dimensional $S$-orbit if and only if sets $\vec {I}$ and $\vec {J}$ satisfy the following three conditions: 
	    \begin{enumerate}
		\item There exists a unique node $v_0$ such that  for $v$ not smaller than $v_0$, $I_v=J_v$.
		\item 
        There exist at most two maximal nodes that are smaller than $v_0$. 
         If there is only one maximal node $v_1$ that is smaller than $v_0$, then $I_{v_1}\cap J_{v_1}=r_{v_1}-1$.
        If there are exactly two maximal nodes $v_1$ and $v_2$ that are smaller than $v_0$, then $(I_{v_k},J_{v_k}),\,k=1,2$ must satisfy one of the following three conditions:
		\begin{enumerate}
		\item $|I_{v_1}\cap J_{v_1}|=r_{v_1}-1$, $I_{v_2}=J_{v_2}$,
		\item $|I_{v_2}\cap J_{v_2}|=r_{v_2}-1$, $I_{v_1}=J_{v_1}$,
		\item $|I_{v_k}\cap J_{v_k}|=r_{v_k}-1$, $k=1, 2$, and $I_{v_1}\cup I_{v_2}=J_{v_1}\cup J_{v_2}$.
	   \end{enumerate}
       
        \item For $w\prec v_k$, $I_{w}=J_w$ or $I_{v_k}\backslash I_w=J_{v_k}\backslash J_w$, $k=1, 2$.	
        \end{enumerate}
\end{prop}
\begin{proof}
	We prove by induction on the number of gauge nodes $n$. 
    For $n=1$ and $2$, this is proved by Lemmas \ref{1 dim Gr}, \ref{lem:1dimorbflag} and \ref{lem:1 dim mut}. 
    Assume that the Lemma is true for $n-1$.
    We're going to prove the case when the number of gauge node is $n$. 
    
    Let $\ell$ be an one-dimensional $S$-orbit closure in $\mc Z$.
    When the frozen node $v_f$ has only one maximal node $v_1$ smaller than it, then we have the following fiber bundle 
    \begin{equation*}
    \begin{tikzcd}
    \mc Z'\arrow[hookrightarrow]{r}{} 
     & \mc Z \arrow[]{d}{\pi}\\
    & Gr(r_{v_1}, N_{n+1})
     \end{tikzcd}
    \end{equation*}
    where the fiber $\mc Z'$ is a variety constructed from a quiver diagram in $\Omega^A_{n-1}$.
    We can prove the Lemma by the same induction argument with Lemma \ref{lem:1dimorbflag}.

    When the frozen node $v_f$ has two maximal nodes $v_1,v_2$ smaller than it, 
    we assume that the three nodes form a 3-cycle $v_1\rightarrow v_f\rightarrow v_2\rightarrow v_1$. 
    By deleting the arrow $v_2\rightarrow v_1$, the quiver $\mathbf Q$ is separated into two quivers $\mathbf Q_1,\mathbf Q_2$, and both quivers have one frozen node $v_f$. 
    If we assume that there are $n_1$ nodes smaller than $v_1$ and $n_2$ nodes smaller than $v_2$, then $\mathbf Q_1$ and $\mathbf Q_2$ are in $\Omega^A_{n_1+1}$ and $\Omega^A_{n_2+1}$. 
    Let $\mc Z_1$ and $\mc Z_2$ be the two critical loci of two potentials of quivers $\mathbf Q_1$ and $\mathbf Q_2$, and then $\mc Z$ is defined by additional equations $A_{v_1}A_{v_2}=0$.

    Consider two composite maps $f_k: \mc Z\hookrightarrow  \mc Z_1\times \mc Z_2 \rightarrow \mc Z_k$, $k=1, 2$. 
    The images of $f_k(\ell)$ have the following situations.
	\begin{enumerate}
		\item When $f_1(\ell)$ is a point $p$ in $\mc Z_1$, then $p$ is an $S$-fixed point in $\mc Z_1$. Then $f_2(\ell)$ is an $S$-invariant projective line in $\mc Z_2$. Then we can use the induction assumption on $\mc Z_2$ to check this case. Notice that now $v_0$ is a node in $\mathbf Q_2$ and all gauge nodes in $\mathbf Q_1$ are not smaller than $v_0$. 
		\item When $f_2(\ell)$ is a point in $\mc Z_2$, this is similar as above.
		\item When both $f_1(\ell)$ and $f_2(\ell)$ are projective lines, there is a fiber bundle structure on $\mc Z_1$ and $\mc Z_2$, 
        \begin{equation*}
         \begin{tikzcd}
         \mc Y_k\arrow[hookrightarrow]{r}{} 
        & \mc Z_k \arrow[]{d}{\pi_k}\\
        & Gr(r_{v_k}, N_{n+1})
         \end{tikzcd}
        \end{equation*}
        We first claim $\pi_k\circ f_k(\ell)$, for $k=1, 2$, are both lines.
        In fact, if $q=\pi_1\circ f_1$ is a point, then $\ell$ is a line in $\pi_1^{-1}(q)\times \mc Z_2$, and $\pi_1^{-1}(q)\iso \mc Y_1$. Let $p_0$ be a point in $f_2(\ell)\subset \mc Z_2$, one can find a $(\C^*)^{r_{v_1}}$-action on the $\mc Y_1$-fiber over $p_0$. This action forces the projection of $\ell$ to $\mc Y_1$ to be a $(\C^*)^{r_1}$-fixed point, since $\ell$ is one dimensional. Then $f_1(\ell)$ must be a point, which is a contradiction. 

        By a similar argument as in the proof of Lemma \ref{lem:1 dim mut}, we have \begin{align*}
			&|I_{v_k}\cap J_{v_k}|=r_{v_k}-1, \quad k=1, 2, \\
			& I_{v_1}\cup I_{v_2}=J_{v_1}\cup J_{v_2}.
		\end{align*}
		Fix a point $p_1$ in $\pi_1\circ f_1(\ell)$. Similar with the proof of Lemma \ref{lem:1dimorbflag},
        there is a $\C^{r_{v_1}}$-action on the fiber $\pi_1^{-1}(p_1)\cong \mc Y_1$, which makes $f_1(\ell)\cap \pi_1^{-1}(p_1)$ a $\C^{r_1}$-fixed point in $\mc Y_1$. 
        By continuity, we have for $w\prec v_1$, $I_{w}=J_w$ or $I_{v_1}\backslash I_w=J_{v_1}\backslash J_w$. 
        The same argument works for $v_2$. 
        Then we're done with the induction for the case $n$. 
		\end{enumerate}
\end{proof}
\subsection{Behavior of one-dimensional torus-fixed orbits under quiver mutations}
Adopting the same notations as the previous section, let $(\mathbf Q,\mathbf r)$ and $(\mathbf Q',\mathbf r')$ be two QPs related by a quiver mutation $\mu_v$, 
and $\mc Z$ and $\mc Z'$ be the corresponding critical loci of potentials. 
\begin{prop}\label{prop:torusfixedgraphmu}
   Let $p_{\vec I}$ and $p_{\vec J}$ be two $S$-fixed points in $\mc Z$ contained in a one-dimensional $S$-orbit closure $\ell$. 
   Then images $\varphi(p_{\vec I})$ and $\varphi(p_{\vec J})$ are also contained in a one-dimensional $S$-orbit closure in $\mc Z'$, which we denote by $\varphi(\ell)$.
\end{prop}
\begin{proof}
    Denote the image $\varphi( p_{\vec I})$ and $\varphi( p_{\vec J})$ by $p_{\vec I'}$ and $p_{\vec J'}$ with $\vec I'=(I'_v)_{v\in Q_0}$ and $\vec J'=(J_v')_{v\in Q_0}$.
    Without loss of generality, we assume that at the gauge node $v$, we have $N_f(v)>N_a(v)$ and the node $v$ has four adjacent nodes $v_i,i=1,\cdots,4$ as shown in Figure \ref{fig:localbehavior} (a). 
    Then after the quiver mutation $\mu_v$, 
    we have $I_{v}'=(I_{v_1}\backslash I_{v})\cup I_{v_4}$, 
    and $I_{u}'=I_{u}$ for $u\neq v$, and the same for $\vec J'$.
    We then need to discuss several situations and 
    we will check in each situation, the two sets $\vec I'$ and $\vec J'$ satisfy the three conditions in Proposition \ref{prop:oneskeletonAnmutation}.
    \begin{enumerate}
    \item When $I_v=J_v$, 
    there are three different cases. The first case is when $v_0$ is no greater than $v$, and $I_{v_i}=J_{v_i}$ for each $i=1,\cdots,4$. 
    The second case is when $v_0$ is greater than $v$, and $I_v=J_v$, it is possible for $I_{v_2}$ to differ with $J_{v_2}$ by one element, $I_{v_k}=J_{v_k}$ for $k=3,4$. The third case is when $v_0=v$, with $I_{v_k}\cap J_{v_k}=r_{v_k}-1$ for $k=3$ or $k=4$. 
    It is easy to check that $\vec I'$ and $\vec J'$ satisfy the three conditions in Proposition \ref{prop:oneskeletonAnmutation} for all three cases and we omit the proof.  
    \item The second situation is when $I_{v}\cap J_{v}=r_v-1$, $I_{v_2}=J_{v_2}$. 
    Assume that $I_v\backslash J_v=\{i_1\}$ and $J_v\backslash I_v=\{j_1\}$. 
    Furthermore, we have one of the three cases holds in terms of relations of the sets $I_{v_k}, \,J_{v_k}$, $k=3,4$. 
    \begin{enumerate}
        \item If we further have $I_{v_3}=J_{v_3}, \,I_{v_4}=J_{v_4}$, then it is easy to check $I_v'\cap J_{v}'=r_v'-1$, $I_v'\backslash J_v'=\{j_1\}$, and $J_v'\backslash I_v'=\{i_1\}$. Hence $\vec I'$ and $\vec J'$ satisfy conditions in Proposition \ref{prop:oneskeletonAnmutation}. 
        \item If we further have $I_{v_3}=J_{v_3}, |I_{v_4}\cap J_{v_4}|=r_{v_4}-1$, then we have $I_{v_4}\backslash J_{v_4}=\{i_1\}$ and $J_{v_4}\backslash I_{v_4}=
        \{j_1\}$.
        Hence the two sets $I_v'=\left(I_{v_1}\backslash (I_v\cup J_v)\right)\cup \{j_1\}\cup (I_{v_4}\cap J_{v_4})\cup \{i_1\}$ and $J_v'=\left(J_{v_1}\backslash (I_v\cup J_v)\right)\cup \{i_1\}\cup (I_{v_4}\cap J_{v_4})\cup \{j_1\}$  are equal. 
        One can check that the two sets $\vec I'$ and $\vec J'$ satisfy conditions in Proposition \ref{prop:oneskeletonAnmutation}. 
        \item The case $I_{v_4}=J_{v_4}, |I_{v_3}\cap J_{v_3}|=r_{v_3}-1$ is similar as above and we omit it.
    \end{enumerate}
    \item The third situation is when
    $I_v\cap J_v=r_v-1$ and $I_{v_2}\cap J_{v_2}=r_{v_2}-1$.
    We assume that $I_{v}\backslash J_{v}=\{i_1\}$, $J_v\backslash I_v=\{j_1\}$, $I_{v_2}\backslash J_{v_2}=\{j_1\}$ and $J_{v_2}\backslash I_{v_2}=\{i_1\}$. 
    Then we still need to discuss the relations of sets $I_{v_k}$ and $J_{v_k}$ for $k=3,4$ as the above situation, and the proofs are exactly similar as above, and we omit them. 
\end{enumerate} 
Hence, we have proved that the two sets $\vec I'$ and $\vec J'$ satisfy conditions in Proposition \ref{prop:oneskeletonAnmutation} in each situation. 
\end{proof}
The above Lemma tells us that there is a bijection between the one-dimensional $S$-orbits which we denote by
\begin{equation}\label{eqn:phi1dim}
    \varphi:\mc Z^1\rightarrow (\mc Z')^1.
\end{equation}
Hence, for each torus fixed graph $\Upsilon\subset \mc  Z^1$, $\varphi(\Upsilon)\subset (\mc Z')^1$.
\subsection{Behavior of weights  under quiver mutations}
Fix a torus fixed graph $\Upsilon\subset \mc Z^1$ and denote $\Upsilon'=\varphi(\Upsilon)\subset (\mc Z')^1$. 
The image $(\varphi(\tau),\varphi(\sigma))$ of each flag is also a flag in $F(\Upsilon')$. 
Since the torus fixed loci $\mc Z^S$ is parameterized by sets $\vec I=(I_v)_{v\in Q_0\backslash\{v_f\}}$, we will alternatively use the set $\vec I$ to denote vertices of graphs  $\Upsilon$.

The main goal of this subsection is to investigate the behavior of weights of flags defined in \eqref{eqn:weights} under quiver mutations. 

\subsubsection{Grassmannian case}
Let $\sigma_1=I=\{i_1<\cdots<i_r\},\sigma_2=J=\{j_1<\cdots<j_r\}$ be two $S$-fixed points in $Gr(r,N)$ which, without loss of generality, satisfy $i_1\neq j_1,i_k=j_k$ for $k\neq 1$. Let $\tau$ be the 1-dimensional $S$-orbit connecting the two points whose matrix representation is
\begin{equation}\label{eqn:1dimensionalorbit}
    \tau=
    \begin{bmatrix}
        \mathbf 0&a&\mathbf 0&b&\mathbf 0\\
       \mathbf 0&  \mathbf 0 &\mathbf 0 & \mathbf 0&B
    \end{bmatrix}
\end{equation}
where $a$ and $b$ lie in $i_1$-th and $j_1$-th columns respectively, $\mathbf 0$ represents vectors or zero matrices of suitable size, and the matrix $B$ satisfies $B_{kj}=\delta_{j_k,j}$. 
The closure of $\tau$ is $\mathbb P^1$ whose homogeneous coordinate is $[a:b]$.
The torus $(\mathbb C^*)^N$ acts on the matrices representations of $Gr(r,N)$ as 
\begin{equation}\label{eqn:torusactsonGrone}
    t\cdot A=At^{-1}, \,A\in Gr(r,N), t\in (\mathbb C^*)^N\,.
\end{equation}
One can easily check that 
\begin{equation}\label{eqn:weightsGr}
    \mathbf{w}(\tau,\sigma_1)=c_1^S(T_{\sigma_1}\tau)=\lambda_{i_1}-\lambda_{j_1},\,\mathbf{w}(\tau,\sigma_2)=c_1^S(T_{\sigma_2}\tau)=\lambda_{j_1}-\lambda_{i_1}.
\end{equation}

Performing a quiver mutation $\mu$, the two points are mapped to $\varphi(\sigma_1)=I'=[N]\backslash I=([N]\backslash (I\cup J))\cup \{j_1\}$ and $\varphi({\sigma_2})=([N]\backslash (I\cup J))\cup \{i_1\}$. 
Denote $[N]\backslash (I\cup J)=\{i_2'<\ldots<i_{N-r}'\}$. 
Then the matrix representation of the one-dimensional $S$-orbit $\varphi(\tau)$ is
\begin{equation}\label{eqn:rcrf}
    \begin{bmatrix}
        \mathbf 0&\mathbf 0\\
        b&\mathbf 0\\
        \mathbf 0&\mathbf 0\\
        a&\mathbf 0\\
        \mathbf 0&C
    \end{bmatrix}
\end{equation}
where $b$  and $a$ lie in the $i_1$-th and $j_1$-th rows and entries of matrix $C$ are $C_{kl}=\delta_{ki_l'}$. 
When $a=1,b=0$ we get the point $\varphi(\sigma_1)$ and when $a=0,b=1$, we get the point $\varphi(\sigma_2)$.
The torus $S$ acts on $Gr(N-r,N)$ as 
\begin{equation}\label{eqn:torusactionGr1dimn}
    t\cdot A=tA,\,t\in S,\,A\in Gr(N-r,N). 
\end{equation}
The weights of the flag at vertices $\varphi(\sigma_1)$ and $\varphi(\sigma_2)$ are
\begin{equation}\label{eqn:weightsdualGr}
    \mathbf w(\varphi(\tau_1),\varphi(\sigma))=\lambda_{i_1}-\lambda_{j_1},\,\mathbf w(\varphi(\tau),\varphi(\sigma_2))=\lambda_{j_1}-\lambda_{i_1}.
\end{equation}
Compare equations \eqref{eqn:weightsGr} and \eqref{eqn:weightsdualGr}, and one can find 
\begin{equation*}
    \mathbf{w}(\tau,\sigma_1)=\mathbf w(\varphi(\tau_1),\varphi(\sigma)),\,\,
    \mathbf{w}(\tau,\sigma_2)=\mathbf w(\varphi(\tau),\varphi(\sigma_2))\,.
\end{equation*}
\subsubsection{General case}\label{subsec:1dorbitmutation}
\begin{prop}\label{prop:weightsmu}
    For each flag $(\tau, \sigma)\in \Upsilon$, we have
    \begin{equation*}
    \mathbf w(\tau,\sigma)= 
    \mathbf w(\varphi(\tau),\varphi(\sigma)).
    \end{equation*}
\end{prop}
\begin{proof}
Let $\sigma_1=(I_u)_{u\in Q_0\backslash v_f}$ and $\sigma_2=(J_u)_{u\in Q_0\backslash v_f}$  be two vertices in $\Upsilon$ in $\mc Z$ that are connected by an edge $\tau$.
Let $\sigma_1':=\varphi(\sigma_1)=(I_u')$ and $\sigma_2':=\varphi(\sigma_2)=(J_u')$ be images in $\Upsilon'$.

Without loss of generality, we assume that the node $v$ where we perform a quiver mutation has four adjacent nodes $v_1,v_2,v_3,v_4$ as shown in Figure \ref{fig:localbehavior} (a) with the assumption $N_f(v)>N_a(v)$. 
One can check that 
if $I_v=J_v$ as the situation (1) in the proof of Proposition \ref{prop:torusfixedgraphmu}, the quiver mutation $\mu_v$ preserves weights easily. 
We only have to check the change of weights when $I_v\cap J_v=r_v-1$. 
We restrict ourselves to $I_{v_2}=J_{v_2}$ as the situation $(2)$ in the proof of Proposition \ref{prop:torusfixedgraphmu}, and leave the case $|I_{v_2}\cap J_{v_2}|=r_{v_2}-1$ to readers. 
Without loss of generality, we assume that 
    $I_v\backslash J_v=\{i_1\}$,  $J_v\backslash I_v=\{j_1\}$, and $I_v\cap J_v=\{i_2<\cdots <i_{r_v}\}$.
Then we need further to consider the following cases.  
\begin{enumerate}
    \item When $I_{v_3}=J_{v_3}$, $I_{v_4}=J_{v_4}$, one matrix representative $(A_u)_{u\in Q_0\backslash \{v_f\}}$ can be written as  $A_u, u\neq v$ are in reduced row/column echelon forms whose columns/rows numbers of pivots are $I_u,u\neq v$. 
    Meanwhile, the matrix  $A_v$ can be written as the same formula as \eqref{eqn:1dimensionalorbit}. 
    We get the vertex $\sigma_1$ if $a=1,b=0$ and $\sigma_2$ is $a=0,b=1$. There is an induced torus action $(\mathbb C^*)^{r_{v_1}}$ on the matrix $A_v$ as \eqref{eqn:torusactsonGrone}. Hence, we have $\mathbf w(\tau, \sigma_1)=\lambda_{i_1}-\lambda_{j_1}$, and $\mathbf w(\tau, \sigma_2)=\lambda_{j_1}-\lambda_{i_1}$.

    After the quiver mutation $\mu_v$, we then have $I_{v}'\backslash J_v'=\{j_1\}$ and $J_{v}'\backslash I_v'=\{i_1\}$. The matrices representatives $A_u',\,u\neq v$ are the same with $A_u, \,u\neq v$, but $A_v'$ is changed to the formula \eqref{eqn:rcrf}, with the induced torus action $(\mathbb C^*)^{r_{v_1}}$ the same with \eqref{eqn:torusactionGr1dimn}. One can easily find that the weights of the flag at the vertex $\varphi(\sigma_1)$ and $\varphi(\sigma_2)$ are also the same with \eqref{eqn:weightsdualGr}.
    Hence, we have proved the Proposition. 
    
    \item When $|I_{v_3}\cap J_{v_3}|=r_{v_3}-1, I_{v_4}=J_{v_4}$, 
    the matrices representation for $\tau$ can be described as follows. Matrices $A_u$ for $u\neq v$ are reduced row/columns echelon forms with non-pivots entries vanishing and pivots columns/rows numbers are $I_u, u\neq v$.
    The matrix $A_v$ is in the formula \eqref{eqn:1dimensionalorbit}.
    Similar as the above situation, we still have
    \begin{equation}\label{eqn:weightsgeneral}
        \mathbf w(\tau,\sigma_1)=\lambda_{i_1}-\lambda_{j_1}, \mathbf w(\tau,\sigma_2)=\lambda_{j_1}-\lambda_{i_1}.
    \end{equation}

    After the quiver mutation $\mu_v$, the matrices representations for $\varphi(\tau)$ can be described as follows. 
    For $u\neq v, u\neq v_3$, matrices $A_u'$ are exactly the same with $A_u,\, u\neq v, \,u\neq v_3$. 
    The matrices $A_{v_3}'$ and $A_v'$ are in the formula \eqref{eqn:oneskeletonA1} and \eqref{eqn:oneskeletonA2}.
    The induced torus $(\mathbb C^*)^{r_{v_1}}$ acts on $\tau$  as \begin{equation*}
        t\cdot (A_{v_3}',A_v')=(A_{v_3}' t^{-1}, tA_{v}').
    \end{equation*}
    Hence 
    \begin{equation}\label{eqn:weightsgeneraldual}
        \mathbf w(\varphi(\tau),\varphi(\sigma_1))= \lambda_{i_1}-\lambda_{j_1}, \,
    \mathbf w(\varphi(\tau),\varphi(\sigma_2))= \lambda_{j_1}-\lambda_{i_1}.
    \end{equation}
    Comparing equations \eqref{eqn:weightsgeneral} and \eqref{eqn:weightsgeneraldual}, we can find that the Proposition is true in this situation. 
    
    \item The case for $I_{v_3}=J_{v_3}$, and $|I_{v_4}\cap J_{v_4}|=r_{v_4}-1$ is the similar as the above situation and we omit it. 
\end{enumerate} 
\end{proof}
\subsection{Contributions of decorated graphs under quiver mutations}
In this section, we will prove Theorem \ref{thm:GWinvarmu} by comparing the contributions of decorated graphs of stable maps to $\mc Z$ and $\mc Z'$ and hence their Gromov-Witten invariants.\\
\textbf{ Proof of theorem \ref{thm:GWinvarmu}}
    Utilizing localization Theorem \ref{thm:localization}, we only have to prove that the decorated graphs and
    the contributions of each decorated graph for both sides are the same. 
    According to Proposition \ref{prop:torusfixedgraphmu}, 
    each graph of $\mc Z^1$ is mapped to $(\mc Z')^1$, 
    so it is natural to know that decorated graphs are the same.  
    
    For each decorated graph $\vec \Gamma$, whose underlying graph is $\Gamma$, the ingredients in Theorem \ref{thm:localization} are weights of flags $\mathbf w(\tau, \sigma)$, the valency $val(v)$ of each vertex,  and $h(\tau,d_\tau)$ for each edge. 
    Flags $(\tau, \sigma)\in \Upsilon$  are mapped to flags in $\varphi(\Upsilon)$, so $val(v)=val(\varphi(v))$. 
    The weights $\mathbf w(\tau, \sigma)$ are preserved by Proposition \ref{prop:weightsmu}. 

    Let $\tau$ be an edge, and $\sigma_1,\sigma_2$ be two vertices contained in $\tau$.
    Let $E(\sigma_1)=\{\tau_1,\ldots,\tau_{h}\}$ be the set of edges containing the point $\sigma_1$. 
    Without loss of generality, we assume that $\tau_{h}=\tau$. 
    For each $i$, we have defined $a_i$ in Equation \eqref{eqn:ai}.
   We don't have to compute all the $a_i$, but instead we only have to know the behavior of $a_i$ under quiver mutations. 
   By Proposition \ref{prop:torusfixedgraphmu}, $E(\varphi(\sigma_1))=\{\varphi(\tau_1),\ldots,\varphi(\tau_{h})\}$, and $\varphi(\tau)$ contains points $\varphi(\sigma_1)$ and $\varphi(\sigma_2)$.
   By Proposition \ref{prop:weightsmu}, weights are also preserved, so the Equation \eqref{eqn:ai} still holds. We have
   \begin{equation*}
        \mathbf w(\varphi(\tau_i'),\varphi(\sigma_2))=\mathbf w(\varphi(\tau_i),\varphi(\sigma_1))-a_i\mathbf w(\varphi(\tau), \varphi(\sigma_1)). 
    \end{equation*}
    Hence, the $a_i$ are also preserved, so the ingredients $h(e, d_e)$ are. 

   Furthermore, the map $\varphi: \mc Z^1\rightarrow (\mc Z')^1$ induces a map $\varphi_*: H_2(\mc Z)\rightarrow H_2(\mc Z')$. We claim that $\varphi_*$ coincides with $\phi_*$ in Proposition \ref{prop:transformeff}. In fact, suppose $\ell$ is a curve in $\mc Z$ corresponding to an edge of a graph in $\mc Z^1$. For an arbitrary $\gamma\in H^2(\mc Z)$, 
   we apply the Atiyah-Bott localization formula \cite{ATIYAH19841} on the following two integrals $\int_{[\ell]}\gamma$ and $\int_{\varphi_*([\ell])}\phi(\gamma)$.
    Due to \eqref{eqn:gammaresP} and Proposition \ref{prop:weightsmu}, we know the above two integrals are equal after comparing the graph contribution on both sides. Therefore,   for an arbitrary  $\gamma\in H^2(\mc Z)$, we have
    \begin{equation*}
         \int_{[\ell]}\gamma=\int_{\varphi_*([\ell])}\phi^*(\gamma)
    \end{equation*}
    By Proposition \ref{prop:transformeff}, $\varphi_*$ satisfies the same property \eqref{eqn:homologicalchange}. Due to the Poincaré duality, we have $\varphi_*=\phi_*$.
   The proof of Theorem \ref{thm:GWinvarmu} is completed. 

\section{Proof of the cluster algebra conjecture for A-type}\label{sec:proofclusteralgconj}
Until now, we have introduced the story of the algebraic side (cluster algebra) and the geometric side (quantum cohomology). In this section, we are going to establish a connection between them.

We consider the $A_n$ cluster algebra constructed from the quiver in Figure \ref{fig:Anqv}. 
Denote the cluster algebra by $\mathscr A_{n}$ and denote the initial cluster variables by $(x_i)_{i=1}^{n+1}$.

On the other hand, we consider the flag variety $Fl:=Fl(N_1,\ldots,N_{n+1})$ with $N_1<\ldots<N_{n+1}$ discussed in Example \ref{ex:An}, which admits a torus action $S=(\mathbb C^*)^{N_{n+1}}$.
Let $S_1\rightarrow S_2\rightarrow\ldots\rightarrow S_{n+1}=\mathcal O^{N_{n+1}}$ be a sequence of tautological bundles over the flag variety.
For each $l<k$, there is a short exact sequence
\begin{equation}\label{eqn:sesquotient}
    0\rightarrow S_l\rightarrow S_k\rightarrow S_k\slash S_l \rightarrow 0,
\end{equation}
where we formally use $S_k\slash S_l$ to denote the quotient bundle.

Consider the quantum cohomology ring $QH^*_S(Fl)[t]:=H_S^*(Fl)\otimes\mathbb Q[q]\otimes \mathbb Q[t]$ with an additional formal variable $t$ where $q=(q_1,\ldots,q_{n})$ are the K\"ahler variables. We view $QH^*_S(Fl)[t]$ as a $\mathbb Q$-algebra.

Introduce another set of variables $\{\xi_i\}_{i=1}^{n+1}$ and let 
\begin{equation*}\label{xi and q}
    q_i=(-1)^{N_i+N_{i+1}}\xi_{i-1}\xi_{i+1}^{-1}.
\end{equation*}

\begin{thm}\label{thm:maincluster2}
    Define a map
    \begin{equation*}
        \psi:\mathscr A_{n}\rightarrow QH^*_S(Fl)[t]
    \end{equation*}
    as follows.
    \begin{enumerate}
        \item The map $\psi$ sends each initial cluster variable $x_i$ to the equivariant Chern polynomial of the $i$-th tautological bundle: $\psi(x_i)=(-1)^{N_i}\xi_ic_t^S(S_i)$ for $i=1\ldots n+1$.
        \item For each non-initial cluster variable $x_v'$, let $(\mathbf Q=(Q_0,Q_1,W),\mathbf c')$ be one non-initial seed that contains the cluster variable $x_v'$.
        Once the seed is chosen, the decorations assigned to the quiver $(r_v)_{v\in Q_0}$ is determined. 
        As discussed in Lemma \ref{lem:decorationsofQinP}, each integer $r_v=N_{j}-N_{i}$ for some $1\leq i<j\leq n+1$.
        The map $\psi$ sends the cluster variable $c_v'$ to $(-1)^{N_j-N_i}\xi_j\xi_i^{-1}c_t^S(S_j\slash S_i)$.
    \end{enumerate}
    The map $\psi$ is a well-defined, injective $\mathbb Q$-algebra homomorphism.
\end{thm}
Since all relations in cluster algebra are from exchange relations \eqref{eqn:clusterrelation}, we need to prove that the map $\psi$ preserves cluster exchange relations. 

More precisely, let $(\mathbf Q=(Q_0,Q_1,W),\mathbf r=(r_v))$ and $(\mathbf Q'=(Q_0,Q_1',W'),\mathbf r'=(r_v'))$ be two decorated QPs in $\Omega^A_n$ related by a quiver mutation $\mu_v$.
Since the node $v$ has at most four adjacent nodes with two outgoing and two incoming, we denote the four nodes by $v_1,v_2,v_3,v_4$.  
According to the Lemma \ref{lem:decorationsofQinP}, we assume that $r_{v}=N_k-N_l$, $r_{v_1}=N_m-N_l$, $r_{v_2}=N_m-N_{k+1}$, $r_{v_3}=N_p-N_l$, and $r_{v_4}=N_k-N_{p+1}$  with $m>k>p\geq l$ as  shown in Figure \ref{fig:localbehavior} (a). The Figure \ref{fig:localbehavior} (b) is about the behavior of nodes adjacent to the node $v$ in QP $(\mathbf Q',\mathbf r')$. 
\begin{thm}\label{thm:quantumcohexrel}
\begin{enumerate}
    \item In the quantum cohomology ring of a flag variety $QH^*_S(Fl)[t]$, for $m> k> p\geq l\geq 0$, we have the following 
    \textit{quantum cohomological cluster exchange relations}. \begin{align}\label{eqn:quantumcohexrel0}
        c_t^S(S_m/S_{p+1})*c_t^S(S_k/S_l)&=
        c_t^S(S_m/S_l)*c_t^S(S_k/S_{p+1})\nonumber\\
        &+ \prod_{a=p+1}^{k}(-1)^{N_a+N_{a-1}}q_ac_t^S(S_m/S_{k+1})*c_t^S(S_p/S_l)\,.
    \end{align}
\item The map $\psi$ in Theorem \ref{thm:maincluster2} is an injective ring homomorphism if the quantum cohomological cluster exchange relations \eqref{eqn:quantumcohexrel0} hold.
\end{enumerate}
\end{thm}
\begin{proof}
    We only prove the part $(2)$ in this proof and leave the part (1) to the remaining of this section. 
    We denote cluster variables of Figure \ref{fig:localbehavior} (a) by 
    $\{x_u'\}_{u\in Q_0}$. 
    Performing a quiver mutation, we get a new cluster variable $\tilde x_v$, and the remaining ones are invariant. 
    In $\mathscr A_n$, the cluster exchange relation is 
    \begin{equation*}
        x_v'\tilde x_v=x_{v_1}'x_{v_4}'+x_{v_2}'x_{v_4}'.
    \end{equation*}
   In order to prove that $\psi$ is a well-defined ring homomorphism, we only have to prove that
   \begin{align*}
       &(-1)^{N_k-N_l+N_m-N_{p+1}}\xi_k\xi_{l}^{-1}c_t^S(S_k/S_l)*\xi_m\xi_{p+1}^{-1}c_t^S(S_m/S_{p+1})\nonumber\\
       &=(-1)^{N_m-N_l+N_k-N_{p+1}} \xi_m\xi_l^{-1}c_t^S(S_m/S_l)*\xi_k\xi_{p+1}^{-1}c_t^S(S_k/S_{p+1})\nonumber\\
       &+(-1)^{N_m-N_{k+1}+N_p-N_{l}}\xi_m\xi_{k+1}^{-1}c_t^S(S_m/S_{k+1})*\xi_p\xi_{l}^{-1}c_t^S(S_p/S_l)\,.
   \end{align*}
   Cancel the factor $(-1)^{N_k-N_l+N_m-N_{p+1}}\xi_k\xi_{l}^{-1}\xi_m\xi_{p+1}^{-1}$, and we get 
   \begin{align*}
       &c_t^S(S_k/S_l)* c_t^S(S_m/S_{p+1})=
       c_t^S(S_m/S_l)*
       c_t^S(S_k/S_{p+1})
       \nonumber\\
       &
       +(-1)^{N_k-N_{k+1}+N_p-N_{p+1}}\xi_p\xi_{p+1}\xi_k^{-1}\xi_{k+1}^{-1}c_t^S(S_m/S_{k+1})*c_t^S(S_p/S_l)\,,
   \end{align*}
  By letting $q_a=(-1)^{N_a-N_{a+1}}\xi_{a-1}\xi_{a+1}^{-1}$, the second term of the right-hand side is exactly $\prod_{a=p+1}^{k}(-1)^{N_a+N_{a-1}} q_ac_t^S(S_m/S_{k+1}) *c_t^S(S_p/S_l)$.
   Thus $\psi$ is a ring homomorphism. 

   The map $\psi$ is injective. 
   Suppose $\alpha=f(x_1, x_2, \ldots, x_{n+1}, \mathbf{x}')\in \mathscr A_{n}$ is in the kernel of $\psi$. We will prove that $\alpha=0$ in $\mathscr A_{n}$. 
   By the Laurent phenomenon of cluster variables \cite[Theorem 3.1]{clusteralg:1}, any cluster variable can be expressed as a Laurent polynomial of initial cluster variables.
  Therefore, there exists a polynomial $P$ and a monomial $Q$ in $\mathbb Z(x_1,x_2,\cdots,x_{n+1})$ such that
   $Q(x_1,\ldots, x_{n+1})\cdot \alpha=P(x_1,\ldots, x_{n+1})$.
   Since there is no zero divisor in $\mathscr A_{n}$, we only have to prove that $P=0$ in $\mathscr A_{n}$.
   We can assume $\alpha=P(x_1,\ldots, x_{n+1})$ without loss of generality. 
   Suppose $\alpha$ is expressed as the following finite sum
   \begin{equation*}
       \alpha=\sum_{\mathbf{i}=(i_1, \ldots, i_{n+1})\in \Z_{\geq 0}^{n+1}}a_{\mathbf{i}}\prod_{k=1}^{n+1}x_k^{i_k}
   \end{equation*}
   with $a_{\mathbf{i}}\in \Q$. We claim that $\psi(\alpha)=0$ implies $a_{\mathbf{i}}=0$ for arbitrary $\mathbf{i}$. We introduce the following partial order $\ge$ in the set $\Z^{n+1}$: for $\mathbf{i}$ and $\mathbf{i}'$ in $\Z^{n+1}$
   \begin{equation*}
     \mathbf{i}\ge \mathbf{i}' \quad\text{if} \quad \prod_{k=1}^{n+1}\xi_k^{i_k}\big/\prod_{k=1}^{n+1}\xi_k^{i_k'}=\pm\prod_{j=1}^{n}q_j^{\geq 0},
   \end{equation*}
   where $\xi$ and $q$ satisfy \eqref{xi and q}. Let $ CS_{\alpha}$ be the following finite set
   \begin{equation*}
       CS_{\alpha}:=\{\mathbf{i}\in  \Z_{\geq 0}^{n+1}|a_{\mathbf{i}}\neq 0\}.
   \end{equation*}
   It suffices to check $CS_{\alpha}$ is empty when $\psi(\alpha)=0$. Suppose $CS_{\alpha}$ is not empty, then there is a minimal element $\mathbf{i}_0$ in $ CS_{\alpha}$ with respect to $\ge$.
   We consider the term $\xi^{\mathbf{i}_0}t^{\mathbf{r}\cdot \mathbf{i}_0}$ in $\psi(\alpha)$ after the quantum reduction. 
   Since $\mathbf{i}_0$ is minimal, this term will not be affected by the quantum reduction, which means it is actually the highest order term of $t$ in  $\psi(a_{\mathbf{i}_0}\mathbf{x}^{{\mathbf{i}_0}})$. Therefore, the coefficient of $\xi^{\mathbf{i}_0}t^{\mathbf{r}\cdot \mathbf{i}_0}$ in $\psi(\alpha)$ is exactly $a_{\mathbf{i}_0}$. 
   Hence $\psi(\alpha)=0$ implies $a_{\mathbf{i}_0}=0$, contradicting to $a_{\mathbf{i}_0}\in CS_{\alpha}$. 
\end{proof}
\begin{rem}\label{rem:integersmkpl}
\begin{enumerate}
    \item As discussed in Section \ref{sec:quivervarietieswithmu}, if we take special values for integers $m,k,p,l$, we can get situations when the node $v$ has less adjacent nodes.
    Actually, the quantum cohomological exchange relations are also reduced to special formulas in such situations. 
   In Section \ref{sec:proofofmaintheorem}, we will discuss those special cases in detail and give the corresponding quantum cohomological cluster exchange relations.
\item We can formally assume that $l$    can be zero and let $N_0=0$. 
In such situations the integers are $r_{v}=N_k, r_{v_1}=N_m, r_{v_3}=N_p$.
\item 
Although the map $\psi$ is injective, it is not surjective as $\mathbb Q$-algebra homomorphism. 
For instance, the image doesn't contain pure cohomology classes in $QH^*_S(Fl)[t]$.
\item 
In the second author's talk at Shanghai Jiao Tong University about this project, Professor Fomin informed us that they included a similar equation with Equation \eqref{eqn:quantumcohexrel0} in their book \cite{fomin:cluster4-5}, within the context of a different ring representation of the quantum cohomology ring of flag varieties.
\end{enumerate}
\end{rem}
In the remaining part of this section, we will prove that the Theorem \ref{thm:quantumcohexrel} holds in $QH^*_S(Fl)[t]$.
In Section \ref{sec:abnonabforcoh}-\ref{sec:cohVG}, we will use the abelian/nonabelian correspondence for cohomology rings to study the cohomology and quantum cohomology rings of flag varieties. 
We want to emphasize that the ring representation
is different from that in  \cite{flag:quantumcohomology, qcohflag:CF99}. 
In Sections \ref{sec:abnonabqcoh}-\ref{sec:proofofmaintheorem}, we will give a detailed proof of the Theorem \ref{thm:quantumcohexrel}.

\subsection{Abelian/nonabelian correspondence}\label{Sec:ab/nonabcoh}\label{sec:abnonabforcoh}
\subsubsection{Abelian/nonabelian correspondence for cohomology rings}
We first review the abelian/nonabelian correspondence for cohomology rings, and refer to \cite{Ab_nab:EG,Ab_nab:Mar}.

Let $\mc X=V\sslash_\theta G$ be a fixed smooth projective variety, $T\subset G$ be the maximal torus and $W=N(T)/T$ be the Weyl group. The character $\theta\in \chi(G)$ is also in $\chi(T)$ by restriction, so we can define an abelian variety $V\sslash_{\theta}T$.
The variety $V\sslash_{\theta}G$ and $V\sslash_{\theta}T$ are related by the following diagram,
\begin{equation}\label{ab/nonab}
\begin{tikzcd}
V^s(G)/T \arrow[hookrightarrow]{r}{a} \arrow[]{d}{p}
  & V^s(T)/T \\
V^s(G)/G &
\end{tikzcd}
\end{equation}
where $a$ is injective and $p$ is a $G/T$-fibration.

Assume that there is a torus action $S$ on $V$ which commutes with the action of $G$, and then the torus acts on $V\sslash_\theta G$ and $V\sslash_\theta T$.

Let $\Phi=\{\rho\}$ be the set of positive roots of $G$, and $\{\mathbb C_{\rho}\}_{\rho\in \Phi}$ the one dimensional representations of $T$ with weights $\rho$.
These one dimensional representations define line bundles on $V\sslash_\theta T$, which are denoted by $L_{\rho}=V\times_T \mathbb C_{\rho}$. Denote their equivariant first Chern classes by $c^S_1(L_{\rho})\in H_S^*(V\sslash_\theta T)$.
Define $\omega:=\prod_{\rho\in \Phi}c^S_1(L_{\rho})$.
Then we have the following facts \cite[Section 3.1]{abelian/nonabelian:CKB}:
\begin{itemize}
    \item $p^*: H^*_S(V\sslash_\theta G)\rightarrow H^*_S(V^s(G)/T)^W$ is an isomorphism.
    \item There is a short exact sequence,
    \begin{equation}\label{eqn:ses}
        0\rightarrow \ker(\omega)\rightarrow H^*_S(V\sslash_\theta T)^W\xrightarrow[]{(p^*)^{-1}\circ a^*} H^*_S(V\sslash_\theta G)\rightarrow 0,
    \end{equation}
\end{itemize}
where $\ker(\omega)\subset H^*_S(V\sslash_\theta T)^W$ is the ideal generated by $\{\alpha \in H^*_S(V\sslash_\theta T)^W| \alpha\cup \omega=0\}$. 

\subsubsection{Abelian/nonabelian correspondence for quantum cohomology rings}
There is a quantum version of \eqref{eqn:ses} which we review by following the work \cite{abelian/nonabelian:CKB}.
Let $U$ be a subspace in $H^*(V\sslash_\theta T)$ which is a lift of $H^*(V\sslash_\theta G)$.
Denote the lifting of an element $\gamma$ by $\tilde\gamma$. 
Let $\mathbf t=\{t_i\}$
be the coordinates on $H^*(X\sslash_\theta G)$ by choosing a basis, and 
$\tilde{\mathbf t}=
\{\tilde{t}_i\}$ be the coordinates on $U$ with respect to the lifted basis. 
Let $(\tilde{\mathbf t}, \mathbf y)$ be an extension of $\tilde{\mathbf t}$ to the entire space $H^*(X\sslash_\theta T)$.

Let $N(V\sslash_\theta G)$ and $N(V\sslash_\theta T)$ be the Novikov rings for the two quotients. 
There is a natural specialization of Novikov variables 
\begin{equation*}
    \pi: N(V\sslash_\theta T)\rightarrow
N(V\sslash_\theta G)
\end{equation*}
 which takes into consideration that there are more curve classes in $V\sslash_\theta T$. 
We denote by  “$*_g$” the \emph{big} quantum product on $V\sslash_\theta T$ with the Novikov
variables being specialized via the map $\pi$ or on $V\sslash_\theta G$ depending on where the cohomology classes are.
The abelian/nonabelian correspondence for quantum cohomology ring is as follows \cite[Conjecture 3.7.1]{abelian/nonabelian:CKB}.

\begin{conjecture}(\cite[Conjecture 3.7.1]{abelian/nonabelian:CKB})
Given $\gamma, \gamma'\in H^*(V\sslash_\theta G)$, 
there are classes
$\zeta, \zeta'\in U\otimes_\mathbb{C} N(V\sslash_\theta G)[[\tilde{\mathbf t}]]$, 
uniquely determined by relations $\zeta *_g \omega = \tilde{\gamma}\cup\omega$ and
 $\zeta' *_g \omega = \tilde{\gamma}'\cup\omega$.
 There is an equality
\begin{equation}\label{ab-nab for Qcoh}
  (\widetilde{(\gamma *_g \gamma'
)} \cup \omega)(\mathbf t) = \zeta*_g \zeta'*_g \omega(\tilde{\mathbf t}, 0) = (\zeta*_g (\widetilde{\zeta'}\cup \omega))(\tilde{\mathbf t}, 0),
\end{equation}
with an explicit change of variables $\tilde{\mathbf t}=\tilde{\mathbf t}(\mathbf t)$.
\end{conjecture}

The conjecture \ref{ab-nab for Qcoh} is proved in the case of flag varieties $Fl$ in \cite[Theorem 4.1.1]{abelian/nonabelian:CKB}, and is proved under the assumptions that the GIT quotient $V\sslash_\theta G$ is Fano of index greater than or equal to two and the equivariant cohomology ring $H^*_R(V\sslash_\theta G)$ is generated by divisor classes after localization in \cite[Corrolary 6.4.5]{abeliannonabelian:Webb} and \cite[Theorem 4.3.6]{abelian/nonabelian:CKB}. 
See \cite{rimhookrule} for an interesting application of abelian/nonabelian correspondence. 
We will only use the restriction of \eqref{ab-nab for Qcoh} to \emph{small} quantum cohomology ring in our case  $V\sslash_\theta G=Fl=Fl(N_1,\cdots,N_{n+1})$. 
We will use $*$ to denote the small quantum products in $V\sslash_\theta G$ and $V\sslash_\theta T$ with Novikov variables being specialized via the map $\pi$. 
\begin{prop}\label{prop:ab-nab for small Qcoh}
In small quantum cohomology ring $QH^*(Fl)$, the abelian/nonabelian correspondence is reduced to the following equation, 
\begin{equation}\label{ab-nab for small Qcoh}
 ( \widetilde{\gamma * \gamma'
}) \cup \omega
= \xi* \xi'* \omega = 
\xi* (\widetilde{\gamma'}\cup \omega),
\end{equation}
Notice that $\gamma * \gamma'$ is uniquely determined in \eqref{ab-nab for small Qcoh} according to the exact sequence \eqref{ab/nonab}.
\end{prop}
\begin{proof}
  The variables change $\widetilde{t}=\widetilde{t}(t)$ in \eqref{ab-nab for Qcoh} is the only thing we need to care about. This change is trivial after being restricted to the small phase subspace since the abelian/nonabelian correspondence for small J-functions holds for $V\sslash_\theta G$ without parameter change, see \cite[Theorem 4.3.6]{abelian/nonabelian:CKB} or \cite[Section 4.2]{rimhookrule}.
\end{proof}

\subsection{Quantum cohomology of the abelianization of a flag variety}\label{sec:qcohVT}
    Consider the flag variety $Fl:=Fl(N_1,\ldots, N_{n+1})$, which can be written as $V\sslash_\theta G$ as in Example \ref{ex:An}.
    The gauge group is $G=\prod_{i=1}^{n}GL(N_i)$,
    the maximal torus is $T=\prod_{i=1}^{n}(\mathbb C^*)^{N_i}$, and the Weyl group $W=\prod_{i=1}^{n}S_{N_i}$ where $S_{N_i}$ is the permutation group of $N_i$ elements.
    The abelianization $V\sslash_\theta T$ is a toric variety, which can be viewed as a tower of products of projective spaces, $(\mathbb P^{N_2-1})^{N_1}\rightarrow (\mathbb P^{N_3-1})^{N_2}\rightarrow\ldots\rightarrow (\mathbb P^{N_{n+1}-1})^{N_{n}}$. Let $[\ell_{k, i}]$ denote the rational curve class lying inside the $i$-th $\mathbb P^{N_{k+1}-1}$-fiber of the $k$-th tower $(\mathbb P^{N_{k+1}-1})^{N_k}$.
    Then $N(V\sslash_\theta T)=\bigoplus_{k, i}\mathbb{Z}_{\geq 0}[\ell_{k, i}]$. Set the Novikov variable $q_{k, i}:=q^{[\ell_{k,i}]}$.

    The natural torus action on $V$ that commutes with $G$ is $(\mathbb C^*)^{N_{n+1}}$ coming from the frozen node. Denote the equivariant parameters by $\lambda_1,\ldots,\lambda_{N_{n+1}}$.

    Consider characters $\chi_{j}^{(i)}:T\rightarrow \mathbb C^*$ which map $(\vec t^i)_{i=1}^{n}\in T$ to the $j$-th component of the $\vec t^i=\{t^i_1,\cdots, t^i_{N_i}\}\in (\mathbb C^*)^{N_i}$.
    These characters define a set of line bundles $L_{\chi^{(i)}_j}$ over $V\sslash_\theta T$. 
    Denote
    \begin{equation}\label{eqn:chernroots}
        x^{(i)}_j= c_1^S(L_{\chi^{(i)}_j})\in H^*_S(V \sslash_\theta T), \text{ and } \vec x^{(i)}=\{x^{(i)}_1,\cdots x^{(i)}_{N_i}\}\,.
    \end{equation}
    Actually, for each $i=1,\cdots,n,j=1,\cdots,N_i$,  the non-equivariant limit of $x^{(i)}_j$ is the hyperplane class of the $j$-th projective space of $(\mathbb P^{N_{i+1}-1})^{N_i}$. 

It is a classical result for toric variety that $\{x^{(i)}_j\}$ generate $ H^*_S(V\sslash_\theta T)$ and satisfy the following relations,
    \begin{equation}\label{eqn:cohrelation}
        \prod_{l=1}^{N_{i+1}}(x^{(i)}_j-x^{(i+1)}_l)=0,\,\,j=1,\ldots,N_i;\,\,i=1,\ldots,n,
    \end{equation}
    where we $x^{(n+1)}_j=\lambda_j$ for $j=1,\ldots,N_{n+1}$.
Let $I^{ab}$ be the ideal generated by the relations $\eqref{eqn:cohrelation}$, and then
\begin{equation*}
    H^*_S(V\sslash_\theta T)=\mathbb Q[\{x^{(i)}_{j}\}_{i=1,\cdots,n+1,j=1,\cdots,N_i}]\slash I^{ab}.
\end{equation*}

 The quantum cohomology rings of toric varieties have been investigated by Batyrev \cite{toric}. We can directly write down the quantum cohomology of $V\sslash_\theta T$ as follows. 
\begin{prop}\label{qcoh toric}
 The quantum cohomology $QH^*_S(V\sslash_\theta T)$ is a $\mathbb Q[\lambda_1,\ldots,\lambda_{N_{n+1}}][q_{k, i}]$-module generated by a set
\begin{equation}\label{standard basis}
 \left\{\prod_{k=1}^n \prod_{j=1}^{N_k} \left(x^{(k)}_j\right)^{a_{kj}}\big|0\leq a_{kj}\leq N_{k+1}-1\right\}
\end{equation}
with the following relations
\begin{align}\label{eqn:qcohrelation}
  \left(x^{(k)}_i\right)^{N_{k+1}-1}* x^{(k)}_i&=\left(x^{(k)}_i\right)^{N_{k+1}}-\prod_{j=1}^{N_{k+1}} \left(x^{(k)}_i-x^{(k+1)}_j\right) +q_{k,i}\prod_{l=1}^{N_{k-1}}\left(x^{(k-1)}_l-x^{(k)}_i\right)\nonumber\\
&=-\sum_{l=0}^{N_{k+1}-1}  e_{N_{k+1}-l}(-\vec x^{(k+1)})\left(x^{(k)}_i\right)^{l}+q_{k,i}\prod_{l=1}^{N_{k-1}}\left(x^{(k-1)}_l-x^{(k)}_i\right)\,.
\end{align}
where $e_i(\vec x^{(k+1)})$ is the $i$-th elementary symmetric polynomial with variable set $\vec x^{(k+1)}$.
\end{prop}
From \eqref{eqn:cohrelation} and \eqref{eqn:qcohrelation}, one can find that $*|_{q\rightarrow 0}=\cup$. 
For any $\alpha, \beta\in H^*_S(V\sslash_\theta T)$, we call $\alpha\cup \beta$ the \emph{classical part} and $\alpha* \beta-\alpha\cup \beta$ the \emph{quantum correction} of $\alpha* \beta$. 
Then we know from \eqref{eqn:qcohrelation} that the quantum correction of $\left(x^{(k)}_i\right)^{N_{k+1}-1}* x^{(k)}_i$ is $q_{k,i}\prod_{l=1}^{N_{k-1}}\left(x^{(k-1)}_l-x^{(k)}_i\right)$, and the classical part is $-\sum_{l=0}^{N_{k+1}-1}  e_{N_{k+1}-l}(-\vec x^{(k+1)})\left(x^{(k)}_i\right)^{l}$.

  From Proposition \ref{qcoh toric}, we can easily get that the quantum correction of $\alpha* \beta$ is nontrivial only if the sum of the degree of $x^{(k)}_i$ in $\alpha$ and $\beta$ is greater than or equal to $ N_{k+1}$ for some $k$ and $i$.

\subsection{Quantum cohomology rings of flag varieties}\label{sec:cohVG}
We first introduce Schur polynomials. For arbitrary distinct $N$ variables $(x_1,\ldots,x_N)$, the
Vandermonde determinant of $(x_1,\ldots,x_N)$ is defined by
\begin{equation}\label{Vandermonde}
    \prod_{1\leq i<j\leq N}(x_i-x_j)=
    \det
    \begin{bmatrix}
        x_1^{N-1}&x_2^{N-1}&\ldots&x_{N}^{N-1}\\
        x_1^{N-2}&x_{2}^{N-2}&\ldots&x_{N}^{N-2}\\
        \vdots&\vdots&\ddots&\vdots\\
        1&1&\ldots&1
    \end{bmatrix}.
\end{equation}
For each partition ${\eta}=(\eta_1,\ldots,\eta_N)$,
such that $\eta_1\geq \eta_2\geq \cdots\geq \eta_N$, the Schur polynomial $s_{\eta}(x_1,\ldots,x_N)$ is a symmetric polynomial of degree $\abs{\eta}=\sum_{i=1}^{N}\eta_i$ defined as follows,
\begin{equation}\label{Schur}
    s_{\eta}(x_1,\ldots,x_N)\prod_{1\leq i<j\leq N}(x_i-x_j)=
    \det
    \begin{bmatrix}
        x_1^{\eta_1+N-1}&x_2^{\eta_1+N-1}&\cdots&x_{N}^{\eta_1+N-1}\\
        x_1^{\eta_2+N-2}&x_{2}^{\eta_2+N-2}&\cdots&x_{N}^{\eta_2+N-2}\\
        \vdots&\vdots&\ddots&\vdots\\
    x_1^{\eta_N}&x_2^{\eta_N}&\cdots&x_N^{\eta_N}
    \end{bmatrix}.
\end{equation}
In particular, when $\eta_1=\cdots=\eta_i=1$, and $\eta_{i+1}=\cdots=\eta_N=0$, the corresponding Schur polynomial is the degree $i$ elementary symmetric polynomial $e_i(x_1,\ldots,x_N)=\sum_{1\leq a_1<a_2<\cdots<a_i\leq N}\prod_{l=1}^ix_{a_l}$.

Another special partition is when $\eta_1=i$ and $\eta_2=\ldots=\eta_N=0$. Denote such a Schur polynomial by $h_i$. Then we have
\begin{equation*}
    h_i(x_1,\cdots,x_N)=\sum_{j_1+\ldots+j_N=i}\prod_{l=1}^{N}{x_l^{j_l}}.
\end{equation*}
We denote $\omega_m=\prod_{1\leq i<j\leq N}(x_i^{(m)}-x_j^{(m)})$ where $x^{(m)}_i$ are Chern roots of line bundles defined in \eqref{eqn:chernroots}.

Consider the equivariant cohomology ring of $Fl$.
 Let $\omega=\prod_{k=1}^{n}\omega_k=\prod_{k=1}^{n}\prod_{1\leq i<j \leq N_k}(x^{(k)}_{i}-x^{(k)}_{j})$. By the short exact sequence \eqref{eqn:ses}, the equivariant cohomology ring of $Fl$ is
\begin{equation*}
   H^*_S(Fl)\iso \mathbb Q[\{x^{(i)}_{j}\}_{i=1,\cdots,n+1,j=1,\cdots,N_i}]^W\slash I
\end{equation*}
where $$I=\{f\in \mathbb Q[\lambda_1,\ldots,\lambda_{N_{n+1}}][\{x^{(i)}_{j}\}_{i=1,\cdots,n+1,j=1,\cdots,N_i}]^W \big{|}\, f\cup \omega\in I^{ab}\}.$$

According to the definition of $S_k$ and $x^{(k)}_i$, 
one lift of $c_t^S(S_k)$ in  $H^*_S(V\sslash_\theta T)[t]$ is
\begin{equation}\label{eqn:lifting}
   \widetilde{ c_t^S(S_k)}=\prod_{j=1}^{N_k}(t-x^{(k)}_j)=\sum_{l=0}^{N_k}t^{N_k-l}e_l(-\vec x^{(k)}).
\end{equation}
In $H^*_S(Fl)$, the Chern polynomial of the quotient bundle $S_l/S_k$ which is defined in \eqref{eqn:sesquotient} satisfies the following relation
\begin{equation}\label{eqn:tautquot}
    c_t^S(S_k)\cup c_t^S(S_l\slash S_k)=c_t^S(S_{l}).
\end{equation}
We can use this relation to get an explicit expression for the lift of $c_t^S(S_l\slash S_k)$.

\begin{lem}\label{lem:vanishingformulaincoh}
  In $H^*_S(V\sslash_\theta T)$, we have, for $l>k$
\begin{equation}\label{eqn:cohrelation2}
  \prod_{j=1}^{N_{l}} \left(x^{(k)}_i-x^{(l)}_j\right)\prod_{m=k+1}^{l-1} \omega_m=0, \,\, i=1, 2, \ldots, N_k.
\end{equation}
\end{lem}
\begin{proof}
  We use an induction argument on the difference $l-k$. For $l-k=1$, the equation \eqref{eqn:cohrelation2} holds by \eqref{eqn:cohrelation}. 
  We assume \eqref{eqn:cohrelation2} holds for $l-k\leq p-1$. 
  For $l-k=p$, we consider
\begin{equation}\label{eqn:Amatrix}
A:=
    \det
    \begin{bmatrix}
       \prod_{j=1}^{N_l}\left(x^{(k)}_i-x^{(l)}_j\right)&\prod_{j=1}^{N_l}\left(x_1^{(l-1)}-x^{(l)}_j\right)&\cdots&\prod_{j=1}^{N_l}\left(x_{N_{l-1}}^{(l-1)}-x^{(l)}_j\right)\\
        \left(x_i^{(k)}\right)^{N_{l-1}-1}&\left(x_1^{(l-1)}\right)^{N_{l-1}-1}&\cdots&\left(x_{N_{l-1}}^{(l-1)}\right)^{N_{l-1}-1}\\
        \vdots&\vdots&\ddots&\vdots\\
    x_i^{(k)}&x_1^{(l-1)}&\cdots&x_{N_{l-1}}^{(l-1)}\\
   1& 1&\cdots& 1\\
    \end{bmatrix}
\cdot \prod_{m=k+1}^{l-2} \omega_m
\end{equation}
Then expanding along the first row we have
\begin{equation*}
  A=\sum_{a=1}^{N_l}e_a(-\vec x^{(l)})h_{N_l-N_{l-1}-a}(x^{(k)}_i, \vec x^{(l-1)})\prod_{j=1}^{N_{l-1}}\left(x^{(k)}_i-x^{(l-1)}_j\right)\prod_{m=k+1}^{l-1} \omega_m
\end{equation*}
which is zero by induction.

On the other hand, expanding \eqref{eqn:Amatrix} along the first row, by \eqref{eqn:cohrelation}, \eqref{Vandermonde} we have
\begin{equation*}
  A= \prod_{j=1}^{N_l}\left(x^{(k)}_i-x^{(l)}_j\right)\prod_{m=k+1}^{l-1} \omega_m
\end{equation*}
So \eqref{eqn:cohrelation2} holds for $l-k=p$.
\end{proof}

\begin{lem}
  For $l>k$, one lift of $c_t^S(S_l\slash S_k)$ in  $H^*_S(V\sslash_\theta T)[t]$ can be written as follows,
\begin{equation}\label{eqn:lifting for quotient}
   \widetilde{ c_t^S(S_l\slash S_k)}=\sum_{a=0}^{N_l-N_k}t^{a} \left[\sum_{b=0}^{N_l-N_k-a}e_b(-\vec x^{(l)})h_{N_l-N_k-a-b}(\vec x^{(k)})\right].
\end{equation}
\end{lem}
\begin{proof}
  By \eqref{eqn:ses} and \eqref{eqn:tautquot}, it suffices to check
\begin{equation*}
  \widetilde{c_t^S(S_k)}\cup\sum_{a=0}^{N_l-N_k}t^{a} \left[\sum_{b=0}^{N_l-N_k-a}e_b(-\vec x^{(l)})h_{N_l-N_k-a-b}(\vec x^{(k)})\right]\cup\omega= \widetilde{c_t^S(S_l)}\cup\omega.
\end{equation*}

By simply exchanging indices of summations,
the left hand side can be written as 
\begin{equation}\label{eqn:quotientbundleeqn1}
    \prod_{i=1}^{N_k}(t-x^{(k)}_i)\sum_{b=0}^{N_l-N_k}e_b(-\vec x^{(l)})\left[\sum_{a=0}^{N_l-N_k-b}t^{a} h_{N_l-N_k-a-b}(\vec x^{(k)})\right]\omega\,.
\end{equation}
The formula inside the bracket is actually $h_{N_l-N_k-b}(t, \vec x^{(k)})$.
We view $t$ as $x_0^{(k)}$, and then $\prod_{i=1}^{N_k}(t-x^{(k)}_i)\omega_k$ is the Vandermonde formula of variables $(x^{(k)}_i)_{i=0}^{N_k}$.
The formula \eqref{eqn:quotientbundleeqn1} can then be written as follows,
\begin{equation}\label{eqn:quotientbundleeqn2}
    \sum_{b=0}^{N_l-N_k}e_b(-\vec x^{(l)})\det
    \begin{bmatrix}
      t^{N_l-b}&\left(x_1^{(k)}\right)^{N_{l}-b}&\cdots&\left(x_{N_k}^{(k)}\right)^{N_{l}-b}\\
       t^{N_{k}-1}&\left(x_1^{(k)}\right)^{N_{k}-1}&\cdots&\left(x_{N_k}^{(k)}\right)^{N_{k}-1}\\
        \vdots&\vdots&\ddots&\vdots\\
   1& 1&\cdots& 1\\
    \end{bmatrix}\prod_{m\neq k}\omega_m
\end{equation}
Notice that the determinant is zero once $b$ is greater than $N_l-N_k$, so we can change the summation index from $0$ to $N_l$.
Multiply $e_b(-\vec x^{(l)})$ and $\prod_{m\neq k}\omega_m$ to the first row of the determinant. 
Then entries in the 1st row and $j$-th column for $j\geq 2$ are $\prod_{i=1}^{N_l}(x^{(k)}_{j-1}-x^{(l)}_i)\omega_{m\neq k}$ which are all zero by Lemma \ref{lem:vanishingformulaincoh}. 
The formula \eqref{eqn:quotientbundleeqn2} becomes 
\iffalse
becomes
\begin{equation*}
    \det
    \begin{bmatrix}
     \prod_{i=1}^{N_l}\left(t-x_i^{(l)}\right)& 0&\cdots&0\\
       t^{N_{k}-1}&\left(x_1^{(k)}\right)^{N_{k}-1}&\cdots&\left(x_{N_k}^{(k)}\right)^{N_{k}-1}\\
        \vdots&\vdots&\ddots&\vdots\\
   1& 1&\cdots& 1\\
    \end{bmatrix}\prod_{m\neq k}\omega_m
\end{equation*}
which is 
\fi
$\prod_{i=1}^{N_l}(t-x^{(l)}_i)\omega=\widetilde{c_t^S(S_l)}\cup \omega$.
\end{proof}

For our purpose, we rewrite the lifting $\widetilde{c_t^S(S_l\slash S_k)}\omega$ as a linear combination of elements in \eqref{standard basis}.
For $\vec x=(x_1, x_2, \ldots, x_n)$ and $\vec m=(m_1, m_2,\ldots, m_n)\in \mathbb{Z}^n$, denote
\begin{equation*}
D^{\vec m}(\vec x):=
    \det
    \begin{bmatrix}
      x_1^{m_n}&x_2^{m_n}&\cdots&x_n^{m_n}\\
         x_1^{m_{n-1}}&x_2^{m_{n-1}}&\cdots&x_n^{m_{n-1}}\\
        \vdots&\vdots&\ddots&\vdots\\
 x_1^{m_1}&x_2^{m_1}&\cdots&x_n^{m_1}\\
    \end{bmatrix}\,.
\end{equation*}

\begin{lem}
  In $H^*_S(V\sslash_\theta T)[t]$,
\begin{equation}\label{eqn:standard for quotient}
\widetilde{c_t^S(S_l\slash S_k)}\omega=\sum_{a=0}^{N_l-N_k}t^{a} \left[\sum_{b=0}^{N_l-N_k-a}e_b(-\vec x^{(l)})\sum_{0\leq m_i< N_{i+1}; i=k,\ldots, l-1}(-1)^\bullet D^{\vec m_k}(\vec x^{(k)})D^{\vec m_{k+1}}(\vec x^{(k+1)})\cdots D^{\vec m_{l}}(\vec x^{(l)})\right]\prod_{j<k \text{ or } j>l} \omega_j.
\end{equation}
where
\begin{align*}
  &(-1)^\bullet=(-1)^{N_{k+1}+\cdots+N_{l}-m_k-\cdots-m_{l-1}-l+k}\,,\nonumber\\
  &\vec m_k=(0, 1, \ldots, N_k-2, m_k),\\
  &\vec m_{i+1}=(0, 1, \ldots, m_i-1, m_i+1,\ldots, N_{i+1}-1, m_{i+1}),\quad i=k,\ldots, l-2,\\
  &\vec m_{l}=(0, 1, \ldots, m_{l-1}-1, m_{l-1}+1,\ldots, N_{l}-1, N_l-a-b-1).
\end{align*}
\end{lem}
\begin{proof}
By \eqref{eqn:lifting for quotient}, we have
  \begin{align}\label{eqn1}
   \widetilde{c_t^S(S_l\slash S_k)}\omega=\sum_{a=0}^{N_l-N_k}t^{a} \left[\sum_{b=0}^{N_l-N_k-a}e_b(-\vec x^{(l)})D^{(0, 1, \ldots, N_k-2, N_l-a-b-1)}(\vec x^{(k)})\right]\prod_{j\neq k}\omega_j.
\end{align}
Consider the Laplace expansion of the following determinant down the first column,
\begin{align*}
    &\det
    \begin{bmatrix}
      \left(x_i^{(k)}\right)^{N_l-a-b-1}& \left(x_1^{(k+1)}\right)^{N_l-a-b-1}&\cdots& \left(x_{N_{k+1}}^{(k+1)}\right)^{N_l-a-b-1}\\
        \left(x_i^{(k)}\right)^{N_{k+1}-1}&\left(x_1^{(k+1)}\right)^{N_{k+1}-1}&\cdots&\left(x_{N_{k+1}}^{(k+1)}\right)^{N_{k+1}-1}\\
        \vdots&\vdots&\ddots&\vdots\\
 1&1&\cdots&1\\
    \end{bmatrix}\nonumber\\
=&\left(x_i^{(k)}\right)^{N_l-a-b-1}\omega_{k+1}+\sum_{m=0}^{N_{k+1}-1}(-1)^{N_{k+1}-m} \left(x_i^{(k)}\right)^{m}\cdot D^{(0, 1,\ldots, m-1, m+1, \ldots, N_{k+1}-1, N_l-a-b-1)}(\vec x^{(k+1)})
\end{align*}
Notice that the left-hand side of the above equation vanishes by  \eqref{eqn:cohrelation}, so does the right-hand side. 
Then we have
\begin{equation*}
  \left(x_i^{(k)}\right)^{N_l-a-b-1}\omega_{k+1}=\sum_{m_k=0}^{N_{k+1}-1}(-1)^{N_{k+1}-m_k-1} \left(x_i^{(k)}\right)^{m_k}\cdot D^{(0, 1,\ldots, m_k-1, m_k+1, \ldots, N_{k+1}-1, N_l-a-b-1)}(\vec x^{(k+1)})\,.
\end{equation*}
Substituting the above formula into the first row of $D^{(0, 1, \ldots, N_k-2, N_l-a-b-1)}(\vec x^{(k)})\prod_{j\neq k}\omega_j$, we get
 \begin{align*}
   &D^{(0, 1, \ldots, N_k-2, N_l-a-b-1)}(\vec x^{(k)})\cdot\prod_{j\neq k}\omega_j\nonumber\\
=&\sum_{m_k=0}^{N_{k+1}-1}(-1)^{N_{k+1}-m_k-1}D^{\vec m_k}(\vec x^{(k)})D^{(0, 1,\ldots, m_k-1, m_k+1, \ldots, N_{k+1}-1, N_l-a-b-1)}(\vec x^{(k+1)})\prod_{j\neq k, k+1}\omega_j
\end{align*}
Perform a similar argument on $D^{(0, 1,\ldots, m_k-1, m_k+1, \ldots, N_{k+1}-1, N_l-a-b-1)}(\vec x^{(k+1)})$ and iterate this procedure. We finally obtain
 \begin{equation}\label{eqn2}
   D^{(0, 1, \ldots, N_k-2, N_l-a-b-1)}(\vec x^{(k)})\prod_{j\neq k}\omega_j=\sum_{\substack{0\leq m_i< N_{i+1};\\ i=k,\ldots, l-1}}(-1)^\bullet D^{\vec m_k}(\vec x^{(k)})D^{\vec m_{k+1}}(\vec x^{(k+1)})\cdots D^{\vec m_{l}}(\vec x^{(l)})\prod_{j<k \text{ or } j>l} \omega_j.
\end{equation}
Finally, we only need to substitute \eqref{eqn2} into \eqref{eqn1}.
\end{proof}
Although, the right-hand side of \eqref{eqn:standard for quotient} looks complicated,
one can check that it is a linear combination of elements in the set \eqref{standard basis}. 
In particular, the degree of $x^{(i)}_j$ is less than $N_{i+1}$ for any $i, j$.
\subsection{Fundamental relations in quantum cohomology rings of flag varieties }\label{sec:abnonabqcoh}
Recall that $Fl$ is a tower of Grassmannian $Gr(N_1, N_2)\rightarrow Gr(N_2, N_3)\rightarrow \cdots \rightarrow Gr(N_n, N_{n+1})$ (cf. \eqref{Gr-fiber}). 
Let $[\ell_{k}]$ denote the rational curve class lying inside the $k$-th tower $Gr(N_k,N_{k+1})$.  
Then the Novikov ring of $Fl$ is $N(Fl)=\bigoplus_{k}\mathbb{Z}_{\geq 0}[\ell_{k}]$. Set the Novikov variable $q_{k}:=q^{[\ell_k]}$. By \cite{rimhookrule} the specialization of Novikov variable in Proposition \ref{prop:ab-nab for small Qcoh} is
\begin{equation}\label{Novikov-spe}
 q_{k,i}\mapsto (-1)^{N_k-1}q_k,
\end{equation}
where $q_{k, i}$ are the Novikov variables of $V\sslash_{\theta} T$.

\begin{prop}\label{lem:special}
  In $QH^*_S(Fl)[t]$, we have for $k<l$,
\begin{equation}\label{eqn:special1}
  c^S_t(S_k)* c^S_t(S_l\slash S_k)=c^S_t(S_l)+(-1)^{N_k+N_{k-1}}q_k c^S_t(S_l\slash S_{k+1})* c^S_t(S_{k-1}).
\end{equation}
\end{prop}
\begin{proof}
We only have to prove that the quantum correction of the quantum product is $(-1)^{N_k+N_{k-1}}q_k c^S_t(S_l\slash S_{k+1})* c^S_t(S_{k-1})$ by utilizing
    Proposition \ref{prop:ab-nab for small Qcoh}. 
Z  For $\gamma=c^S_t(S_k)$ and the lifting $\widetilde \gamma$ in \eqref{eqn:lifting}, we need to find $\zeta\in H_S^*(V\sslash_{\theta} T)^W$ such that $\zeta* \omega=\widetilde{\gamma}\cup \omega$. 
  Actually, we have $\widetilde{\gamma}* \omega=\widetilde{\gamma}\cup \omega$ by degree reason. Therefore,
\begin{equation*}
  \zeta= \widetilde{\gamma}=\prod_{j=1}^{N_k}(t-x^{(k)}_j)=\sum_{l=0}^{N_k}t^{N_k-l}e_l(-\vec x^{(k)}).
\end{equation*}

For $\gamma'= c^S_t(S_l\slash S_k)$,  by \eqref{eqn:standard for quotient} we have
\begin{equation*}
  \widetilde{\gamma'}\cup\omega=\sum_{a=0}^{N_l-N_k}t^{a} \left[\sum_{b=0}^{N_l-N_k-a}e_b(-\vec x^{(l)})\sum_{\substack{0\leq m_i< N_{i+1};\\ i=k,\ldots, l-1}}(-1)^\bullet D^{\vec m_k}(\vec x^{(k)})D^{\vec m_{k+1}}(\vec x^{(k+1)})\cdots D^{\vec m_{l}}(\vec x^{(l)})\right]\prod_{j<k \text{ or } j>l} \omega_j.
\end{equation*}

The quantum correction $\xi* (\widetilde{\gamma'}\cup\omega)$ in $QH^*(V\sslash_{\theta} T)$ under the specialization \eqref{Novikov-spe} can be computed as follows.
The quantum correction only happens  when $m_k=N_{k+1}-1$, 
\begin{align}\label{principle term}
  &\sum_{l=0}^{N_k}t^{N_k-l}e_l(-\vec x^{(k)})* \sum_{a=0}^{N_l-N_k}t^{a} \left(\sum_{b=0}^{N_l-N_k-a}e_b(-\vec x^{(l)}) \det
 \begin{bmatrix}
     \left(x_1^{(k)}\right)^{N_{k+1}-1}&\cdots&\left(x_{N_{k}}^{(k)}\right)^{N_{k+1}-1}\\
      \left(x_1^{(k)}\right)^{N_{k}-2}&\cdots&\left(x_{N_k}^{(k)}\right)^{N_{k}-2}\\
        \vdots&\ddots&\vdots\\
   1&\cdots& 1\\
    \end{bmatrix}\cup\right.\nonumber\\
&\left.\sum_{\substack{0\leq m_i< N_{i+1};\\ i=k+1,\ldots, l-1}}(-1)^\bullet D^{\vec m_{k+1}}(\vec x^{(k+1)})\cdots D^{\vec m_{l}}(\vec x^{(l)})\right)\prod_{j<k \text{ or } j>l} \omega_j.
\end{align}
We consider one term in \eqref{principle term}
\begin{align}\label{term1}
  \sum_{l=0}^{N_k}t^{N_k-l}e_l(-\vec x^{(k)})* \det\begin{bmatrix}
     \left(x_1^{(k)}\right)^{N_{k+1}-1}&\cdots&\left(x_{N_{k}}^{(k)}\right)^{N_{k+1}-1}\\
      \left(x_1^{(k)}\right)^{N_{k}-2}&\cdots&\left(x_{N_k}^{(k)}\right)^{N_{k}-2}\\
        \vdots&\ddots&\vdots\\
   1&\cdots& 1\\
    \end{bmatrix}
\end{align}
which is the same as the quantum correction of the following term, 
  \begin{align}\label{eqn:computquantumcorrection}
  &\sum_{l=0}^{N_k}t^{N_k-l}e_l(-\vec x^{(k)})* \sum_{i=0}^{N_{k+1}-N_k}t^i\det\begin{bmatrix}
     \left(x_1^{(k)}\right)^{N_{k+1}-1-i}&\cdots&\left(x_{N_{k}}^{(k)}\right)^{N_{k+1}-1-i}\\
      \left(x_1^{(k)}\right)^{N_{k}-2}&\cdots&\left(x_{N_k}^{(k)}\right)^{N_{k}-2}\\
        \vdots&\ddots&\vdots\\
   1&\cdots& 1\\
    \end{bmatrix}\,.
    \end{align}
    The formula in \eqref{eqn:computquantumcorrection} is equal to 
\begin{align}\label{eqn:quantumprodeqn}
&=\prod_{j=1}^{N_k}(t-x^{(k)}_j)*\prod_{1\leq i<j\leq N_k}(x^{(k)}_i-x^{(k)}_j)h_{N_{k+1}-N_k}(t, \vec x^{(k)}) \text{ (one changes $\cup$ to $*$ by degree reason)}\nonumber
    \\
    &=\det
    \begin{bmatrix}
   t^{N_{k+1}}&   \left(x_1^{(k)}\right)^{N_{k+1}}&\cdots&\left(x_{N_{k}}^{(k)}\right)^{N_{k+1}}\\
   t^{N_{k}-1}&   \left(x_1^{(k)}\right)^{N_{k}-1}&\cdots&\left(x_{N_k}^{(k)}\right)^{N_{k}-1}\\
    \vdots&    \vdots&\ddots&\vdots\\
  1& 1&\cdots& 1\\
    \end{bmatrix}^* \text{(the upper  $*$ means all product  is $*$)}
\end{align}
The quantum correction of the above formula \eqref{eqn:quantumprodeqn} is 
\begin{align}\label{eqn:quantumproductcomp}
    -\sum_{i=1}^{N_k} \left(\left(x_i^{(k)}\right)^{*N_{k+1}}-\left(x_i^{(k)}\right)^{N_{k+1}}
   \right)\det A_i(t)=\sum_{i=1}^{N_k}(-1)^{N_k}q_k\prod_{a=1}^{N_{k-1}}(x^{(k-1)}_a-x^{(k)}_i)\det A_i(t)
\end{align}
where the matrix $A_i(t)$ is 
\begin{align*}
    \begin{bmatrix}
        \left(x^{(k)}_1\right)^{N_k-1}&\cdots&t^{N_k-1}&\cdots&\left(x^{(k)}_{N_k}\right)^{N_k-1}\\
        \vdots&\ddots&\vdots&\ddots& \vdots\\
        1&\cdots&1&\cdots&1
    \end{bmatrix}
\end{align*}
where the $\begin{bmatrix} t^{N_k-1}&\cdots& 1\end{bmatrix}^t$ lies in the $i$-th column.
 
The formula \eqref{eqn:quantumproductcomp} and hence the quantum correction of \eqref{term1} by Lagrangian interpolation is $(-1)^{N_k+N_{k-1}}q_k\widetilde{c^S_t(S_{k-1})}\omega_{k}$. 
The remaining part of \eqref{principle term} excluding \eqref{term1} is $\widetilde{c^S_t(S_{l}\slash S_{k+1})}(\omega\slash \omega_k)$. Then one gets the quantum correction of \eqref{principle term}  is $(-1)^{N_k+N_{k-1}}q_k\widetilde{c^S_t(S_{k-1})}*\widetilde{c^S_t(S_{l}\slash S_{k+1})}\omega$. 

The reduction of \eqref{term1} via \eqref{eqn:qcohrelation} will also produce a classical part. A quantum correction may appear in the quantum multiplication between this classical part and the remaining part of \eqref{principle term}. We claim this quantum correction vanishes. In fact, the classical part is a summation of the following polynomials,
\begin{equation*}
  f(t, x^{(k)}_1, \ldots, x^{(k)}_{i-1},  x^{(k)}_{i-1}, x^{(k)}_{i-1},\ldots, x^{(k)}_{N_k} )\cdot\sum_{l=1}^{N_k}\left(x_i^{(k)}\right)^{N_k-l}e_l(-\vec x^{(k+1)}), \quad i=1, 2,\ldots, N_k
\end{equation*}
where the degree of $x^{(k)}_{\bullet}$ in $f$ is less than $N_{k}$. Hence, the quantum correction only happens in the case $m_{k+1}=N_{k+2}-1$, which is
\begin{align}\label{term2}
  \sum_{l=1}^{N_{k+1}}\left(x_i^{(k)}\right)^{N_{k+1}-l}e_l(-\vec x^{(k+1)})* \begin{bmatrix}
     \left(x_1^{(k+1)}\right)^{N_{k+2}-1}&\cdots&\left(x_{N_{k}}^{(k+1)}\right)^{N_{k+2}-1}\\
      \left(x_1^{(k+1)}\right)^{N_{k+1}-2}&\cdots&\left(x_{N_k}^{(k+1)}\right)^{N_{k+1}-2}\\
        \vdots&\ddots&\vdots\\
   1&\cdots& 1\\
    \end{bmatrix}
\end{align}
Via the same procedure as in \eqref{term1} (just by replacing $t$ by $x_i^{(k)}$), one can prove that the quantum correction has a factor $(-1)^{\bullet}q_{k+1}\prod_{l=1}^{N_{k}}(x_i^{(k)}-x^{(k)}_l)=0$. (One may be concerned about the summation index in \eqref{term2} starting from $l=1$, which is different from \eqref{term1}. In fact, one can change $l=0$ to $l=1$ in \eqref{term1} without changing the quantum correction).
Hence all quantum corrections in the reduction procedure vanish except the first one. Therefore the quantum correction of $ c^S_t(S_k)* c^S_t(S_l\slash S_k)$ is $(-1)^{N_k+N_{k-1}}q_kc^S_t(S_{k-1})* c^S_t(S_{l}\slash S_{k+1})$. By \eqref{eqn:tautquot}, the classical part is $c^S_t(S_l)$, which completes the proof.     
\end{proof}
\begin{rem}
In the special case $l=k+1$, the formula \eqref{eqn:special1} originally appears in a physical literature \cite{Witten} for $Fl=Gr(k, N)$, and it is also discussed in \cite[Section 4.6]{MR4701783} for general partial flag varieties.
\end{rem}
\begin{cor}
	 In the quantum cohomology ring $QH^*_S(Fl)[t]$, for each $n+1\geq m>k\geq 1$, the following relation holds, 
\begin{equation}\label{eqn:special2}
		c_t^S(S_k/S_{k-1})* c_t^S(S_{m}/S_k)=c^S_t(S_m/S_{k-1})+(-1)^{N_k+N_{k-1}}q_kc^S_t(S_m/S_{k+1}).
\end{equation}
\end{cor}
\begin{proof}
In the ring $QH^*_S(Fl)[t]$, the Chern polynomials $c_t^S(S_i)$ are not zero divisors, since $t$ is a formal variable and the coefficient of the leading term of $c_t(S_i)$ as a polynomial of $t$ is 1.
We only need to prove that 
\begin{equation*}
  c_t^S(S_{k-1})* c_t^S(S_k/S_{k-1})* c_t^S(S_{m}/S_k)=c_t(S_{k-1})* \left[c_t^S(S_{m}/S_{k-1})+(-1)^{N_k+N_{k-1}}q_k c_t^S(S_m/S_{k+1})\right].
\end{equation*}
According to \eqref{eqn:special1}, the left-hand side is equal to 
\begin{align*}
  &c_t^S(S_{k-1})* c_t^S(S_k/S_{k-1})* c_t^S(S_{m}/S_k)\\
=&\left[c_t^S(S_k)+(-1)^{N_{k-1}+N_{k-2}}q_{k-1}c_t^S(S_{k-2})\right]* c_t^S(S_{m}/S_k)\\
=&c_t^S(S_k)* c_t^S(S_{m}/S_k)+(-1)^{N_{k-1}+N_{k-2}}q_{k-1}c_t^S(S_{k-2})* c_t^S(S_{m}/S_k)\\
=& c_t^S(S_{m})+(-1)^{N_{k}+N_{k-1}}q_k c_t^S(S_{m}\slash S_{k+1})* c_t^S(S_{k-1})+(-1)^{N_{k-1}+N_{k-2}}q_{k-1}c_t^S(S_{k-2})* c_t^S(S_{m}/S_k)\\
=&c_t^S(S_{k-1})* c_t^S(S_{m}/S_{k-1})+(-1)^{N_k+N_{k-1}}q_k c_t^S(S_{k-1})* c_t^S(S_m/S_{k+1}),
\end{align*}
which is equal to the right-hand side. 
\end{proof}

\subsection{proof of the Theorem \ref{thm:quantumcohexrel}(1)}\label{sec:proofofmaintheorem}
As discussed in the beginning of Section \ref{sec:proofclusteralgconj}, in order to prove that the $\psi$  in Theorem \ref{thm:maincluster2} is a ring homomorphism, we only have to prove that it preserves the cluster exchange relations 
in  Theorem \ref{thm:quantumcohexrel}. 

The quantum cohomological cluster exchange relation in Theorem \ref{thm:quantumcohexrel} is the most general case when the node $v$ has four adjacent nodes. For special values of $m,k,p,l$ in Figure \ref{fig:localbehavior}, we can get other cases when the node $v$ has less adjacent nodes. 
In this section, we will prove the Theorem \ref{thm:quantumcohexrel} by discussing the number of node $v$'s adjacent nodes.
\subsubsection{Case one: the node $v$ has one adjacent node}
Letting $m=k+1$, $l=p$, and $k=p+1$, we get the local picture where the $v$ has only one adjacent node. 
Then the node $v$ is either a source or a sink. 
We assume that it is a source, and the case when the node $v$ is a sink can be obtained by performing a quiver mutation at the node $v$, which satisfy the same quantum cohomological exchange relation with the case when the node $v$ is a source.
\iffalse
The two diagrams in Figure \ref{fig:localproof} are now in a special situation as in Figure \ref{fig:node1valmuN}.
\begin{figure}[H]
    \centering
    \includegraphics[width=3.6in]{node1valmuN.png}
    \caption{}
    \label{fig:node1valmuN}
\end{figure}
\noindent
\fi
Then in this situation, the Theorem \ref{thm:quantumcohexrel} is reduced to the following Lemma.
\begin{lem}\label{lem:node1eqn}
In $QH^*_S(Fl)[t]$, we have the following relation,
\begin{equation*}
     c_t^S(S_k/S_{k-1})* c_t^S(S_{k+1}/S_k)=c_t^S(S_{k+1}/S_{k-1})+(-1)^{N_k+N_{k-1}}q_k.
\end{equation*}
\end{lem}
\begin{proof}
This is just a special case of \eqref{eqn:special2} when $m=k+1$.
\end{proof}
\subsubsection{Case two: the node $v$ has two adjacent nodes}
There are two situations when a node has two valences.
The first situation is when the node is a source. 
In this situation, $m=k+1$, and $p=l$.
\begin{lem}\label{lem:quantumclsuter1}
    In $QH^*_S(Fl)[t]$, the following relation holds
    \begin{equation*}
    c_t^S(S_k/S_p)* c_t^S(S_{k+1}/S_{p+1}) =c_t^S(S_{k+1}/S_p)* c_t^S(S_k/S_{p+1})+(-1)^{N_k+N_p}\prod_{l=p+1}^{k}q_{l}.
    \end{equation*}
\end{lem}
\begin{proof}
Prove by induction on the difference $k-p$. \\
When $k-p=1$, the situation is reduced to Lemma \ref{lem:node1eqn}, which has been proved. \\
Assume the equation holds for some $a=k-p$. Then when $k-p=a+1$, we can apply Equation \eqref{eqn:special2} to $c_t^S(S_{k+1}/S_p)$ and get
\begin{align*}
	&c_t^S(S_{k+1}/S_{p})* c_t^S(S_{k+2}/S_{p+1}) \nonumber \\
 &=c_t^S(S_{k+2}/S_{p+1})* \left(c_t^S(S_{k+1}/S_{p+1})* c_t^S(S_{p+1}/S_{p}) -(-1)^{N_{p+1}+N_p}q_{p+1}c_t^S(S_{k+1}/S_{p+2})\right).
\end{align*}
Again by Equation \eqref{eqn:special2}, the first term is
\begin{align*}
	c_t^S(S_{k+1}/S_{p+1})* \left(c^S_t(S_{k+2}/S_{p})+(-1)^{N_{p+1}+N_p}q_{p+1}c_t^S(S_{k+2}/S_{p+2}) \right).
\end{align*}
Then
\begin{align*}
	&c_t^S(S_{k+1}/S_{p})* c_t^S(S_{k+2}/S_{p+1})-c_t^S(S_{k+1}/S_{p+1})* c_t^S(S_{k+2}/S_{p})\nonumber\\
	&=(-1)^{N_{p+1}+N_p}q_{p+1}\left(c^S_t(S_{k+1}/S_{p+1})* c_t^S(S_{k+2}/S_{p+2})-c_t^S(S_{k+2}/S_{p+1})* c_t^S(S_{k+1}/S_{p+2})\right),
\end{align*}
whose right hand side by induction on $a=k+1-(p+1)$ is equal to
\begin{equation*}
	(-1)^{N_{p+1}+N_p}q_{p+1}(-1)^{N_{k+1}+N_{p+1}}\prod_{l=p+2}^{k+1}q_l=(-1)^{N_{k+1}+N_p}\prod_{l=p+1}^{k+1}q_l.
\end{equation*}
Hence, we have proved the Lemma.
\end{proof}

The second situation for a node admitting two valences is when there is a path passing through the node and no other arrows are adjacent to the node. 
There are two case. The first is when $m=k+1,\,p=k-1$ and the second case is when $p=k-1,\,l=k-1$. 
\iffalse
\begin{figure}[H]
    \centering
    \includegraphics[width=5in]{node2valmuN.png}
   \caption{}
    \label{fig:node2valmuN}
\end{figure}

\begin{figure}[H]
    \centering
    \includegraphics[width=5in]{node2valmuN2.png}
    \caption{}
    \label{fig:node2valmuN2}
\end{figure}
\fi
Then the theorem \ref{thm:quantumcohexrel} is reduced to the following Lemma.
\begin{lem}\label{lem:node2case2}
    In the quantum cohomology ring $QH^*(Fl)[t]$, we have
    \begin{subequations}
    \begin{align}
         & c_t^S(S_k/S_l)* c_t^S(S_{k+1}/S_k)=c_t^S(S_{k+1}/S_l)+(-1)^{N_{k-1}+N_k}q_k c_t^S(S_{k-1}/S_l),\,\label{subeqn:node2valcase1}\\
         & c_t^S(S_k/S_{k-1})* c_t^S(S_{m}/S_{k})= c_t^S(S_m/S_{k-1}) +(-1)^{N_k+N_{k-1}}q_kc_t^S(S_m/S_{k+1})\label{subeqn:node2valcase2}
    \end{align}
    \end{subequations}
The index $l$ in Equation \eqref{subeqn:node2valcase1} is allowed to be zero. 
\end{lem}
\begin{proof}
	The Equation \eqref{subeqn:node2valcase2} is exactly the Equation \eqref{eqn:special2}. We only have to prove the Equation \eqref{subeqn:node2valcase1}.  We prove it by induction on the difference $k-l$.
	When $k-l=1$, it is exactly the equation in Lemma \ref{lem:node1eqn}. \\
	For fixed $l$, assume that the equation holds for $k$. 
   Then for $k+1$, it suffices to prove
\begin{align}\label{eqn:2nodesproof}
   &c_t^S(S_l)* c_t^S(S_{k+1}/S_l)* c_t^S(S_{k+2}/S_{k+1})\nonumber\\
   &= c_t^S(S_l)* \left[c_t^S(S_{k+2}/S_l)+(-1)^{N_{k}+N_{k+1}}q_{k+1} c_t^S(S_{k}/S_l)\right]
\end{align}
By utilizing \eqref{eqn:special1}, we expand the left hand side to
\begin{align}\label{eqn:2nodes1proof}
  c_t^S(S_{k+2})+(-1)^{N_{k+1}+N_{k}}q_{k+1} c_t^S(S_{k})+(-1)^{N_{l}+N_{l-1}}q_l c_t^S(S_{l-1})*\left[c_t^S(S_{k+1}/S_{l+1})* c_t^S(S_{k+2}/S_{k+1})\right]. 
\end{align}
Again, by utilizing Equation \eqref{eqn:special1}, we expand the quantum product $c_t^S(S_{k+2}/S_l)*c_t^S(S_l)$ and $c_t^S(S_{k}/S_l)*c_t^S(S_l)$ and rewrite the right hand side as 
\begin{align}\label{eqn:2nodes2proof}
 &c_t^S(S_{k+2})+(-1)^{N_{k+1}+N_{k}}q_{k+1}c_t^S(S_{k})+(-1)^{N_{l}+N_{l-1}}q_l c_t^S(S_{l-1})*\nonumber\\
 &\left[c_t^S(S_{k+2}/S_{l+1})+(-1)^{N_{k+1}+N_{k}}q_{k+1} c_t^S(S_{k}/S_{l+1})\right].
\end{align}
The terms inside brackets in Equations \eqref{eqn:2nodes1proof} and \eqref{eqn:2nodes2proof} are equal by the induction assumption. Hence the two sides of \eqref{eqn:2nodesproof} are equal and the Lemma has been proved.
\end{proof}

\subsubsection{Proof of the Theorem \ref{thm:quantumcohexrel}: the node $v$ has three adjacent nodes}
When a node has exactly three adjacent nodes, two of them are in a 3-cycle, and the remaining one is not in any $3$-cycle.
The local behavior of the quiver diagram can be classified into the following three cases.
\begin{itemize}
    \item The first case is  when $m=k+1$, and the node $v_2$ disappears. 
    \item The second case is when $p=k-1$, and then the node $v_4$ disappears.
    \item The third case is when $p=l$ and the node $v_3$ disappears. 
    \end{itemize}
    Then the Theorem \ref{thm:quantumcohexrel} is reduced to the following three equations which correspond to the above three cases. 
\begin{lem}\label{lem:3nodeseqn}
In $QH^*_S(Fl)[t]$, the following relations hold.
For $k>p>l$, we have 
\begin{align}\label{eqn:node3val2}
    &c_t^S(S_k/S_l)* c_t^S(S_{k+1}/S_{p+1})
	\nonumber\\
   & =c_t^S(S_{k+1}/S_l)* c_{t}^S(S_{k}/S_{p+1})+(-1)^{N_p+N_k}(\prod_{a=p+1}^k q_a) c_t^S(S_p/S_l)\,.
\end{align}
For $m>k>l$, we have
\begin{align}\label{eqn:node3val4}
    &c_t^S(S_k/S_l)* c_t^S(S_m/S_{k})\nonumber\\
	&=c_t^S(S_m/S_l)+(-1)^{N_k+N_{k-1}}q_kc_t^S(S_m/S_{k+1})* c_t^S(S_{k-1}/S_l)\,.
\end{align}
For $m>k>l$,
\begin{align}\label{eqn:node3val1}
    &c_t^S(S_k/S_l)* c_t^S(S_m/S_{l+1})\nonumber\\
	&= c_t^S(S_m/S_{l})* c_t^S(S_{k}/S_{l+1})+ (-1)^{N_k+N_{l}} (\prod_{a=l+1}^{k}q_a) c_t^S(S_m/S_{k+1})\,.
\end{align}
In all the above three equations, the index $l$ can be zero.
\end{lem}
\begin{proof}
	\textbf{Proof of Equation \eqref{eqn:node3val2}:} We prove it by induction on the difference $k$. \\
	When $k=p+1$, the equation is reduced to the Equation \eqref{subeqn:node2valcase1} which has been proved to be true. \\
	Assume that the equation holds for $k$. 
    For $k+1$, utilizing \eqref{subeqn:node2valcase1} we get
    \begin{align}\label{eqn:node3val1proof}
		&c_t^S(S_{k+1}/S_l)* c_t^S(S_{k+2}/S_{p+1})\nonumber\\
		&=c_t^S(S_{k+1}/S_{l})* \left(c_t^S(S_{k+2}/S_{k+1})* c_t^S(S_{k+1}/S_{p+1})-(-1)^{N_{k+1}+N_k}q_{k+1}c_t^S(S_k/S_{p+1})\right)\,.
	\end{align}
	Again by applaying \eqref{subeqn:node2valcase1}, we rewrite the first term as
	\begin{equation}\label{eqn:node3val1proofterm11}
		c_t^S(S_{k+1}/S_{p+1})* \left(c_t^S(S_{k+2}/S_{l})+(-1)^{N_{k+1}+N_k}q_{k+1}c_t^S(S_k/S_l)\right).
	\end{equation}
	Combine \eqref{eqn:node3val1proof} and \eqref{eqn:node3val1proofterm11}, and we get
	\begin{align}\label{eqn:node3val1proof2}
		&c_t^S(S_{k+1}/S_l)* c_t^S(S_{k+2}/S_{p+1})-
  c_t^S(S_{k+1}/S_{p+1})* c_t^S(S_{k+2}/S_{l})
  \nonumber\\
		&=(-1)^{N_{k+1}+N_k}q_{k+1}\left(c_t^S(S_{k+1}/S_{p+1})* c_t^S(S_k/S_l)-c_t^S(S_{k+1}/S_l)*c_t^S(S_k/S_{p+1})\right).
	\end{align}
	By the assumption that the equation \eqref{eqn:node3val2} holds for $k$,  the two terms in the parentheses equal $(-1)^{N_p+N_k}\prod_{a=p+1}^k q_ac_t^S(S_p/S_l)$. Therefore,
	\begin{align*}
		&c_t^S(S_{k+1}/S_l)* c_t^S(S_{k+2}/S_{p+1})
\nonumber\\
		&=c_t^S(S_{k+1}/S_{p+1})* c_t^S(S_{k+2}/S_{l})+(-1)^{N_{k+1}+N_p}\prod_{a=p+1}^{k+1} q_ac_t^S(S_p/S_l)\,.
	\end{align*}
	which proves \eqref{eqn:node3val2} for $k+1$.\\
	
	\textbf{Proof of Equation \eqref{eqn:node3val4} and Equation \eqref{eqn:node3val1}:} 
 
 The proof is similar with that of \eqref{eqn:node3val2}.
 One can prove by induction on $m$. When $m=k+1$, the Equation \eqref{eqn:node3val4} is reduced to \eqref{subeqn:node2valcase1}, and the Equation \eqref{eqn:node3val1} is reduced to  a special case of Equation \eqref{eqn:node3val2}. 
Assume that they hold for index smaller than or equal to $m$. For $m+1$, one needs to repeatedly use the Equation \eqref{subeqn:node2valcase1} and prove the two equations. We leave the detail to readers. 
 \end{proof}
\subsubsection{Proof of the Theorem \ref{thm:quantumcohexrel}: the node $v$ has four adjacent nodes}
We finally arrive at the most general case when the node $v$ has four adjacent nodes. 
Actually, there are two situations for a node admitting four adjacent arrows, and one can find that they are related by a quiver mutation in Figure \ref{fig:localbehavior}. 
\begin{lem}\label{lem:node4val}
In $QH^*_S(Fl)[t]$, we have the following relation. For $m>k>p>l\geq 0$, we have
\begin{align}\label{eqn:proofnode4val}
     &c_t^S(S_k/S_l)* c_t^S(S_m/S_{p+1})= c_t^S(S_m/S_l)* c_t^S(S_k/S_{p+1})\nonumber\\
     &+(-1)^{N_p+N_{k}}\prod_{a=p+1}^k q_ac_t^S(S_m/S_{k+1})* c_t^S(S_{p}/S_l).
    \end{align}
\end{lem}
\begin{proof} We prove by induction on $m$.
When $m=k+1$, the equation is reduced to the Equation \eqref{eqn:node3val2}. 
Assume that the Equation \ref{eqn:proofnode4val} holds for the index less than or equal to $m$. For $m+1$, we expand the $c_t^S(S_{m+1}/S_{p+1})$ by Equation \eqref{subeqn:node2valcase1}. Then we have 
\begin{align}\label{eqn:proofnode41}
    &c_t^S(S_k/S_l)* c_t^S(S_{m+1}/S_{p+1})
    =c_t^S(S_k/S_l)*\nonumber\\
   & \left(
    c_t^S(S_{m+1}/S_m)* c_t^S(S_{m}/S_{p+1})-(-1)^{N_m+N_{m-1}}q_mc_t^S(S_{m-1}/S_{p+1})
    \right)
\end{align}
By applying the induction assumption for $m$, the first term is equal to 
\begin{align}\label{eqn:proofnode42}
    c_t^S(S_{m+1}/S_m)*
    \left(
    c_t^S(S_m/S_l)* c_t^S(S_k/S_{p+1})+(-1)^{N_p+N_{k}}\prod_{a=p+1}^k q_ac_t^S(S_m/S_{k+1})* c_t^S(S_{p}/S_l)
    \right)
\end{align}
The first term of 
\eqref{eqn:proofnode42} again by Equation \eqref{subeqn:node2valcase1} can be rewritten as 
\begin{align}\label{eqn:proofnode43}
    c_t^S(S_k/S_{p+1})*
    c_t^S(S_{m+1}/S_{l})+(-1)^{N_m+N_{m-1}}q_mc_t^S(S_{m-1}/S_{l})*c_t^S(S_k/S_{p+1}).
\end{align}
The second term of \eqref{eqn:proofnode43} and the second term of the right hand side of the Equation \eqref{eqn:proofnode41}, by the induction assumption on $m-1$, is equal to
\begin{equation}\label{eqn:proofnode44}
    -(-1)^{N_p+N_{k}+N_m+N_{m+1}}(q_m\prod_{a=p+1}^kq_a)c_t^S(S_{m-1}/S_{k+1})*c_t^S(S_p/S_l).
\end{equation}
The second term of the formula \eqref{eqn:proofnode42}, by using Equation \eqref{subeqn:node2valcase1}  is equal to 
\begin{equation}\label{eqn:proofnode45}
   (-1)^{N_p+N_{k}}
   (\prod_{a=p+1}^k q_a)c_t^S(S_{p}/S_l)*
   \left(c_t^S(S_{m+1}/S_{k+1})+(-1)^{N_m+N_{m+1}}q_mc_t^S(S_{m-1}/S_{k+1}) 
   \right)
\end{equation}
whose second term cancels with the formula in \eqref{eqn:proofnode44}.
Combine \eqref{eqn:proofnode41}-\eqref{eqn:proofnode45}, we have proved the Equation \eqref{eqn:proofnode4val}.
\end{proof}

\begin{cor}\label{cor:Aclustertrue}
    The images $\{c_t^S(S_k/S_l)\}$ of cluster variables for $A$-type quivers coincide with
    the conjectural images of cluster variables in the quantum cohomology rings for general quivers in  Conjecture \ref{conj:clusteralgconj}. 
    That is to say, the Chern polynomials of quotient bundles $c_t^S(S_k/S_l)$ are pullbacks of Chern polynomials of the dual of tautological bundles over varieties of $QPs$ in $\Omega^A_n$. 
\end{cor}
\begin{proof}
    Assume that $x_v$ lies in a seed $( {\mathbf Q},  {\mathbf{x}})$ 
        which can be obtained by a sequence of quiver mutations $\mu_{v_i}\cdots\mu_{v_1}$ from the initial $A_n$-type quiver. 
        Let $\mc Z$ be the quiver variety 
        associated with
        the decorated QP $(\mathbf Q,\mathbf r)$, and we assume that the integer assigned to the node $v$ is $N_k-N_l$. 
        Let $\tilde S_v$ be the tautological bundle over $\mc Z$ at the node $v$. 
        We consider the pullback of the Chern polynomial $c_t^S(\tilde S_v)$ via the map
       \begin{equation*}
            \Phi: H^*(\mc Z)\rightarrow H^*(Fl(N_1, N_2, \ldots, N_{n+1}))
        \end{equation*}
        where $\Phi$ is a ring isomorphism  according to Corollary  \ref{cor:qcohisomorphic}. 
        
        We claim that
        \begin{equation}\label{eqn:5.71}
            \Phi(c_t^S(\tilde S_v))=c_t^S(S_k/S_l)
        \end{equation}
       This can be checked by localization. 
       We only have to prove that for any $P=(I_u)_{u=1}^n\in Fl^S$,
    \begin{equation}\label{eqn:5.72}
        c_t^S(S_k/S_l) \big{|}_{P} =c_t^S(\tilde S_v)\big{|}_{P'}\,,
    \end{equation}
    where $P'=(I_u')_{u\in Q_0\backslash Q_f}$ is the image of $P$ in $(\mc Z)^S$ by performing the rule in Lemma \ref{lem:Sfixedpointsmu} along the quiver mutations $\mu_{v_i}\cdots\mu_{v_1}$. 
    Denote the set of equivariant weights of of tautological bundles of flag variety $S_i|_P$ by $\mathcal I_i$, 
    for $i=1, 2,\ldots, n+1$, and the set of weights  $\tilde S_v|_{P'}$ by $\mc I_v'$. 
    Due to the rule in Lemma \ref{lem:Sfixedpointsmu}, $\mc I_v'$ can be obtained from $\{\mc I_i\}_{i=1,\ldots, n+1}$ via a composition of disjoint unions and complements. Furthermore, the composition procedure is independent of the choice of integers $N_1,\ldots, N_{n+1}$. 
    For arbitrary $N_1,\ldots, N_{n+1}$, 
    we have $|\mc I_i|=N_i$, and $|\mc I'_v|=N_k-N_l$, this implies $\mc I_v'=\mc I_k\backslash\mc I_l$. So the \eqref{eqn:5.72} is proved.

    Suppose the local behavior of $\mc Z$ near node $v$ is as Figure \ref{fig:localbehavior} (a). 
    Now we consider the pullback of the cluster relation \eqref{eqn:quantumcohexrel0} via the map $\Phi^{-1}$. We claim that it becomes the following cluster relation in $QH^*(\mc Z)$
    \begin{equation}\label{eqn:quantumcohclusterinZ}
        c_t^S(\tilde S_v)*c_t^S(\tilde Q_v)=
        c_t^S(\tilde S_{v_1})*c_t^S(\tilde S_{v_4}^*)+ (-1)^{N_k-N_l}\tilde q_vc_t^S(\tilde S_{v_2}^*)*c_t^S(\tilde S_{v_3})\,.
    \end{equation}
    where $Q_v$ is defined in Section \ref{sec:seibergduality}.
    Let $\tilde \beta=(\tilde \beta_u)\in H_2(\mc Z, \mathbb Z)$ and $\beta=(\beta_i)\in H_2(Fl, \mathbb Z)$ be the dual basis of first Chern classes of tautological bundles, as defined in \eqref{eqn:basisofeff}. 
    The $\tilde q_v$ in \eqref{eqn:quantumcohclusterinZ} is the K\"ahler variable tracking $\tilde \beta_v$. 
    By \eqref{eqn:5.71}, the equation \eqref{eqn:quantumcohclusterinZ} holds if we can prove that
    \begin{equation*}
        \phi_*(\tilde \beta_v)=\beta_{p+1}+\cdots+\beta_k,   
    \end{equation*}
    It is equivalent to say that, for any node $w$ in ${\mathbf Q}$,
    \begin{equation}\label{eqn:5.73}
        \int_{ \tilde \beta_v}c_1(\tilde S_w)=\int_{\beta_{p+1}+\cdots+\beta_k}\Phi^{-1}(c_1(\tilde S_w))
    \end{equation}
    For $w=v, v_1, v_2, v_3, v_4$, the equation \eqref{eqn:5.73} is easily checked by \eqref{eqn:5.71} and \eqref{eqn:basisofeff}. For other nodes, we use the following observation implied by Lemma \ref{lem:decorationsofQinP}.
   For each nodes $u, w\in \mathbf Q$, the subscripts $a,b,c,d$ of the associated integer $r_w=N_a-N_b$, $r_u=N_c-N_d$ satisfying the following condition 
\begin{equation}\label{eqn:5.77}
        \begin{cases}
            d\leq a, b\leq c, &\text{  if }w\prec u\\
          a, b\notin (d, c) &\text { otherwise }\,.
        \end{cases}
    \end{equation}   
 By applying \eqref{eqn:5.77} for nodes  $u=v_1, v_2, v_3, v_4$, we can easily check the equation \eqref{eqn:5.73} by \eqref{eqn:5.71} and \eqref{eqn:basisofeff}. 
\end{proof}
\printbibliography 
\noindent
Weiqiang He, Department of Mathematics, Sun Yat-sen University, Xingang West Road No.135, Haizhu District, Guangzhou, China. \\
Email: hewq@mail2.sysu.edu.cn. \\
\noindent
 Yingchun Zhang, Room 221, School of Mathematical Sciences, Shanghai Jiao Tong University, Shanghai, China. \\
 Email: yingchun.zhang@sjtu.edu.cn
\\

\end{document}